\documentclass{ws-m3as}

\usepackage{import}
\usepackage{comment}

\usepackage{stmaryrd}
\usepackage[shortlabels]{enumitem}

\usepackage[utf8]{inputenc}
\usepackage[T1]{fontenc}
\usepackage[english]{babel}

\usepackage{bm}
\usepackage{amssymb}
\usepackage{amsmath}
\usepackage{mathrsfs}
\usepackage{derivative}
\usepackage{accents}
\usepackage{mathtools}
\mathtoolsset{showonlyrefs}
\usepackage[normalem]{ulem}
\usepackage{hyperref}

\makeatletter
\@ifundefined{assumption}{\newtheorem{assumption}[theorem]{Assumption}}{}
\@ifundefined{alg}{}{}
\@ifundefined{example}{}{}
\@ifundefined{problem}{\newtheorem{problem}[theorem]{Problem}}{}
\@ifundefined{remark}{\newtheorem{remark}[theorem]{Remark}}{}
\makeatother

\DeclareMathOperator*{\argmin}{arg\,min}
\DeclareMathOperator*{\esssinf}{ess\,inf}
\DeclareMathOperator*{\esssup}{ess\,sup}

\newcommand{\norm}[1]{\left\lVert #1 \right\rVert}
\newcommand{\snorm}[1]{\left\lvert #1 \right\rvert}

\newcommand{\dual}[2]{\left\langle #1,#2 \right\rangle}
\newcommand{\dotp}[2]{\left( #1,#2 \right)}

\newcommand{\IR}{\mathbb{R}}
\newcommand{\IC}{\mathbb{C}}

\newcommand{\IN}{\mathbb{N}}

\newcommand{\IT}{\mathbb{T}}

\newcommand{\IV}{\mathbb{V}}

\newcommand{\fh}[1]{{\color{cyan}{#1}}}
\newcommand{\rev}[1]{#1} 

\usepackage{tikz}

\newcommand\encircle[1]{%
  \tikz[baseline=(X.base)] 
    \node (X) [draw, shape=circle, inner sep=0] {\strut #1};
}

\usepackage{graphicx}
\usepackage{multirow}
\usepackage{caption}
\usepackage{subcaption}

\usepackage{pgf}
\usepackage{pgfplots}

\usepackage[acronym]{glossaries}

\newacronym{pod}{POD}{Proper Orthogonal Decomposition}
\newacronym{svd}{SVD}{Singular Value Decomposition}
\newacronym[longplural={Reduced Order Models}]{rom}{ROM}{Reduced Order Model}
\newacronym[longplural={Full Order Models}]{fom}{FOM}{Full Order Model}
\newacronym[longplural={Spectral Elements}]{se}{SE}{Spectral Element}
\newacronym{rmsd}{RMSD}{Root-Mean-Square Deviation}

\begin{document}

\markboth{Fernando Henr\'iquez \& Jan S. Hesthaven}{Fast Approximation of Parabolic Problems}

\title{Fast Numerical Approximation of Parabolic Problems Using Model Order Reduction and the Laplace Transform\thanks{Version of \today}}

\author{Fernando Henr\'iquez}

\address{Institute for Analysis and Scientific Computing, TU Wien \\ Wiedner Hauptstra{\ss}e 8-10, A-1040 Wien, Austria\\
fernando.henriquez@asc.tuwien.ac.at}

\author{Jan S. Hesthaven}

\address{Karlsruhe Institute of Technology \\ Kaiserstraße 12, 76131 Karlsruhe, Germany\\
jan.hesthaven@kit.edu}

\maketitle

\begin{abstract}
We introduce a method for the \rev{fast} numerical approximation of linear, second-order
parabolic partial differential equations (PDEs for short) \rev{with time-independent coefficients} based on model order reduction techniques and the Laplace transform. We start by applying \rev{this} transform to the evolution problem, thus yielding a time-independent boundary value problem solely depending on the complex Laplace \rev{variable}. 
In an offline stage, we judiciously sample the Laplace \rev{variable} and numerically solve the corresponding collection of high-fidelity or full-order problems.
Next, we apply a proper orthogonal decomposition (POD) to this collection of solutions in order to obtain a reduced basis in the Laplace domain. We project the linear parabolic problem onto this basis and then, using any suitable time-stepping method, we solve the evolution problem.
A key insight to justify the implementation and analysis of the proposed method
\rev{consists of using Hardy spaces of analytic functions and establishing}, through the Paley-Wiener theorem, an isometry between the solution of the time-dependent problem and its Laplace transform.
As a result, one may conclude that computing a POD with samples taken in the Laplace domain produces an exponentially accurate reduced basis for the time-dependent problem.
Numerical experiments illustrate the performance of the method in terms of accuracy and, in particular, speed-up when compared to the solution obtained by solving the full-order model.
\end{abstract}



\section{Introduction}
The fast and efficient solution of \rev{parametric PDEs (pPDEs)}
is an essential task in numerous applications within the fields of science and engineering.
In this \rev{context}, parameters can be used to describe material properties, source terms, domain perturbations,
and initial conditions, among others.
The increasing complexity of existing mathematical models, combined with
the demand for real-time and multi-query computational simulations,
necessitates the development and implementation of fast and efficient algorithms
capable of addressing these challenges.
The numerical approximation of these problems is conducted 
by computing a \emph{high-fidelity} approximation of the \emph{full-order} 
model by means of techniques such as Finite Elements, Finite Volumes,
Finite Differences or spectral methods,
coupled with a time integration \rev{scheme} in the case of evolution problems. However, the repeated computation of these high-fidelity approximations
for new sets of parametric inputs using said methods rapidly becomes unfeasible.

\subsection{Model Order Reduction for Parametric PDEs}
Model order reduction (MOR) encompasses a family of techniques aiming
at reducing the complexity of a certain \rev{class} of parametric problems, such 
as pPDEs. The success of MOR is based on the identification of an intrinsic low-dimensional 
dynamic of the full-order model, even
if the problem at first glance seems high-dimensional in nature.
In other words, the high-fidelity \rev{approximation}
is replaced by one of \rev{considerably} lower computational complexity, referred to
as the low-fidelity approximation, that can, however, 
be rapidly and accurately \rev{evaluated} for different parametric inputs.

Among the many techniques covered by MOR, the RB method stands as one of the most 
commonly used.
The RB method is divided into two distinct phases.
Firstly, in the \emph{offline} phase we compute a collection
of so-called \emph{snapshots} or {high-fidelity} solutions
of the parametrized problem for a number of parametric inputs.
Using these data, one computes a \emph{reduced} basis that
captures the \rev{behavior} of the underlying low-dimensional
structure driving the parametric problem.
Two main approaches have been reported in the literature to tackle
the construction of a reduced basis: Proper Orthogonal Decomposition (POD)\cite{liang2002proper}
and greedy strategies\cite{hesthaven2014efficient,bui-thanh_model_2008,buffa_priori_2012,devore_greedy_2013}.
POD defines \emph{a priori} a set of samples in the parameter space, to then compute the corresponding \emph{snapshots} or
high-fidelity solutions.
Using the SVD of the snapshot matrix, a reduced basis of any desired dimension can be easily and straightforwardly extracted.
On the other hand, greedy strategies aim at carefully \rev{selecting} one snapshot after the other in a serial fashion, at each  
step selecting the high-fidelity solution that improves the most the approximation of the parametric problem's \emph{solution manifold}, \rev{whereas \emph{weak} greedy strategies,
though computationally more efficient, require sharp a posteriori error estimators for \rev{their} implementation.} 
Provided that the high-fidelity problem has an intrinsic low-dimensional structure, and that the constructed reduced
basis is capable of properly \rev{representing this feature}, one only  
needs to solve a problem of dimension much smaller than that of the high-fidelity model
in order to adhere to a target accuracy. This corresponds to the \emph{online} phase of the RB method.

However, we remark that any advantage \rev{that} one  
may obtain by reducing the dimension of the problem might be diluted by the need of assembling
the high-fidelity model before projecting onto the reduced basis space for each parametric input,
\rev{for example, in the presence of non-linear terms.}
To \rev{mitigate} this issue, one needs to resort to techniques such as the
\rev{Empirical Interpolation Method\cite{barrault2004empirical}
and its discrete counterpart\cite{chaturantabut2010nonlinear}.}

There exists a vast literature on the RB method for stationary problems with certified error control.
We refer to Refs. \refcite{hesthaven2016certified,quarteroni2015reduced,prud2002reliable,rozza2014fundamentals} for 
further details and comprehensive reviews on the topic.

\subsection{MOR for Time-Dependent Problems}
For time-dependent problems, a variety of approaches have been proposed.
In Ref.~\refcite{haasdonk2008reduced}, the RB method is applied to linear evolution equations,
which are discretized in space using the Finite Volume Method.
Non-linear MOR based on local reduced-order bases is \rev{discussed} in Ref.~\refcite{lieu2006reduced}.
We also point out developments of recently proposed structure-preserving RB approaches
for Hamiltonian problems\cite{afkham2017structure,hesthaven2022rank,hesthaven2021structure,hesthaven2023adaptive},
and non-intrusive frameworks based on radial basis function interpolation\cite{audouze2013nonintrusive,xiao2017parameterized}.
Recently, \emph{data-driven} approaches have \rev{gained} traction as tractable approaches
to MOR for time-dependent problems. We mention as examples the
dynamic mode decomposition (DMD)\cite{schmid2010dynamic,kutz2016dynamic,duan2023non}, 
and operator learning\cite{peherstorfer2016data,duan2023machine}.
For a comprehensive survey of these \rev{and} other MOR techniques applied to time-dependent problems,
we refer to Ref.~\refcite{Hesthaven2022}.

In the \rev{aforementioned} approaches, a key part of the algorithm 
consists \rev{of} performing a discretization in time for the computation of the snapshots. 
An alternative approach for constructing a reduced model consists \rev{of} applying
the Laplace transform to the original time-dependent problem, thus yielding a time-independent 
problem \rev{that depends solely} on the Laplace \rev{variable}. For this setting, a reduced basis for the parametric
problem can be constructed in the Laplace domain using available techniques for stationary
problems. This approach has been recently studied in Ref.~\refcite{guglielmi2023model} for parametric, linear
second-order parabolic problems, in which contour deformation techniques are used to compute the inverse
Laplace transform\cite{guglielmi2021pseudospectral}. Furthermore, the idea of using the
Laplace transform to construct a reduced basis has also been pursued in
Ref.~\refcite{huynh_laplace_2011} and Ref.~\refcite{Bigoni2020a}.
In the latter, emphasis is put 
on the wave equation, a problem for which contour deformation techniques fail, and the inverse Laplace
transform is computed using Weeks' method\cite{kuhlman_review_2013}.
Nonetheless, this technique still faces significant issues in terms of computational stability and precision.
Indeed, the computation of the inverse Laplace transform becomes a computational challenge,
especially when dealing with extended time intervals. The present work is driven by this issue
and aims to provide a first step toward a stable and accurate RB method for
parametric, time-dependent problems
using the Laplace transform.

\subsection{Main Contribution}
The primary objective of this work is to introduce a novel fast numerical method,
hereafter referred to as the Laplace Transform \rev{Model Order Reduction} (LT-MOR) method,
designed for the efficient numerical approximation of a specific class of time-evolution problems: \rev{linear, second-order parabolic problems with time-independent coefficients.} 
The \rev{method introduced here} is based on two existing mathematical tools:
The RB method and the Laplace transform.

This \rev{LT-MOR} method, like any reduced-basis scheme, is divided into two main steps

\begin{itemize}
	\item[(1)]
	{\bf Offline.}
	Firstly, we apply the Laplace transform to the 
	time-dependent parabolic PDE, and obtain an elliptic PDE
	depending on the (complex) Laplace \rev{variable}. We solve 
	this problem on a judiciously selected \emph{a priori} set 
	of instances of the complex Laplace \rev{variable}, and using POD we construct a reduced basis tailored to the problem.
	We remark at this point that no discretization in time has
	been performed. 
	\item[(2)]
	{\bf Online.}
	In a second step, we solve the time-evolution problem by projecting the high-fidelity 
	model, i.e.,~the parabolic evolution problem, onto the previously computed reduced space.
	Then, using any suitable time-stepping scheme, we compute the solution of the time-dependent problem.
\end{itemize}
A few remarks are in \rev{order}. 
In this work, we do not \rev{consider a parametric evolution problem}, just the plain parabolic evolution problem with fixed data.
However, once the Laplace transform is applied to the evolution problem we obtain a
\rev{parametric family} of elliptic PDEs, where the parameter corresponds to the
Laplace variable. In our approach, we construct a reduced basis by sampling 
the Laplace \rev{variable}. Then, in the online phase, we project the high-fidelity model
onto this low-dimensional reduced space, and solve the time-evolution problem \rev{using} any suitable
time-stepping scheme.
Observe that, by doing so, only a small number of coefficients need to be updated at each time step,
as opposed to standard time-stepping methods that update all the degrees of freedom
involved in the high-fidelity model. This renders the online step considerably faster, at the 
(hopefully low) price of computing a few high-fidelity solutions in the Laplace domain.

An important question arising from the previous description of the LT-MOR method
is the following: \emph{Why is the reduced basis constructed in the Laplace domain 
able to capture the intrinsic low-dimensional behavior of the parabolic problem?} 
To effectively answer this inquiry, we need to resort to Hardy spaces of analytic functions,
and in particular make use of the so-called Paley-Wiener representation theorem.   
In this work, we not only present the LT-MOR method as an \emph{off-the-shelf} algorithm;
we also provide a rigorous analysis of the proposed method with a particular focus on the following aspects:
\begin{itemize}
	\item[(i)]
	Rigorous convergence analysis of the LT-MOR to the high-fidelity solution.
	\item[(ii)]
	Construction of a precise rule to define the snapshots to be computed in the Laplace domain.
	\item[(iii)]
	As mentioned above, a thorough explanation of why the reduced
	basis constructed in the Laplace domain is suitable for the accurate
	numerical solution of the time-evolution parabolic problem is given
\end{itemize}

There exists a large body of work proposing \rev{approaches} to
\rev{approximate} parabolic problems.
We refer to Ref.~\refcite{thomee2007galerkin}
for a comprehensive survey. Indeed, recently space-time methods have
gained traction as a suitable approach\cite{andreev2013stability,fuhrer2021space,gantner2021further,gantner2023applications,langer2016space}.
We point out that the use of the Laplace transform for parabolic \rev{problems}
has been explored previously, see, e.g.,~Ref.~\refcite[Chapter 20]{thomee2007galerkin}
and Ref.~\refcite{sheen2003parallel}. These methods are based
on contour deformation techniques for the computation of the inverse Laplace 
transform, \rev{which in turn requires precise information of the pseudo-spectrum} 
of the leading elliptic operator involved in the parabolic problem.
\rev{To our knowledge, the closest algorithm aiming at approximating the 
Laplace transform or the frequency content of a given
dynamical system comes from the system theoretic and dynamical systems communities\cite{Benner2015,gugercin2008h_2}. 
However, none of them address the fast solution of said parabolic problems by using model order reduction techniques in the Laplace domain as we propose in the present work.}

\subsection{Outline}
This work is structured as follows. In Section~\ref{sec:Problem} we introduce the model problem 
to be considered throughout this work,~i.e. linear, second-order parabolic PDEs
\rev{with time-independent coefficients} in bounded domains. 
In Section~\ref{sec:FRB} we introduce the \rev{Laplace Transform Model Order Reduction} (LT-MOR for short). 
Next, in Section~\ref{sec:Basis} we provide a thorough mathematical analysis of the LT-MOR method.
In Section~\ref{sec:computation_rb} we discuss computational aspects concerning the implementation of the LT-MOR method.
Next, in Section~\ref{sec:numerical_results} we present numerical examples portraying the advantages of our method, and 
we conclude this work by providing in Section~\ref{sec:Summary} some final remarks and sketching possible directions of future research.

\section{Problem Model} 
\label{sec:Problem}
In this section, we introduce the problem model to be considered in this work.
\subsection{Notation and Functional Spaces}\label{sec:notation_functional_spaces}
Let $X,Y$ be real or complex Banach spaces.
We denote by a prime superscript, i.e.,~by $X'$,
the (topological) dual space of $X$, which consists of all bounded linear functionals acting on $X$.
In addition, we denote by $\mathscr{L}(X,Y)$ the Banach space of bounded linear operators
from $X$ into $Y$, and by $\mathscr{L}_{\text{iso}}(X,Y)$ we denote the (open) subspace of
$\mathscr{L}(X,Y)$ of bounded linear operators with a bounded inverse.

Let $\Omega \subset \IR^d$, $d \in \IN$, be a bounded Lipschitz domain with boundary
$\partial \Omega$. 
Let $L^p(\Omega;\mathbb{K})$, $\mathbb{K} \in \{\mathbb{R},\mathbb{C}\}$
and $p \in [1,\infty)$, be the Banach space over the field $\mathbb{K}$ of $p$-integrable functions in $\Omega$, with
the usual extension to $p=\infty$.
In particular, for $p=2$, we have that $L^2(\Omega;\mathbb{K})$ is a Hilbert
space when equipped with inner product $\dotp{\cdot}{\cdot}_{L^2(\Omega;\mathbb{K})}$
and the induced norm $\norm{\cdot}_{L^2(\Omega;\mathbb{K})} = \sqrt{\dotp{\cdot}{\cdot}_{L^2(\Omega;\mathbb{K})}}$, whereas
by $H^k(\Omega;\mathbb{K})$ we refer to the Hilbert space of functions
with $k$-th weak derivatives in $L^2(\Omega;\mathbb{K})$, which as it
is endowed with the standard inner product $(\cdot,\cdot)_{H^k(\Omega;\mathbb{K})}$
and the induced norm $\norm{\cdot}_{H^1(\Omega;\mathbb{K})}$.

For $\mathbb{K} \in \{\mathbb{R},\mathbb{C}\}$, we consider as well the
closed space $H^1_0(\Omega;\mathbb{K})$ of $H^1(\Omega;\mathbb{K})$
with vanishing Dirichlet trace on $\partial \Omega$, and denote by $H^{-1}(\Omega;\mathbb{K})$
its dual with respect to the $L^2(\Omega;\mathbb{K})$-duality pairing.
By identifying the dual space of $L^2(\Omega;\mathbb{K})$ with itself, we get that
$H^1_0(\Omega;\mathbb{K}) \subset L^2(\Omega;\mathbb{K}) \subset H^{-1}(\Omega;\mathbb{K})$
is a Gelfand triple. The duality pairing between $H^1_0(\Omega;\mathbb{K})$
and $H^{-1}(\Omega;\mathbb{K})$ is denoted by $\dual{\cdot}{\cdot}_{H^{-1}(\Omega;\mathbb{K})\times H^{1}_0(\Omega;\mathbb{K})} $.


Poincar\'e's inequality states that there exists $C_P(\Omega)>0$, depending solely on the
domain $\Omega$,
such that for any $H^1_0(\Omega;\mathbb{K})$ it holds
$\norm{u}_{L^2(\Omega;\mathbb{K})} \leq C_P(\Omega) \norm{\nabla u}_{L^2(\Omega;\mathbb{K})}$.
Therefore, $\dotp{u}{v}_{H^1_0(\Omega;\mathbb{K})} = \dotp{\nabla u}{\nabla v}_{L^2(\Omega;\mathbb{K})}$, for all $u,v \in H^1_0(\Omega)$, defines an inner product in $H^1_0(\Omega)$, thus
making $$\norm{u}_{H^1_0(\Omega;\mathbb{K})} \coloneqq \sqrt{(\nabla u,\nabla u)_{L^2(\Omega;\mathbb{K})}}$$
an equivalent norm to $\norm{\cdot}_{H^1(\Omega;\mathbb{K})}$ in $H^1_0(\Omega;\mathbb{K})$.

\subsection{Sobolev Spaces Involving Time}
Given $T>0$ we set $\mathfrak{J}=(0,T)$ and consider
either a complex or real Banach space $\left(X,\norm{\cdot}_{X}\right)$.
For each $r \in \IN_0$, we define 
$H^r(\mathfrak{J};{X})$, $r\in \IN_0$, as the Bochner space of $X$-valued, measurable
functions $u: \mathfrak{J} \rightarrow {X}$ satisfying
\begin{equation}
	\norm{u}_{H^r(\mathfrak{J} ; {X})}
	\coloneqq
	\left(
		\sum_{j=0}^r 
		\int\limits_0^T
		\norm{\partial_t^j u(t)}_{{X}}^2
		\text{d} 
		t
	\right)^{\frac{1}{2}}
	<
	\infty,
\end{equation}
where $\partial^j_t$ signifies the weak time derivative
of order $j\in \IN_0$, and  $\partial_t =  \partial^1_t$.
In particular, if $r=0$ we set $L^2(\mathfrak{J}, {X})\coloneqq
H^0(\mathfrak{J}; {X})$.

In addition, we set $\IR_+ \coloneqq \{t \in \IR : t>0 \}$, and
given $\alpha\geq0$ we denote by $L^2_\alpha(\IR_+;X)$ the Hilbert space
of $X$-valued, measurable functions $u:\IR_+ \rightarrow X$
satisfying
\begin{align}
	\norm{u}_{L^2_\alpha(\IR_+;X)}
	\coloneqq
	\sqrt{
		\dotp{
			u
		}{
			u
		}_{{L^2_\alpha(\IR_+;X)}}
	}
	< 
	\infty,
\end{align}
where for any $u,v \in L^2_\alpha(\IR_+;X)$
\begin{align}
	\dotp{u}{v}_{L^2_\alpha(\IR_+;X)} 
	\coloneqq
	\int\limits_{0}^{\infty} \dotp{u(t)}{v(t)}_{X} 
	\exp({-2\alpha t})
	\odif{t}.
\end{align}
defines an inner product in $L^2_\alpha(\IR_+;X)$.
In addition, as in Ref.~\refcite{dautray2012mathematical} (Chapter XVIII, Section 2.2, Definition 4)
we set
\begin{equation}\label{sec:funct_space_W}
	\mathcal{W}_\alpha(\IR_+;X)
	\coloneqq
	\left\{
		v \in L^2_\alpha(\IR_+;X):
		\partial_tv \in  L^2_\alpha(\IR_+;X')
	\right\},
\end{equation}
where $X'$ denotes the (topological) dual space of $X$,
and we equip it with the norm
\begin{equation}\label{eq:def_norm_W_alpha}
	\norm{
		u
	}_{\mathcal{W}_\alpha(\IR_+;X)}
	\coloneqq
	\left(
		\norm{u}^2_{L^2_\alpha(\IR_+;X)}
		+
		\norm{\partial_t u}^2_{L^2_\alpha(\IR_+;X')}
	\right)^{\frac{1}{2}},
\end{equation}
thus rendering it a Banach space
(cf. Ref.~\refcite{dautray2012mathematical}, Chapter XVIII, Section 2.2, Proposition 6).

\subsection{Linear, Second-Order Parabolic Problems}
\label{sec:parabolic_problems}
Let $\bm{A}(\bm{x}) \in L^\infty(\Omega;\mathbb{R}^{d\times d})$ be 
a symmetric, positive definite matrix satisfying
\begin{equation}\label{eq:elliptic_c}
	\esssinf_{\bm{x} \in \Omega} 
	\boldsymbol\xi^\top 
	\bm{A}(\bm{x}) 
	\boldsymbol\xi 
	\geq 
	\underline{c}_{\bm{A}} \norm{\boldsymbol\xi}^2_2, 
	\quad \forall \boldsymbol\xi \in \IR^{d} \backslash \{\bm{0}\},
\end{equation}	
for some $\underline{c}_{\bm{A}}>0$, and 
\begin{equation}\label{eq:continuous_c}
	\esssup_{\bm{x}\in \Omega} 
	\norm{
		\bm{A}(\bm{x}) 
	}_2
	\leq
	\overline{c}_{\bm{A}},
\end{equation}
with $\overline{c}_{\bm{A}}>0$.
Again, given a time horizon $T>0$ we
set $\mathfrak{J}=(0,T)$ and consider the following \rev{linear, second-order} parabolic problem in $\Omega$: 
We seek $u: \Omega \times \mathfrak{J} \rightarrow \IR$ satisfying 
\begin{equation}\label{eq:scalar_wave_eq}
	\partial_{t} u(\bm{x},t) 
	+
	\rev{
	\mathcal{A} u(\bm{x},t)
	}
	= 
	f(\bm{x},t), \quad (\bm{x},t) \in \Omega \times \mathfrak{J},
\end{equation}
where $\mathcal{A}v(\bm{x})=-\nabla \cdot \left(\bm{A} (\bm{x}) \nabla v(\bm{x}) \right) $,
$f :\Omega \times \mathfrak{J} \rightarrow \IR$, and
equipped with homogeneous Dirichlet boundary
and initial conditions
\begin{equation}\label{eq:scalar_wave_eq_bc}
	u(\bm{x},t) = 0, \quad (\bm{x},t) \in \Gamma \times \mathfrak{J}
	\quad
	\text{and}
	\quad
	u(\bm{x},0) 
	= 
	u_0(\bm{x}),
	\quad \bm{x} \in \Omega,
\end{equation}
respectively, where $u_0:\Omega \rightarrow \IR$.
\rev{
\begin{remark}
Consider
\begin{equation}
	\mu \in L^\infty(\Omega)
	\quad
	\text{and}
	\quad
	\boldsymbol{\beta} \in W^{1,\infty}(\Omega;\mathbb{R}^d),
\end{equation}
where $d$ is the problem's physical dimension. Assume that there exists
$\mu_0 >0$ such that
\begin{equation}
	\Lambda
	\coloneqq
	\mu
	-
	\nabla \cdot \boldsymbol{\beta} 
	\geq
	\mu_0
	\text{ a.e. in }
	\Omega.
\end{equation}
In general, we can consider the more general second-order partial differential operator of the form
\begin{equation}\label{eq:diff_op}
	\mathcal{A}v(\bm{x})
	=
	-
	\nabla \cdot \left(\bm{A} (\bm{x}) \nabla v(\bm{x}) \right) 
	+
	\boldsymbol{\beta}(\bm{x})
	\cdot
	\nabla v(\bm{x})
	+
	\mu(\bm{x}) v(\bm{x})
\end{equation}
instead of the $\mathcal{A}$ in \eqref{eq:scalar_wave_eq}.
\end{remark}
}
\subsection{Variational Formulation of Parabolic Problems}\label{sec:var_for}
Firstly, let us define the \emph{sesquilinear} form
$\mathsf{a}: H^1_0(\Omega;\mathbb{C}) \times H^1_0(\Omega;\mathbb{C}) \rightarrow \IC$ as
\rev{
\begin{equation}\label{eq:sesq_form}
\begin{aligned}
	\mathsf{a}(u,v) 
	\coloneqq 
	\int\limits_{\Omega}
	\nabla u(\bm{x})^\top \bm{A} (\bm{x})
	\overline{\nabla v(\bm{x})}
	\odif{\bm{x}},
	\quad
	\forall u,v \in H^1_0(\Omega;\mathbb{C}).
\end{aligned}
\end{equation}
}%
It follows that
\begin{equation}
	\Re\left\{\mathsf{a}(u,u)\right\}
	\geq
	\underline{c}_{\bm{A}}
	\norm{u}^2_{H^1_0(\Omega;\mathbb{C})},
	\;\;
	\forall
	u\in H^1_0(\Omega;\mathbb{C}),
\end{equation}
and
\begin{equation}
\begin{aligned}
	\snorm{
		\mathsf{a}(u,v)	
	}
	\leq
	\overline{c}_{\bm{A}}
	\norm{u}_{H^1_0(\Omega;\mathbb{C})}
	\norm{v}_{H^1_0(\Omega;\mathbb{C})}
	\;\;
	\forall
	u,v\in H^1_0(\Omega;\mathbb{C}),
\end{aligned}
\end{equation}
respectively. 
\rev{Consequently, the sesquilinear form $\mathsf{a}(\cdot)$
is elliptic (a property sometimes referred to as strongly coercive)
and continuous.
}%
Observe that when
restricted to real-valued Sobolev spaces
$\mathsf{a}(\cdot,\cdot)$ as in \eqref{eq:sesq_form}
becomes a bilinear form. 

We cast \eqref{eq:scalar_wave_eq} together with 
the homogeneous Dirichlet boundary conditions into a variational formulation
as stated below.

\begin{problem}[Variational Formulation of the Parabolic Problem] 
\label{pbm:wave_equation}
Let $u_0 \in L^2(\Omega;\mathbb{R})$ and $f \in L^2(\mathfrak{J};L^2(\Omega;\mathbb{R}))$
be given.

We seek $u \in L^2(\mathfrak{J};H^1_0(\Omega;\mathbb{R}))$ with 
$\partial_t u \in L^2(\mathfrak{J};H^{-1}(\Omega;\mathbb{R}))$
such that for a.e. $t>0$ it holds
\begin{align}\label{eq:wave_eq_variational}
	\dual{\partial_{t}u\rev{(t)}}{v}_{H^{-1}(\Omega)\times H^{1}_0(\Omega)} 
	+ 
	\mathsf{a}(u\rev{(t)},v) 
	= 
	\dotp{f\rev{(t)}}{v}_{L^2(\Omega)},
	\quad
	\forall
	v \in H^1_0(\Omega),
\end{align}
and satisfying $u(0) = u_0$ in $L^2(\Omega;\mathbb{R})$.
\end{problem}

\begin{remark}[Real- and Complex-valued Sobolev Spaces]
In Section~\ref{sec:notation_functional_spaces} and Section~\ref{sec:var_for},
we \rev{have} been exhaustive in differentiating real- and complex-valued 
Sobolev spaces. Parabolic problems as the one described in Section~\ref{sec:parabolic_problems}
\rev{are} usually set in a real-valued framework. However, ahead in Section~\ref{sec:FRB}, the application
of the Laplace transform requires the use of complex-valued Sobolev spaces. 
For the sake of simplicity, in what follows for real-valued Sobolev spaces we keep
the notation described in  Section~\ref{sec:notation_functional_spaces}, e.g.
$H^1_0(\Omega;\mathbb{R})$ and $H^{-1}(\Omega;\mathbb{R})$, whereas
for complex-valued ones we remove the reference to the field, i.e.~we set
$H^1_0(\Omega) \equiv H^1_0(\Omega;\mathbb{C})$. In the definition 
of norms, inner products, and duality pairings, we just drop the field, i.e. we write, for example,
$\dual{\cdot}{\cdot}_{H^{-1}(\Omega)\times H^{1}_0(\Omega)} $.

\end{remark}

\subsection{Semi-Discrete Problem}
\label{ssec:fe_problem}
Throughout, let $\{\mathbb{V}_h\}_{h>0}$ be a family of finite
dimensional subspaces of $H^1_0(\Omega;\mathbb{R})$ with discretization
parameter $h>0$.
Set $N_{h} = \text{dim}(\IV_{h})$ and consider a basis
 $\{\varphi_1,\dots,\varphi_{N_{h}}\}$ of $\mathbb{V}_{h}$.
In addition, we consider the \emph{complexification} $\mathbb{V}^\mathbb{C}_h$
of $\mathbb{V}_h$ as defined in Chapter 1, p. 53, of Ref.~\refcite{roman2005advanced}.
 
\rev{For any subspace $X\subset H^1_0(\Omega)$
we define $\mathsf{P}_{X}: H^1_0(\Omega) \rightarrow \IV_{h}$
as the projection operator onto $X$, i.e.~for each $v \in H^1_0(\Omega)$, $\mathsf{P}_{X} v$
is defined as the unique solution of the following variational problem
\begin{equation}\label{eq:projection_H1}
	\dotp{\mathsf{P}_{X} v}{w}_{H^1_0(\Omega)}
	=
	\dotp{v}{w}_{H^1_0(\Omega)},
	\quad
	\forall
	w \in X.
\end{equation}}%
%
\rev{In particular, we set $\mathsf{P}_h \coloneqq \mathsf{P}_{\mathbb{V}_h}$.}
Equipped with \rev{these} tools, we state the semi-discrete version of Problem~\ref{pbm:wave_equation}.
\begin{problem}[Semi-discrete Formulation of Problem~\ref{pbm:wave_equation}]
\label{prbm:semi_discrete_problem}
Let $u_0 \in \rev{H^1_0}(\Omega;\mathbb{R})$ and $f \in L^2({\mathfrak{J}}; L^2(\Omega;\mathbb{R}))$
be given. 
We seek $u_h \in H^1\left({\mathfrak{J}};\mathbb{V}_h\right)$
such that for a.e. $t \in \mathfrak{J}$
it holds
\begin{align}\label{eq:semi_discrete_time}
	\dotp{\partial_t u_h(t)}{v_h}_{L^2(\Omega)} 
	+ 
	\mathsf{a} \left(u_h(t),v_h\right) 
	= 
	\dotp{f(t)}{v_h}_{L^2(\Omega)},  
	\quad \forall v_h \in \mathbb{V}_h,
\end{align}
with initial conditions \rev{$u_h(0) = \mathsf{P}_h u_0 \in \mathbb{V}_h$}.
\end{problem}

We proceed to describe Problem~\ref{prbm:semi_discrete_problem} in matrix form.
To this end, let us consider the following solution \emph{ansatz}
\begin{equation}
	u_h(t) 
	= 
	\sum_{j=1}^{N_h} 
	\mathsf{u}_j(t)
	\varphi_j 
	\in
	\mathbb{V}_h, 
	\quad 
	\text{for a.e. } 
	t\in {\mathfrak{J}},
\end{equation}
and set $\bm{\mathsf{u}}(t) = (\mathsf{u}_1(t),\dots, \mathsf{u}_{N_h}(t))^{\top} \in \IR^{N_h}$.
We also define $\bm{\mathsf{M}}_h \in \IR^{N_h \times N_h}$
and $\bm{\mathsf{A}}_h \in \IR^{N_h \times N_h}$ as
\begin{equation}
	(\bm{\mathsf{M}}_h)_{i,j} 
	\coloneqq
	\dotp{\varphi_i}{\varphi_j}_{L^2(\Omega)} 
	\quad 
	\text{and} 
	\quad 
	(\bm{\mathsf{A}}_h)_{i,j} 
	\coloneqq
	\mathsf{a} \left(\varphi_i,\varphi_j \right), 
	\quad i,j \in \{1,\dots,N_h\},
\end{equation}
referred to as the mass matrix and the stiffness matrix of the (bilinear in this case)
form $\mathsf{a}(\cdot,\cdot)$.
In addition, we set $\bm{\mathsf{B}}_h \in \IR^{N_h \times N_h}$ as
\begin{equation}
	(\bm{\mathsf{B}}_h)_{i,j} 
	\coloneqq
	\rev{
	\dotp{
		\varphi_i
	}{	
		\varphi_j
	}_{H^1_0(\Omega)}, 
	}
	\quad i,j \in \{1,\dots,N_h\},
\end{equation}
together with the discrete right-hand side
\begin{equation}
	(\bm{\mathsf{f}}_h(t))_{i} 
	\coloneqq
	\dotp{
		f(t)
	}{
		\varphi_i
	}_{L^2(\Omega)}, 
	\quad i\in \{1,\dots,N_h\}.
\end{equation}
Then, Problem~\ref{prbm:semi_discrete_problem} reads as follows:
Provided that $f \in \mathscr{C}^0(\overline{\mathfrak{J}}; L^2(\Omega))$,
we seek $\bm{\mathsf{u}}  \in \mathscr{C}^1\left(\overline{\mathfrak{J}};\mathbb{R}^{N_h}\right)$
such that
\begin{equation}
	\bm{\mathsf{M}}_h
	\odv[order=1]{}{t}
	\bm{\mathsf{u}}(t)  
	+
	\bm{\mathsf{A}}_h
	\bm{\mathsf{u}}(t)  
	=
	\bm{\mathsf{f}}(t),
	\quad
	\rev{t\in \mathfrak{I},}
\end{equation}
with $\bm{\mathsf{u}}(0)  = \bm{\mathsf{u}}_{0,h}\in \mathbb{R}^{N_h}$
and $ \bm{\mathsf{u}}_{0,h}$ such that
\begin{equation}
	u_h(0) 
	=
	\sum_{j=1}^{N_h}
	\left(
		\bm{\mathsf{u}}_{0,h}
	\right)_j
	\varphi_j,
\end{equation}
where \rev{$u_h(0) = \mathsf{P}_h u_0 \in \mathbb{V}_h$}
is as in Problem~\ref{prbm:semi_discrete_problem}.

\subsection{Construction of a RB Using POD: The Time-dependent Approach}
\label{sec:reduced_time_dependent_problem_traditional}
We are interested in \rev{the construction of} a reduced basis for the 
discrete in space, yet continuous in time, solution manifold
\begin{equation}\label{eq:tsnap}
	\mathcal{M}_h 
	= 
	\left\{ 
		u_h(t)\vert \;
		t
		\in
		\overline{\mathfrak{J}}
	\right\}
	\rev{\subset 
	\mathbb{V}_h.}
\end{equation}
To this end, firstly we perform a discretization in time.
Given $N_t \in \IN$, we consider a uniform partition of $\overline{\mathfrak{J}}$
defined as $t_j = \frac{j}{N_t}T$ for $j=0,\dots,N_t$.
The traditional approach to model reduction of time-dependent problems
consists \rev{of} finding a reduced space $\IV_R^{(\textrm{rb})}\subset \IV_h$
of dimension $R\in \IN$ that is hopefully considerably smaller
than that of $\IV_h$
such that
\begin{equation} \label{eq:tPOD}
	\IV_R^{(\textrm{rb})}  
	= 
	\argmin_{
		\substack{
		X_R \subset \IV_h
		\\ 
		\text{dim}(X_R) \leq R
		}
	}
    	\sum_{j=0}^{N_t} 
	\norm{
		u_h(t_j) 
		- 
		\mathsf{P}_{X_R}
		u_h(t_j) 
	}^2_{H^1_0(\Omega;\mathbb{R})},
\end{equation}
where $u_h(t_j)$ corresponds to the solution of Problem~\ref{prbm:semi_discrete_problem}
at time $t_j$\footnote{For simplicity, we disregard the effect of the time discretization
in the computation of $u_h(t_j)$.}.

The formulation stated in \eqref{eq:tPOD} may be expressed in algebraic form as follows
\begin{equation}\label{eq:tPOD_algebraic}
	\bm{\Phi}^{\textrm{(rb)}}_R
	=
	\min_{{\bm\Phi} \in \mathscr{V}_{R}}
	\sum_{j=0}^{N_t}
	\norm{
		\bm{\mathsf{u}}_h(t_j)
		-
		\bm{\Phi}
		\bm{\Phi}^\top 
		{\bm{\mathsf{B}}_h}
		\bm{\mathsf{u}}_h(t_j)
	}_{{\bm{\mathsf{B}}_h}}^2,
\end{equation}
where
$$
	\mathscr{V}_{R}
	\coloneqq
	\left\{
		\bm{\Phi} \in \IR^{N_h \times R}:
		\bm{\Phi}^\top 
		\bm{\mathsf{B}}_h
		\bm{\Phi}
		=
		\bm{\mathsf{I}}_{R}
	\right\},
$$
and
$
	\norm{\bm{\mathsf{v}}}_{\bm{\mathsf{B}}_h}
	=
	\sqrt{
		\left(\bm{\mathsf{B}}_h\bm{\mathsf{v}},\bm{\mathsf{v}}\right)_{\IC^{N_h}}
	},
	\;
	\bm{\mathsf{v}} \in \IC^{N_h}.
$
The connection between the solution to \eqref{eq:tPOD}
and \eqref{eq:tPOD_algebraic} is as follows: Set
\begin{equation}\label{eq:def_basis_rb}
	\varphi^{\textrm{(rb)}}_k 
	= 
	\sum_{j=1}^{N_h} 
	\left( \bm{\phi}^{\textrm{(rb)}}_k \right)_j \varphi_{j}
	\in 
	\IV_h,
	\quad
	k=1,\dots,R,
\end{equation}
where
$
	\bm{\Phi}^{\textrm{(rb)}}_R
	=
	\left(
		\bm{\phi}^{\textrm{(rb)}}_1 ,
		\dots,
		\bm{\phi}^{\textrm{(rb)}}_R 
	\right),
$ and \rev{$\left( \bm{\phi}^{\textrm{(rb)}}_k \right)_j$} signifies the $j$-th component
of $\bm{\phi}^{\textrm{(rb)}}_k$.
Then $\left\{\varphi^{\textrm{(rb)}}_1,\cdots,\varphi^{\textrm{(rb)}}_R\right\}$
is an orthonormal basis of $\IV_R^{(\textrm{rb})}$ in the $H^1_0(\Omega)$-inner
product.

Let us define the snapshot matrix 
\begin{equation}
	{\bm{\mathsf{S}}}
	\coloneqq
	\left(
		\bm{\mathsf{u}}_h(t_0),
		\bm{\mathsf{u}}_h(t_1),
		\dots,
		\bm{\mathsf{u}}_h(t_{N_t-1}),
		\bm{\mathsf{u}}_h(t_{N_t})
	\right)
	\in 
	\IR^{N_h \times (N_t+1)}.
\end{equation}
and consider the matrix $\widetilde{\bm{\mathsf{S}}} = \bm{\mathsf{R}}_h{\bm{\mathsf{S}}}$,
where $\bm{\mathsf{B}}_h =  \bm{\mathsf{R}}^\top_h  \bm{\mathsf{R}}_h$ is the
Cholesky decomposition of  $\bm{\mathsf{B}}_h$ with $ \bm{\mathsf{R}}_h$ an upper triangular
matrix.

Let $\widetilde{\bm{\mathsf{S}}} = \widetilde{\bm{\mathsf{U}}} \Sigma \widetilde{\bm{\mathsf{V}}}^\top$ 
be the SVD of $\widetilde{\bm{\mathsf{S}}}$, where
\begin{equation}
	\widetilde{\bm{\mathsf{U}}}
	=
	\left(
		\widetilde{\boldsymbol{\zeta}}_1,
		\ldots,
		\widetilde{\boldsymbol{\zeta}}_{N_h}
	\right)\in \mathbb{R}^{N_h \times N_h},
	\quad 
	 \widetilde{\bm{\mathsf{V}}}
	=
	\left(
		\widetilde{\boldsymbol{\psi}}_1,
		\ldots,
		\widetilde{\boldsymbol{\psi}}_{N_t}
	\right) 
	\in 
	\mathbb{R}^{(N_t+1) \times (N_t+1)}
\end{equation}
are orthogonal matrices and 
$\Sigma \in \mathbb{R}^{N_h \times N_t}$ 
contains the singular values $\sigma_1 \geq \cdots \geq \sigma_r>0$
of $\widetilde{\bm{\mathsf{S}}}$
in decreasing order, where $r = \text{rank}(\widetilde{\bm{\mathsf{S}}}) \leq \min\{N_h,N_{t}+1\}$.

It follows from the Schmidt-Eckart-Young theorem, as stated in Proposition 6.2 of Ref.~\refcite{quarteroni2015reduced},
that for any $R\leq r$ the POD basis 
\begin{equation}\label{eq:reduced_basis_Phi}
	\bm{\Phi}^{\textrm{(rb)}}_R
	=
	\left(
		\bm{\mathsf{R}}_h ^{-1}\widetilde{\boldsymbol{\zeta}}_1, 
		\ldots,
		\bm{\mathsf{R}}_h ^{-1} \widetilde{\boldsymbol{\zeta}}_{R}
	\right),
\end{equation} 
is solution to the minimization problem \eqref{eq:tPOD_algebraic}, and it holds \rev{that}
\begin{equation}
\begin{aligned}	
	\min_{\bm{\Phi} \in \mathscr{V}_{R}}
	\sum_{i=0}^{N_t}
	\norm{
		\bm{\mathsf{u}}_h(t_i)
		-
		\bm{\Phi}
		\bm{\Phi}^\top 
		{\bm{\mathsf{B}}_h}
		\bm{\mathsf{u}}_h(t_i)
	}_{{\bm{\mathsf{B}}_h}}^2
	\\
	&
	\hspace{-2cm}
	=
	\sum_{i=0}^{n_t}
	\norm{
		\bm{\mathsf{u}}_h(t_i)
		-
		\bm{\Phi}^\textrm{(rb)}_R
		\left(
			\bm{\Phi}^\textrm{(rb)}_R
		\right)^\top
		{\bm{\mathsf{B}}_h}
		\bm{\mathsf{u}}_h(t_i)
	}_{{\bm{\mathsf{B}}_h}}^2
	\\
	&
	\hspace{-2cm}
	=
	\sum_{i=N+1}^r \sigma_i^2.
\end{aligned}
\end{equation}

Now that we have constructed the reduced basis, we can follow the
traditional approach to model reduction for time-dependent problems
by projecting Problem~\ref{pr:sdp} onto the reduced space $\IV_R^{\textrm{(rb)}}$. 

\begin{problem}[Reduced Semi-Discrete Problem] \label{pr:sdpr}
Let $u_0\in \rev{H^1_0}(\Omega;\mathbb{R})$ and $f \in L^2(\mathfrak{J};L^2(\Omega;\mathbb{R}))$.
We seek $u^{\normalfont\textrm{(rb)}}_{R} \in H^1\left(\mathfrak{J};\IV^{\normalfont\textrm{(rb)}}_R\right)$
such that for a.e. $t \in \mathfrak{J}$ it holds that
\begin{equation}\label{eq:semi_discrete_rom}
\begin{aligned}
	\dotp{
		\partial_t
		u_R^{\normalfont\textrm{(rb)}}(t)
	}{
		v_R^{\normalfont\textrm{(rb)}}
	}_{L^2(\Omega)}
	+ 
	\mathsf{a}
	\left(
		u_R^{\normalfont\textrm{(rb)}}(t)
		,
		v_R^{\normalfont\textrm{(rb)}}
	\right) 
	= 
	\dotp{f(t)}{v_R^{\normalfont\textrm{(rb)}} }_{L^2(\Omega)}, 
	\quad \forall v_R^{\normalfont\textrm{(rb)}} \in \IV_{R}^{\normalfont\textrm{(rb)}},
\end{aligned}
\end{equation}
with initial condition $u_R^{\normalfont\textrm{(rb)}}(0) \in \IV_{R}^{\normalfont\textrm{(rb)}}$
given as $u_R^{\normalfont\textrm{(rb)}}(0) = \mathsf{P}^{\normalfont\textrm{(rb)}}_R {u_{0,h}}$.
\end{problem}

We set $\bm{\mathsf{u}}^{\textrm{(rb)}}(t) = (\mathsf{u}^{\textrm{(rb)}}_1(t),\dots, \mathsf{u}^{\textrm{(rb)}}_{R}(t))^{\top} \in \IR^{R}$
and define
$\bm{\mathsf{M}}^{\textrm{(rb)}}_R\in \IR^{R \times R}$ and 
$\bm{\mathsf{A}}^{\textrm{(rb)}}_R\in \IR^{R \times R}$ as
\begin{equation}
	\bm{\mathsf{M}}^{\textrm{(rb)}}_R
	= 
	\bm{\Phi}^{{\textrm{(rb)}}\top}_R
	\bm{\mathsf{M}}_h 
	\bm{\Phi}^{\textrm{(rb)}}_R
	\quad
	\text{and}
	\quad
	\bm{\mathsf{A}}^{\textrm{(rb)}}_R
	= 
	\bm{\Phi}^{{\textrm{(rb)}}\top}_R
	\bm{\mathsf{A}}_h 
	\bm{\Phi}^{\textrm{(rb)}}_R, 
\end{equation}
together with
\begin{equation}
	\bm{\mathsf{f}}^{\textrm{(rb)}}_R(t)
	= 
	\bm{\Phi}^{{\textrm{(rb)}}\top}_R
	\bm{\mathsf{f}}_h(t),
\end{equation}
where $\bm{\Phi}^{\textrm{(rb)}}_R\in \IR^{N_h \times R}$
is as in \eqref{eq:reduced_basis_Phi}.

Then, Problem~\ref{pr:sdpr} reads as follows:
Provided that $f \in \mathscr{C}^0(\overline{\mathfrak{J}}; L^2(\Omega))$,
we seek $\bm{\mathsf{u}}^{\textrm{(rb)}}_R \in \mathscr{C}^1\left(\overline{\mathfrak{J}};\mathbb{R}^{R}\right)$
such that
\begin{equation}
	\bm{\mathsf{M}}^{\textrm{(rb)}}_R
	\odv[order=1]{}{t}
	\bm{\mathsf{u}}^{\textrm{(rb)}}_R(t)
	+
	\bm{\mathsf{A}}^{\textrm{(rb)}}_R
	\bm{\mathsf{u}}^{\textrm{(rb)}}_R(t)  
	=
	\bm{\mathsf{f}}^{\textrm{(rb)}}_R(t)
\end{equation}
with \rev{$\bm{\mathsf{u}}^{\textrm{(rb)}}_R(0) = \bm{\Phi}^{{\textrm{(rb)}}\top}_R \bm{\mathsf{u}}_{0,h}
\in \mathbb{R}^R$}.
\begin{remark}\label{rmk:compute_rb_time}
We point out that the exercise of computing a reduced basis for the time-evolution problem
as presented here lacks any \rev{practical} usefulness as we need to solve the high-fidelity problem in order
to compute the basis itself.
The only use of this technique would be to compress the solution of
the \rev{time-evolution} problem for storage purposes. 
\end{remark}

\section{The Laplace Transform Reduced Basis Method}
\label{sec:FRB}
In this section, we describe a new approach \rev{to} the construction
of a reduced space tailored to linear parabolic problems.
Instead of directly performing a time discretization of Problem~\ref{prbm:semi_discrete_problem}
as in Section~\ref{sec:reduced_time_dependent_problem_traditional},
we apply the Laplace transform to the high-fidelity problem and obtain an
elliptic PDE parametrically \rev{dependent on}
the complex Laplace \rev{variable}.

Recall that the Laplace transform of $f: [0,\infty)\rightarrow \IC$
is defined as
\begin{equation} \label{eq:laplace}
	\mathcal{L}
	\left\{
		f
	\right \}(s) 
	\coloneqq 
	\int\limits_{0}^{\infty}\exp({-st})f(t) \odif{t},
	\quad
	s \in \Pi_+,
\end{equation}
where $\Pi_+ \coloneqq \{ s \in \IC: \Re\{s\} >0 \}$ denotes the right complex half-plane. 
Throughout, we also use the notation $\widehat{f}(s) = \mathcal{L}\{f\}(s)$ for $s \in \Pi_+$
to denote the Laplace transform of a function.
The inverse Laplace transform admits the following representation
\begin{equation}
	f(t)
	=
	\frac{1}{2\pi \imath}
	\int\limits_{\gamma-\imath \infty}^{\gamma+\imath \infty}
	\exp({st})
	\widehat{f}(s)
	\text{d}s,
\end{equation}
where $\gamma>0$ defines a vertical contour in the complex plane
which is chosen so that all singularities of $\widehat{f}(s)$ are to the left of it.
This representation of the inverse Laplace transform is known as the Bromwich integral.
A large body of work aims at numerically computing the inverse Laplace transform.
We refer to Ref.~\refcite{davies1979numerical} for a survey and comparison of different methods.

Formally, the application of the Laplace transform to Problem~\ref{prbm:semi_discrete_problem}
together with the well-known property
$
	\mathcal{L}
	\left\{
		\partial_t
		f
	\right\}(s)
	=
	s\widehat{f}(s)-f(0)
$
yields the following problem depending on the complex
Laplace \rev{variable} $s\in \Pi_+$.

\begin{problem}[Laplace Domain Discrete Problem]\label{pbrm:laplace_discrete}
\label{pr:sdp}
Let $u_0\in \rev{H^1_0(\Omega;\mathbb{R})}$ and $f \in L^2_\alpha(\IR_+;L^2(\Omega;\mathbb{R}))$ for some $\alpha>0$.
For each $s \in \Pi_\alpha$ we seek 
$\widehat{u}_h(s) \in \IV^{\mathbb{C}}_h$, i.e.,~the complexification of $\IV_h$ introduced at the beginning of Section~\ref{ssec:fe_problem},
such that for all $v_h \in \IV^{\mathbb{C}}_h$ it holds that
\begin{align}\label{eq:semi_discrete}
	s\dotp{\widehat{u}_h(s)}{v_h}_{L^2(\Omega)} 
	+ 
	\mathsf{a}
	\left(
		\widehat{u}_h(s),v_h
	\right) 
	= 
	\dotp{\widehat{f}(s)}{v_h}_{L^2(\Omega)}
	+
	\dotp{u_{0,h}}{v_h}_{L^2(\Omega)} , 
\end{align}
where $\widehat{f}(s) =\mathcal{L}\{f\}(s)$ corresponds to the Laplace transform
of $f$ and $u_{0,h} = \rev{ \mathsf{P}_h} u_0$.
\end{problem}

We construct a reduced basis using POD as in Section~\ref{sec:reduced_time_dependent_problem_traditional}.
However, in the approach described here we rely on solutions
of Problem~\ref{pr:sdp} on a carefully selected collection of $M \in \IN$
complex points $\mathcal{P}_s = \{s_1,\dots,s_{M}\}\subset \Pi_+$.
More precisely, we are interested in finding a finite dimensional subspace
$\mathbb{V}^{\text{(rb)}}_R \subset H^1_0(\Omega;\mathbb{R})$ of dimension $R\in \IN$ such that
\begin{equation}\label{eq:fPOD}
	\mathbb{V}^{\text{(rb)}}_R
	=
	\argmin_{
	\substack{
		\IV_R \subset \mathbb{V}_h
		\\
		\text{dim}(\IV_R)\leq R	
	}
	}
    	\sum_{j=1}^{M} 
	\omega_j
	\norm{
		\Re
		\left\{
			\widehat{u}_h(s_j) 
		\right\}
		- 
		\mathsf{P}_{\IV_R}
		\Re
		\left\{
			\widehat{u}_h(s_j) 
		\right\}
	}^2_{H^1_0(\Omega)},
\end{equation}
where $\{\omega_1,\ldots,\omega_{M}\}$ are strictly positive weights.
We remark at this point that in \eqref{eq:fPOD} we have kept 
only the real part of the snapshots in the construction of
the snapshot matrix. A thorough justification of this choice 
is presented ahead in Section~\ref{sec:computation_rb}.

As in Section~\ref{sec:reduced_time_dependent_problem_traditional},
the formulation introduced in \eqref{eq:fPOD} may be expressed in algebraic form as follows
\begin{equation}\label{eq:tPOD_algebraic_frequency}
	\bm{\Phi}^{\textrm{(rb)}}_R
	=
	\argmin_{\bm\Phi \in \mathscr{V}_{R}}
	\sum_{j=1}^{M}
	\omega_j
	\norm{
		\Re\{\widehat{\bm{\mathsf{u}}}_h(s_j)\}
		- 
		\bm{\Phi}
		\bm{\Phi}^\top 
		{\bm{\mathsf{B}}_h}
		\Re\{\widehat{\bm{\mathsf{u}}}_h(s_j)\}
	}^2_{\bm{\mathsf{B}}_h},
\end{equation}
where $\widehat{\bm{\mathsf{u}}}_h(s)$ at $s \in \mathcal{P}_s$
is such that
\begin{equation}
	\widehat{u}_h(s)
	=
	\sum_{j=1}^{N_h}
	\left(
		\widehat{\bm{\mathsf{u}}}_h(s)
	\right)_j
	\varphi_j 
	\in
	\IV^{\mathbb{C}}_h.
\end{equation}  

Let us define the snapshot matrix 
containing only the real part of the solution 
to Problem~\ref{pr:sdp} for the \rev{instances} of the Laplace \rev{variable} in $\mathcal{P}_s $
\begin{equation}\label{eq:snapshot_M_LT_RB}
	{\bm{\mathsf{S}}}
	\coloneqq
	\left(
		\Re\{ \widehat{\bm{\mathsf{u}}}_h(s_1)\},
		\Re\{ \widehat{\bm{\mathsf{u}}}_h(s_2)\},
		\dots,
		\Re\{ \widehat{\bm{\mathsf{u}}}_h(s_{M-1}) \},
		\Re\{ \widehat{\bm{\mathsf{u}}}_h(s_{M})\}
	\right)
	\in 
	\IR^{N_h \times M},
\end{equation}
and define
\begin{equation}
	{\bm{\mathsf{D}}}
	=
	\text{diag}
	\left(\omega_1, \ldots, \omega_{M} \right) \in \mathbb{R}^{M \times M}.
\end{equation}
Set $\check{\bm{\mathsf{S}}} = \bm{\mathsf{R}}_h{\bm{\mathsf{S}}} {\bm{\mathsf{D}}}^{\frac{1}{2}}$
and consider its SVD $\check{\bm{\mathsf{S}}} = \check{\bm{\mathsf{U}}}\check{\Sigma} \check{\bm{\mathsf{V}}}^\top$,
where
\begin{equation}
	\check{\bm{\mathsf{U}}}
	=
	\left(
		\check{\boldsymbol{\zeta}}_1,
		\ldots,
		\check{\boldsymbol{\zeta}}_{N_h}
	\right)\in \mathbb{R}^{N_h \times N_h},
	\quad 
	\check{\bm{\mathsf{V}}}
	=
	\left(
		\check{\boldsymbol{\psi}}_1,
		\ldots,
		\check{\boldsymbol{\psi}}_{M}
	\right) 
	\in 
	\mathbb{R}^{M \times M}
\end{equation}
are orthogonal matrices, referred to as the left and right singular vectors of $\check\Sigma$,
respectively, and 
$\check\Sigma=\operatorname{diag}\left(\check\sigma_1, \ldots, \check\sigma_r\right) 
\in \mathbb{R}^{N_h \times M}$ with $\check\sigma_1 \geq \cdots \geq \check\sigma_r>0$,
where $r \leq  \min\{N_h,M\}$ \rev{is} the rank of $\check{\bm{\mathsf{S}}}$.

It follows from the Schmidt-Eckart-Young theorem, as stated in Proposition 6.2 of Ref.~\refcite{quarteroni2015reduced},
that for any $R\leq r = \text{rank}(\check{\bm{\mathsf{S}}}) \leq \min\{N_h,M\}$ the POD basis 
\begin{equation}\label{eq:reduced_basis_Phi_2}
	\bm{\Phi}^{\textrm{(rb)}}_R
	=
	\left(
		\bm{\mathsf{R}}^{-1}_h\check{\boldsymbol{\zeta}}_1, 
		\ldots,
		\bm{\mathsf{R}}^{-1}_h \check{\boldsymbol{\zeta}}_{R}
	\right),
\end{equation} 
which consist of the $R$ first left singular vectors of $\check{\bm{\mathsf{S}}}$
multiplied on the left by $\bm{\mathsf{R}}^{-1}_h$, 
is the unique solution to \eqref{eq:tPOD_algebraic_frequency},
and it holds that
\begin{equation}
\begin{aligned}	
	\min_{\bm{\Phi} \in \mathscr{V}_{R}}
	\sum_{j=1}^{M}
	\omega_j
	\norm{
		\Re\{\widehat{\bm{\mathsf{u}}}_h(s_j)\}
		- 
		\bm{\Phi}
		\bm{\Phi}^\top 
		{\bm{\mathsf{B}}_h}
		\Re\{\widehat{\bm{\mathsf{u}}}_h(s_j)\}
	}^2_{\bm{\mathsf{B}}_h}.
	\\
	&
	\hspace{-5cm}
	=
	\sum_{j=1}^{M}
	\omega_j
	\norm{
		\Re\{\widehat{\bm{\mathsf{u}}}_h(s_j)\}
		- 
		\bm{\Phi}^{\textrm{(rb)}}_R
		\bm{\Phi}^{\textrm{(rb)}\top}_R
		{\bm{\mathsf{B}}_h}
		\Re\{\widehat{\bm{\mathsf{u}}}_h(s_j)\}
	}^2_{\bm{\mathsf{B}}_h}
	\\
	&
	\hspace{-5cm}
	=
	\sum_{j=R+1}^r \check{\sigma}_j^2.
\end{aligned}
\end{equation}
\begin{remark}\label{rmk:comp_Vrb}
As pointed out in Ref.~\refcite{quarteroni2015reduced}, Algorithm 6.3,
for $N_h\leq M$ one can compute the reduced basis by following the 
following procedure, which does not require the computation of the Cholesky decomposition
of ${\bm{\mathsf{B}}_h}$:
\begin{itemize}
	\item[(i)]
	Compute the correlation matrix $\check{\bm{\mathsf{C}}} =  \check{\bm{\mathsf{S}}}^\top {\bm{\mathsf{B}}_h} \check{\bm{\mathsf{S}}} $.
	\item[(ii)]
	Solve the eigenvalue problem $\check{\bm{\mathsf{C}}} \check{\bm{\psi}}_i = \sigma^2_i \check{\bm{\psi}}_i$, for $i=1,\dots,r$.
	\item[(iii)]
	Set
	\begin{equation}
	\bm{\Phi}^{\normalfont\textrm{(rb)}}_R
	=
	\left(
		\frac{1}{\sigma_i}\check{\bm{\mathsf{S}}} \check{\bm{\psi}}_1,
		\ldots,
		\frac{1}{\sigma_R}\check{\bm{\mathsf{S}}} \check{\bm{\psi}}_R
	\right).
	\end{equation} 
\end{itemize}
As it will be discussed ahead in Section~\ref{eq:snap_shot_selection} only a \rev{handful} of snapshots
in the Laplace domain are \rev{needed} to compute the reduced basis. Therefore, this algorithm for
the construction of the reduced basis is better suited for the \rev{LT-MOR} method.
\end{remark}

The \rev{LT-MOR} method poses the following questions:
\begin{itemize}
	\item[{\sf \encircle{Q1}}]
	Why is $\bm{\Phi}^{\textrm{(rb)}}_R$ as in \eqref{eq:tPOD_algebraic_frequency}
	a suitable reduced basis for Problem~\ref{pr:sdpr}?
	\item[{\sf \encircle{Q2}}]
	How does the accuracy of the reduced solution improve as the 
	dimension of the reduced space increases?
	\item[{\sf \encircle{Q3}}]
	How can one judiciously \emph{a priori} select the snapshots and the weights in \eqref{eq:tPOD_algebraic_frequency}?
	\item[{\sf \encircle{Q4}}]
	How does the quality of the reduced basis improve as the number of snapshots increases?
	\item[{\sf \encircle{Q5}}]
	Why is only the real part of the snapshots required for the construction of $\bm{\Phi}^{\textrm{(rb)}}_R$?
\end{itemize}

The upcoming sections of this work aim at answering these questions.

\section{Analysis of the \rev{LT-MOR} Algorithm} 
\label{sec:Basis}
In this section, we provide a convergence analysis of the \rev{LT-MOR} method presented in Section~\ref{sec:FRB}.
This section is structured as follows. In Section~\ref{sec:laplace_hardy} we introduce the Laplace transform
in Banach spaces, and important properties to be used in this work.
In Section~\ref{sec:complexity_estimates} we provide an analysis of the dependence of Problem~\ref{pr:sdp} upon
the complex Laplace \rev{variable}. Subsequently, in Section~\ref{sec:well_posed_time_space} we state
the well-posedness of Problem~\ref{pbm:wave_equation} in the space-time Sobolev space defined in Section~\ref{sec:funct_space_W}
by using the tools introduced in Section~\ref{sec:laplace_hardy}, whereas in Section~\ref{sec:estimates_sd}
we provide estimates for the semi-discrete problem. We conclude this section 
by providing best approximation error estimates in Section~\ref{sec:best_app_error}.

\subsection{The Laplace Transform and Hardy spaces}
\label{sec:laplace_hardy}
We recall relevant properties of the Laplace transform in Banach spaces
that are used in the subsequent analysis.

\begin{proposition}[Operational Properties of the Laplace Transform]
\label{eq:operational_properties}
Let $X$ and $Y$ complex Banach spaces.
\begin{itemize}
	\item[(i)]\label{property:laplace_1}
	(Ref.~\refcite{arendt1987vector}, Corollary 1.6.2).
	Let $f \in L^1_{\normalfont\text{loc}}(\mathbb{R}_+;X)$
	and $\mathsf{T} \in \mathscr{L}(X,Y)$, and let
	$(\mathsf{T} \circ f)(t) = \mathsf{T}(f(t))$.
	Then, $\mathsf{T} \circ f \in L^1_{\normalfont\text{loc}}(\mathbb{R}_+;Y)$.
	If $\mathcal{L}\{f\}(s)$ exists, then $\mathcal{L}\{\mathsf{T} \circ f \}(s)$
	exists and equals 
	$$
		\mathcal{L}\{\mathsf{T} \circ f \}(s) = \mathsf{T}(\mathcal{L}\{f\}(s)).
	$$
	\item[(ii)]\label{property:laplace_3}
	(Ref.~\refcite{arendt1987vector}, Corollary 1.6.5).
	Let $f \in L^1_{\normalfont\text{loc}}(\mathbb{R}_+;X)$
	and let $F(t) = \int_{0}^{t}f(\tau)\normalfont\text{d}\tau$.
	If $\Re\{s\}>0$ and $\mathcal{L}\{f\}(s)$ exists, then 
	$\mathcal{L}\{F\}(s)$ exists and
	$$\mathcal{L}\{F\}(s) = \frac{\mathcal{L}\{f\}(s)}{s}. $$ 
	\item[(iii)]\label{property:laplace_2}
	(Ref.~\refcite{arendt1987vector}, Corollary 1.6.6).
	Let $f:\IR_+ \rightarrow X$ be absolutely continuous
	and differentiable a.e.
	If $\Re\{s\}>0$ and $\mathcal{L}\{\partial_tf\}(s)$ exists, then 
	$\mathcal{L}\{\partial_t f\}(s)$ exists and
	$$\mathcal{L}\{\partial_tf\}(s) = s\mathcal{L}\{f\}(s) - f(0).$$
\end{itemize}
\end{proposition}

Following Ref.~\refcite{RR97} (Chapter 4) and Ref.~\refcite{hille1996functional} (Section 6.4)
we introduce Hardy spaces \rev{in} $\Pi_\alpha$.
Throughout, let $V$ be a Banach space equipped with the norm
$\norm{\cdot}_V$, and for $\rev{\alpha \in \mathbb{R}}$ we set
\begin{equation}
	\Pi_\alpha
	\coloneqq
	\left\{
		z\in \mathbb{C}:
		\Re\{z\}
		>\alpha
	\right\}.
\end{equation}

\begin{definition}[Hardy Spaces, {Ref.~\refcite{hille1996functional}, Definition 6.4.1}]
\label{def:hardy_spaces}
Let $V$ be a \emph{complex} Banach space equipped with the norm $\norm{\cdot}_V$.
For $p\in[1,\infty)$ and \rev{$\alpha\in \mathbb{R}$}, we denote by
$\mathscr{H}^{p}_\alpha(V)$ the set of all $V$-valued 
functions $f:\Pi_\alpha\rightarrow V$ satisfying
the following properties:
\begin{itemize}
\item[(i)] 
The function $f:\Pi_\alpha\rightarrow V$ is holomorphic.
\item[(ii)]
It holds\footnote{We have scaled the variable $\tau$ by $2\pi$ so that in Theorem~\ref{eq:laplace_transform_bijective} 
ahead the isometry stated in \eqref{eq:norm_equiv} holds without any constant.}
\begin{align}
	\norm{f}_{\mathscr{H}^{p}_\alpha(V)}
	\coloneqq
	\sup_{\sigma>\alpha}
	\left(
		\,
		\int_{-\infty}^{+\infty}
		\norm{
			f(\sigma+\imath \tau)
		}^p_V
		\frac{
		\normalfont\text{d}
		\tau}{2\pi}
	\right)^{\frac{1}{p}}
	<
	\infty.
\end{align}
\item[(iii)]
For each $f \in \mathscr{H}^p(\Pi_\alpha;V)$
$\displaystyle\lim_{\sigma \rightarrow \alpha} f(\sigma+\imath \tau) = f(\alpha+\imath \tau)$ exists for almost 
all values of $\tau$, and 
\begin{equation}
	\int_{-\infty}^{+\infty}
	\norm{
		f(\alpha+\imath \tau)
	}^p_V
	\normalfont\text{d}
	\tau
	<\infty.
\end{equation}
\end{itemize}
Equipped with the norm $\norm{\cdot}_{\mathscr{H}^{p}_\alpha(V)}$
the space $\mathscr{H}^{p}_\alpha(V)$ is a Banach one.
\end{definition}

\begin{proposition}[{Ref.~\refcite{hille1996functional}, Theorem 6.4.3}]
\label{prop:properties_hardy}
Let $p\in [1,\infty)$ and \rev{$\alpha \in \mathbb{R}$}. 
\begin{itemize}
\item[(i)]
For each $f \in \mathscr{H}^p_\alpha(V)$ the function 
\begin{equation}
	T(\sigma,f)
	= 
	\int_{-\infty}^{+\infty}
	\norm{
		f(\sigma+\imath \tau)
	}^p_V
	\normalfont\text{d}
	\tau
\end{equation}
is a continuous monotone decreasing function of $\sigma$
for $\sigma \geq \alpha$. 
In particular, $T(\alpha,f) = \norm{f}^p_{\mathscr{H}^p_\alpha(V)}$
and $\displaystyle\lim_{\sigma \rightarrow \infty} T(\sigma,f)  = 0$.
\item[(ii)]
For each $f \in \mathscr{H}^p_\alpha(V)$
\begin{equation}
	\lim_{\sigma \rightarrow \alpha}
	\int_{-\infty}^{+\infty}
	\norm{
		f(\sigma+\imath \tau)
		-
		f(\alpha+\imath \tau)
	}^p_V
	\normalfont\text{d}
	\tau
	= 0.
\end{equation}
\end{itemize}
\end{proposition}

\begin{remark}\label{rmk:def_hardy_spaces_item_3}
Even though the precise definition of Hardy spaces involves 
verifying the three items described in Definition~\ref{def:hardy_spaces},
as discussed in Ref.~\refcite{hille1996functional}, Chapter 2, Section 6.4,
the last item can be proved to be redundant. Thus, when verifying
that a given function actually belongs to a Hardy space, we refrain from proving this last statement. 
\end{remark}


The following result is a Hilbert space-valued version of the 
Paley-Wiener representation theorem.

\begin{theorem}[Paley-Wiener Theorem, Ref.~{\refcite{RR97}, Section 4.8, Theorem E}]
\label{eq:laplace_transform_bijective}
Let $X$ be a Hilbert space and let \rev{$\alpha\in \mathbb{R}$}.
Then, the map 
$
	\mathcal{L}:
	L^2_\alpha(\IR_+;X) 
	\rightarrow 
	\mathscr{H}^2_\alpha(X)
$
is an isometric isomorphism, i.e.,
$$\mathcal{L} \in \mathscr{L}_{\normalfont\text{iso}}(L^2_\alpha(\IR_+;X), \mathscr{H}^2_\alpha(X)),$$
and for each $f \in L^2_\alpha(\IR_+;X)$
\begin{equation}\label{eq:norm_equiv}
	\norm{f}_{L^2_\alpha\left(\mathbb{R}_+;X\right)} 
	= 
	\norm{
		\mathcal{L}\{f\}
	}_{\mathscr{H}^2_\alpha\left(X\right)}.
\end{equation}
\end{theorem}

Equipped with these tools, \rev{we analyze} the
\rev{LT-MOR} algorithm introduced in Section~\ref{sec:FRB}.
Let $u \in L^2_\alpha(\IR_+;X)$ for some \rev{$\alpha \in \mathbb{R}$}.
We are interested in finding 
a finite dimensional subspace $X_R$ of $X$ of dimension $R\in \mathbb{N}$
such that
\begin{equation}\label{eq:minimization_problem}
	X_R
	=
	\argmin_{
	\substack{
		X_R \subset X
		\\
		\text{dim}(X_R)\leq R	
	}	
	}
	\norm{
		u
		-
		\mathsf{P}_{X_R} u
	}^2_{L^2_\alpha(\IR_+;X)},
\end{equation}
where $\mathsf{P}_{X_R}: X \rightarrow X_R$ denotes the orthogonal 
projection operator onto $X_R$.

We resort to~\eqref{eq:laplace_transform_bijective}
to obtain an equivalent expression for \eqref{eq:minimization_problem}
in terms of the Laplace transform of $u$.
Firstly, according to~\eqref{eq:laplace_transform_bijective}
one has that $\widehat{u} \coloneqq \mathcal{L}\{u\} \in \mathscr{H}^2_\alpha(X)$
as $u \in L^2_\alpha(\IR_+;X)$.
Next, recalling \eqref{eq:norm_equiv} in Theorem~\ref{eq:laplace_transform_bijective}
\begin{equation}
\begin{aligned}
	\norm{
		u
		-
		\mathsf{P}_{X_R} u
	}^2_{L^2_\alpha(\IR_+;X)}
	&
	=
	\norm{
		\widehat{u}
		-
		\mathsf{P}_{X_R}
		\widehat{u}
	}^2_{\mathscr{H}^2_\alpha(X)}
	\\
	&
	\stackrel{(\clubsuit)}{=}
	\sup_{\sigma>\alpha}
	\int_{-\infty}^{+\infty}
	\,
	\norm{
		\widehat{u}
		(\sigma + \imath \tau)
		-
		\left(
			\mathsf{P}_{X_R} 
			\widehat{u}(\sigma + \imath \tau)
		\right)
	}^2_{X}
	\frac{\normalfont\text{d} \tau}{2\pi}
	\\
	&
	\stackrel{(\spadesuit)}{=}
	\int_{-\infty}^{+\infty}
	\,
	\norm{
		\widehat{u}
		(\alpha + \imath \tau)
		-
		\mathsf{P}_{X_R} 
		\widehat{u}(\alpha + \imath \tau)
	}^2_{X}
	\frac{\normalfont\text{d} \tau}{2\pi}.
\end{aligned}
\end{equation}
In $(\clubsuit)$ we have used the definition
of the $\mathscr{H}^2_\alpha(X)$-norm and
item (i) in Proposition~\ref{eq:operational_properties}, whereas in $(\spadesuit)$
we have used Proposition~\ref{prop:properties_hardy}, item (i).
Therefore, one has
\begin{equation}\label{eq:minimization_problem_2}
\begin{aligned}
	X_R
	&
	=
	\argmin_{
	\substack{
		X_R \subset X
		\\
		\text{dim}(X_R)\leq R	
	}	
	}
	\norm{
		u
		-
		\mathsf{P}_{X_R} u
	}^2_{L^2_\alpha(\IR_+;X)}
	\\
	&
	=
	\argmin_{
	\substack{
		X_R \subset X
		\\
		\text{dim}(X_R)\leq R	
	}	
	}
	\norm{
		\widehat{u}
		-
		\mathsf{P}_{X_R} \widehat{u}
	}^2_{\mathscr{H}^2_\alpha(X)}
	\\
	&
	=
	\argmin_{
	\substack{
		X_R \subset X
		\\
		\text{dim}(X_R)\leq R	
	}	
	}
	\,
	\int_{-\infty}^{+\infty}
	\norm{
		\widehat{u}
		(\alpha + \imath \tau)
		-
		\mathsf{P}_{X_R} 
		\widehat{u}(\alpha + \imath \tau)
	}^2_{X}
	\frac{\normalfont\text{d} \tau}{2\pi},
\end{aligned}
\end{equation}
\rev{i.e.,} the minimization problem \eqref{eq:minimization_problem}
stated in the time domain can be solved in the Laplace domain using the Laplace transform
of $u \in L^2_\alpha(\IR_+;X)$.
Indeed, \eqref{eq:fPOD} corresponds to a numerical approximation
of the last integral in \eqref{eq:minimization_problem_2}
with quadrature points $\mathcal{P}_s$ and quadrature weights 
$\{\omega_1,\ldots,\omega_{M}\}$. We further elaborate on this
ahead in Section~\ref{sec:computation_rb}.
\subsection{Laplace \rev{Variable} Explicit Estimates}
\label{sec:complexity_estimates}
To properly analyze the \rev{LT-MOR} algorithm one needs to precisely
understand the dependence of $\widehat{u}(s)$ upon
the complex Laplace \rev{variable} $s\in \Pi_\alpha$,
where \rev{$\alpha \in \mathbb{R}$}.

Formally, one can notice that, for each
$s \in \Pi_\alpha$ and some $\alpha\geq0$, $\widehat{u}(s)$ is solution to
the following problem in strong form:
Find $\widehat{u}(s): \Omega \rightarrow \IC$ such that
\begin{equation}
	s
	\widehat{u}(s)
	+
	\mathcal{A}
	\widehat{u}(s)
	=
	\widehat{f}(s)
	+
	u_0
	\quad
	\text{in }
	\Omega,
\end{equation}
where $\mathcal{A}$ is as in \eqref{eq:scalar_wave_eq}, $\widehat{f}(s)$ corresponds to the Laplace transform of
$f: \IR_+\times \Omega \rightarrow \IR$, and equipped with homogeneous Dirichlet
boundary conditions.

This problem admits the following variational formulation.

\begin{problem}[Laplace Domain Continuous Variational Formulation]\label{prb:frequency_problem}
Let $f \in L^2_\alpha(\IR_+;L^2(\Omega;\mathbb{R}))$
for some \rev{$\alpha \in \mathbb{R}$}, and let $u_0 \in \rev{H^1_0(\Omega;\mathbb{R})}$.

For each $s\in \Pi_\alpha$ we seek $\widehat{u}(s) \in H^1_0(\Omega)$
satisfying
\begin{equation}\label{eq:complex_frequency_Laplace}
	\mathsf{b}(\widehat{u}(s),v;s)
	=
	\mathsf{g}(v;s),
	\quad
	\forall
	v\in H^1_0(\Omega), 
\end{equation}
where, for each $s \in \Pi_\alpha$, the sesquilinear
form $\mathsf{b}(\cdot,\cdot;s):  H^1_0(\Omega) \times  H^1_0(\Omega) 
\rightarrow \IC$ is defined as 
\begin{equation}
	\mathsf{b}(w,v;s)
	\coloneqq
	s
	\left(
		w
		,
		v
	\right)_{L^2(\Omega)}
	+
	\mathsf{a}
	\left(
		w
		,
		v
	\right),
	\quad
	\forall
	w,v
	\in 
	H^1_0(\Omega),
\end{equation}
whereas the antilinear form $\mathsf{g}(\cdot;s): H^1_0(\Omega)\rightarrow \IC$
is defined as
\begin{equation}\label{eq:f_rhs}
	\mathsf{g}(w;s)
	\coloneqq
	\dotp{
		\widehat{f}(s)
	}{w}_{L^2(\Omega)}
	+
	\dotp{
		u_0
	}{w}_{L^2(\Omega)}	,
	\quad
	\forall
	w
	\in 
	H^1_0(\Omega),
\end{equation}
where $\widehat{f} \in \mathscr{H}^2_\alpha(L^2(\Omega))$
is the Laplace transform of $f \in L^2_\alpha(\IR_+;L^2(\Omega;\mathbb{R}))$
(cf.~(\ref{eq:laplace_transform_bijective})).
\end{problem}

Observe that Problem~\ref{prb:frequency_problem} corresponds 
to the continuous counterpart of Problem~\ref{pr:sdp}, as the former
is set in $H^1_0(\Omega)$ as opposed to the latter, which is set in the
finite dimensional subspace $\mathbb{V}^{\mathbb{C}}_h$ of $H^1_0(\Omega)$.

We prove the following auxiliary result. 
\begin{lemma}\label{lmm:error_bound_L2}
Let \rev{$\alpha>-\frac{\underline{c}_{\bm{A}}}{C_P(\Omega)}$} and $q \in H^{-1}(\Omega)$.
Then, for each $s\in \Pi_\alpha$ there exists 
a unique $p(s) \in H^1_0(\Omega)$ solution to
\begin{equation}\label{eq:expand_p_s}
	\mathsf{b}(p(s),v;s)
	=
	\dual{q}{v}_{H^{-1}(\Omega) \times H^{1}_0(\Omega)},
	\quad
	\forall
	v \in H^1_0(\Omega).
\end{equation}
In addition, for each $s\in \Pi_\alpha$ it holds
\begin{equation}
	\norm{p(s)}_{H^{1}_0(\Omega)}
	\leq
	\frac{1}{\underline{c}_{\bm{A}}}
	\norm{
		q
	}_{H^{-1}(\Omega)}
	\;
	\text{and}
	\;
	\norm{p(s)}_{H^{-1}(\Omega)}
	\leq
	\frac{1}{\snorm{s}}
	\left(
	1
	+
    \frac{
		\rev{\overline{c}_{\bm{A}}}
	}{
		\rev{\gamma(\alpha)}
	}
	\right)
	\norm{
		q
	}_{H^{-1}(\Omega)}
\end{equation}
\rev{with
$
    \gamma(\alpha)
    \coloneqq
    \rev{\underline{c}_{\bm{A}}}
    +
    C_P(\Omega)
    \min
    \{
        \alpha
        ,
       0
    \}.
$
}
\end{lemma}

\begin{proof}
For each $s\in \Pi_\alpha$ \rev{and assuming $\alpha\geq 0$}
\begin{equation}\label{eq:ellipticity_b}
\begin{aligned}
	\Re\{
		\mathsf{b}(w,w;s)
	\}
	&
	=
	\Re
	\left\{
		s
		\left(
			w
			,
			w
		\right)_{L^2(\Omega)}
		+
		\mathsf{a}
		\left(
			w
			,
			w
		\right)
	\right\}
	\\
	&
	\geq
	\alpha
	\norm{
		w
	}^2_{L^2(\Omega)}
	+
	\rev{\underline{c}_{\bm{A}}}
	\norm{\nabla w}^2_{L^2(\Omega)},
	\\
	&
	\geq
	\rev{\underline{c}_{\bm{A}}}
	\norm{w}^2_{H^1_0(\Omega)},
	\quad
	\forall w \in H^1_0(\Omega),
\end{aligned}
\end{equation}
\rev{%
whereas for $\alpha \in (-\frac{\underline{c}_{\bm{A}}}{C_P(\Omega)},0)$
\begin{equation}
    \Re\{
	\mathsf{b}(w,w;s)
    \}
    \geq
    \left(
        \rev{\underline{c}_{\bm{A}}}
        +
        \alpha
        C_P(\Omega)
    \right)
    \norm{w}^2_{H^1_0(\Omega)},
    \quad
    \forall w \in H^1_0(\Omega),
\end{equation}
hence for any $\alpha>-\frac{\underline{c}_{\bm{A}}}{C_P(\Omega)}$
\begin{equation}
    \Re\{
	\mathsf{b}(w,w;s)
    \}
    \geq
    \left(
        \rev{\underline{c}_{\bm{A}}}
        +
        C_P(\Omega)
        \min
        \{
            \alpha
            ,
            0
        \}
    \right)
    \norm{w}^2_{H^1_0(\Omega)},
    \quad
    \forall w \in H^1_0(\Omega),
\end{equation}
}%
In addition, for each $s\in \Pi_\alpha$ the sesquilinear form
$\mathsf{b}(\cdot,\cdot;s):  H^1_0(\Omega) \times  H^1_0(\Omega) 
\rightarrow \IC$ is linear and continuous, i.e.,~$\forall w,v \in  H^1_0(\Omega)$
\begin{equation}
\begin{aligned}
	\snorm{
		\mathsf{b}(w,v;s)
	}
	\leq
	&
	\snorm{s}
	\norm{
		w
	}_{L^2(\Omega)}
	\norm{
		v
	}_{L^2(\Omega)}
	+
	\rev{\overline{c}_{\bm{A}}}
	\norm{
		\nabla
		w
	}_{L^2(\Omega)}
	\norm{
		\nabla
		v
	}_{L^2(\Omega)}
	\\
	\leq
	&
	\left(
		\snorm{s}
		C^2_P(\Omega)
		+
		\rev{\overline{c}_{\bm{A}}}
	\right)
	\norm{
		\nabla
		w
	}_{L^2(\Omega)}
	\norm{
		\nabla
		v
	}_{L^2(\Omega)},
\end{aligned}
\end{equation}
where $C_P(\Omega)>0$ corresponds to Poincar\'e's constant. 

Consequently, for each $s\in \Pi_\alpha$ there exists a unique
$p(s) \in H^1_0(\Omega)$ solution to \eqref{eq:expand_p_s}
satisfying
\begin{equation}\label{eq:a_priori_bound}
	\norm{
		p(s)
	}_{H^1_0(\Omega)}
	\leq
	\frac{
	\norm{
		q
	}_{H^{-1}(\Omega)}
    }{\rev{\gamma(\alpha)}}.
\end{equation}
Let us calculate
\begin{equation}
\begin{aligned}
	\norm{{p}(s)}_{H^{-1}(\Omega)}
	&
	=
	\sup_{0\neq v\in H^1_0(\Omega)}
	\frac{
		\snorm{
		\dual{
			{p}(s)
		}{
			v
		}_{H^{-1}(\Omega) \times H^{1}_0(\Omega)}
		}
	}{
		\norm{v}_{H^1_0(\Omega)}
	}
	\\
	&
	=
	\frac{1}{\snorm{s}}
	\sup_{0\neq v\in H^1_0(\Omega)}
	\frac{
		\snorm{
		\dual{q}{v}_{H^{-1}(\Omega) \times H^{1}_0(\Omega)}
		-
		\mathsf{a}
		\left(
			{p}(s)
			,
			v
		\right)
		}
	}{
		\norm{v}_{H^1_0(\Omega)}
	}
	\\
	&
	\leq
	\frac{1}{\snorm{s}}
	\left(
	\norm{
		q
	}_{H^{-1}(\Omega)}
	+
	\rev{\overline{c}_{\bm{A}}}
	\norm{
		{p}(s)
	}_{H^{1}_0(\Omega)}
	\right)
	\\
	&
	\leq
	\frac{1}{\snorm{s}}
	\left(
	1
	+
	\frac{
		\rev{\overline{c}_{\bm{A}}}
	}{
		\rev{\gamma(\alpha)}
	}
	\right)
	\norm{
		q
	}_{H^{-1}(\Omega)},
\end{aligned}
\end{equation}
therefore concluding the proof.
\end{proof}

It follows from Lemma~\ref{lmm:error_bound_L2} that for each $s\in \Pi_\alpha$
there exists a unique $\widehat{u}\rev{(s)} \in H^1_0(\Omega)$ solution to Problem~\ref{prb:frequency_problem}
satisfying
\begin{equation}\label{eq:a_priori_bound}
	\norm{
		\widehat{u}(s)
	}_{H^1_0(\Omega)}
	\leq
	\frac{C_P(\Omega)}{\rev{\gamma(\alpha)}}
	\left(
		\norm{
			\widehat{f}(s)
		}_{L^2(\Omega)}
		+
		\norm{
			u_0
		}_{L^{2}(\Omega)}
	\right).
\end{equation}

However, this result is not yet satisfactory as the
bound \eqref{eq:a_priori_bound} \rev{cannot} be integrated along any infinite
line in the complex plane that is parallel to the imaginary axis and with real part 
equal to $\alpha$. This will be needed ahead in the proof of \eqref{eq:lemma_Hardy_space_uh}.
The next lemma addresses this issue.

\begin{lemma}\label{lmm:laplace_transform_u_bounds}
Let $f \in L^2_\alpha(\IR_+;L^2(\Omega;\mathbb{R}))$
for some $\rev{\alpha>-\frac{\underline{c}_{\bm{A}}}{C_P(\Omega)}}$, and let $u_0 \in H^1_0(\Omega;\mathbb{R})$.
Then, for each $s\in \Pi_\alpha$, there exists a unique
$\widehat{u}(s) \in H^1_0(\Omega)$ solution to \eqref{prb:frequency_problem}
satisfying
\begin{equation}
	\norm{
		\widehat{u}(s)
	}_{H^1_0(\Omega)}
	\leq
	\frac{C_P(\Omega)}{\rev{\gamma(\alpha)}}
	\norm{
		\widehat{f}(s)
	}_{L^2(\Omega)}
	+
	\frac{1}{\snorm{s}}
	\left(
		1
		+
		\rev{\frac{\overline{c}_{\bm{A}}}{\gamma(\alpha)}}
	\right)
	\norm{
		u_0
	}_{H^1_0(\Omega)},
\end{equation}
\end{lemma}

\begin{proof}
Existence and uniqueness follows from Lemma~\ref{lmm:error_bound_L2}.
Define
\begin{equation}\label{eq:decomposition_h_u_s}
	\widehat{w}(s)
	\coloneqq
	\widehat{u}(s)
	-
	\frac{1}{s}
	u_0
	\in
	H^1_0(\Omega).
\end{equation}

Observe that $\widehat{w}(s) \in H^1_0(\Omega)$ satisfies the following
variational problem
\begin{equation}
\begin{aligned}
	\mathsf{b}(\widehat{w}(s),v;s)
	=
	&
	s
	\left(
		\widehat{w}(s)
		,
		v
	\right)_{L^2(\Omega)}
	+
	\mathsf{a}
	\left(
		\widehat{w}(s)
		,
		v
	\right),
	\\
	=
	&
	s
	\left(
		\widehat{u}(s)
		-
		\frac{1}{s}
		u_0
		,
		v
	\right)_{L^2(\Omega)}
	+
	\mathsf{a}
	\left(
		\widehat{u}(s)
		-
		\frac{1}{s}
		u_0
		,
		v
	\right)
	\\
	=
	&
	s
	\left(
		\widehat{u}(s)
		,
		v
	\right)_{L^2(\Omega)}
	+
	\mathsf{a}
	\left(
		\widehat{u}(s)
		,
		v
	\right)
	-
	\left(
		u_0
		,
		v
	\right)_{L^2(\Omega)}
	-
	\frac{1}{s}
	\mathsf{a}
	\left(
		u_0
		,
		v
	\right)
	\\
	=
	&
	\dotp{
		\widehat{f}(s)
	}{v}_{L^2(\Omega)}
	+
	\dotp{
		u_0
	}{v}_{L^2(\Omega)}
	\\
	&
	-
	\left(
		u_0
		,
		v
	\right)_{L^2(\Omega)}
	-
	\frac{1}{s}
	\mathsf{a}
	\left(
		u_0
		,
		v
	\right),
	\quad
	\forall
	v\in H^1_0(\Omega).
\end{aligned}
\end{equation}
Therefore, for each $s\in \Pi_\alpha$,
$\widehat{w}(s) \in H^1_0(\Omega)$ is solution to the following
variational problem
\begin{equation}\label{eq:equation_w_s}
\begin{aligned}
	\mathsf{b}(\widehat{w}(s),v;s)
	=
	\dotp{
		\widehat{f}(s)
	}{v}_{L^2(\Omega)}
	-
	\frac{1}{s}
	\mathsf{a}
	\left(
		u_0
		,
		v
	\right),
	\quad
	\forall
	v\in H^1_0(\Omega).
\end{aligned}
\end{equation}
Recalling Lemma~\ref{lmm:error_bound_L2}
\begin{equation}\label{eq:bound_hw_s}
\begin{aligned}
	\norm{\widehat{w}(s)}_{H^1_0(\Omega)}
	&
	\leq
	\frac{1}{\rev{\gamma(\alpha)}}
	\left(
		C_P(\Omega)
		\norm{
			\widehat{f}(s)
		}_{L^2(\Omega)}
		+
		\frac{\rev{\overline{c}_{\bm{A}}}}{\snorm{s}}
		\norm{
			u_0
		}_{H^{1}_0(\Omega)}
	\right).
\end{aligned}
\end{equation}
Therefore, for each $s \in \Pi_\alpha$, one has
\begin{equation}
\begin{aligned}
	\norm{\widehat{u}(s)}_{H^1_0(\Omega)}
	&
	\leq
	\frac{1}{\rev{\gamma(\alpha)}}
	\left(
		C_P(\Omega)
		\norm{
			\widehat{f}(s)
		}_{L^2(\Omega)}
		+
		\frac{\rev{\overline{c}_{\bm{A}}}}{\snorm{s}}
		\norm{
			u_0
		}_{H^1_0(\Omega)}
	\right)
	+
	\frac{1}{\snorm{s}}
	\norm{
		u_0
	}_{H^1_0(\Omega)}
	\\
	&
	\leq
	\frac{C_P(\Omega)}{\rev{\gamma(\alpha)}}
	\norm{
		\widehat{f}(s)
	}_{L^2(\Omega)}
	+
	\frac{1}{\snorm{s}}
	\left(
		1
		+
		\rev{\frac{\overline{c}_{\bm{A}}}{\gamma(\alpha)}}
	\right)
	\norm{
		u_0
	}_{H^1_0(\Omega)},
\end{aligned}
\end{equation}
as claimed.
\end{proof}

We show that Problem~\ref{prb:frequency_problem} is not only well-posed 
for each $s\in \Pi_\alpha$, but also in $\mathscr{H}^2_\alpha(H^1_0(\Omega))$.

\begin{lemma}\label{eq:lemma_Hardy_space_uh}
Let $f \in L^2_\alpha(\IR_+;L^2(\Omega;\mathbb{R}))$
for some $\alpha>0$, and let $u_0 \in H^1_0(\Omega;\mathbb{R})$.
Then, $\widehat{u} \in \mathscr{H}^2_\alpha(H^1_0(\Omega))$,
where $\widehat{u}(s) \in H^1_0(\Omega)$ is the solution to Problem~\ref{prb:frequency_problem}
for each $s \in \Pi_{\alpha}$, and
\begin{equation}\label{eq:a_priori_1}
	\norm{
		\widehat{u}
	}_{\mathscr{H}^2_\alpha(H^1_0(\Omega))}
	\rev{\lesssim}
	\frac{C_P(\Omega)}{\rev{\gamma(\alpha)}}
	\norm{
		\widehat{f}
	}_{\mathscr{H}^2_\alpha(L^2(\Omega))}
	+
	\frac{1}{\rev{\sqrt{\alpha}}}
	\left(
		1
		+
		\rev{\frac{\overline{c}_{\bm{A}}}{\gamma(\alpha)}}
	\right)
	\norm{
		u_0
	}_{H^1_0(\Omega)}.
\end{equation}

\end{lemma}

\begin{proof}
We proceed to show that $\Pi_\alpha \ni s \mapsto \widehat{u}(s) \in H^1_0(\Omega)$
verifies item (i) and (ii) in Definition~\ref{def:hardy_spaces} (cf. Remark\ref{rmk:def_hardy_spaces_item_3}).

{\bf Item (i).}
The map $\Pi_\alpha \ni s \mapsto \mathsf{b}(\cdot,\cdot;s) \in  \mathscr{L}_{\rev{\text{sesq}}}
\left(H^1_0(\Omega)\times H^1_0(\Omega); \IC\right)$ is holomorphic as
it depends linearly on  $s \in \Pi_\alpha$.
Since $f\in L^2_\alpha(\IR_+;L^2(\Omega;\mathbb{R}))$ Theorem~\eqref{eq:laplace_transform_bijective}
guarantees that $\widehat{f} \in  \mathscr{H}^2_\alpha(L^2(\Omega))$. 
Therefore, the map $\Pi_\alpha \ni s \mapsto \mathsf{g}(\cdot;s) \in L^2(\Omega)$
is holomorphic, with $\mathsf{g}(\cdot;s)$ as in \eqref{eq:f_rhs}. 
Also, for each $s\in \Pi_\alpha$ the sesquilinear form
$\mathsf{b}(\cdot,\cdot;s) \in  \mathscr{L}_{\rev{\text{sesq}}} \left( H^1_0(\Omega)\times H^1_0(\Omega); \IC\right)$ 
has a bounded inverse as stated in Lemma~\ref{lmm:error_bound_L2}.
The inversion of bounded linear operators with 
bounded inverse is itself a holomorphic map.
We may conclude that the map $\Pi_\alpha \ni s \mapsto \widehat{u}(s) \in H^1_0(\Omega)$
is holomorphic, thus verifying item (i) in Definition~\ref{def:hardy_spaces}.

{\bf Item (ii).} Recalling the definition of the $\mathscr{H}^2_\alpha(H^1_0(\Omega))$-norm
and Proposition~\ref{prop:properties_hardy}
\begin{equation}\label{eq:calculation_hardy}
\begin{aligned}
	\norm{
		\widehat{u}
	}^2_{\mathscr{H}^2_\alpha(H^1_0(\Omega))}
	=
	&
	\int_{-\infty}^{+\infty}
	\norm{
		\widehat{u}
		(\alpha + \imath \tau)
	}^2_{H^1_0(\Omega)}
	\,
	\frac{\text{d}
	\tau}{2\pi}
	\\
	\leq
	&
	2
	\frac{C^2_P(\Omega)}{\underline{c}^2_{\bm{A}}}
	\int_{-\infty}^{+\infty}
	\norm{
		\widehat{f}(\alpha + \imath \tau)
	}^2_{L^2(\Omega)}
	\,
	\frac{\text{d}
	\tau}{2\pi}
	\\
	&
	+
	\frac{1}{\pi}
	\left(
		1
		+
		\rev{\frac{\overline{c}_{\bm{A}}}{\gamma(\alpha)}}
	\right)^2
	\norm{
		u_0
	}^2_{H^1_0(\Omega)}
	\int_{-\infty}^{+\infty}
	\frac{	\text{d}\tau}{\snorm{\alpha + \imath \tau}^2}.
\end{aligned}
\end{equation}
Recalling that
$
	\int_{-\infty}^{+\infty}
	\frac{	\text{d}
	\tau}{\snorm{\alpha + \imath \tau}^2}
	=
	\int_{-\infty}^{+\infty}
	\frac{	\text{d}
	\tau}{{\alpha^2+\tau^2}}
	=
	\frac{\pi}{\alpha}
$
we obtain
\begin{equation}
	\norm{
		\widehat{u}
	}^2_{\mathscr{H}^2_\alpha(H^1_0(\Omega))}
	\leq
	2
	\frac{C^2_P(\Omega)}{\rev{\gamma^2(\alpha)}}
	\norm{
		\widehat{f}
	}^2_{\mathscr{H}^2_\alpha(L^2(\Omega))}
	+
	\frac{1}{\alpha}
	\left(
		1
		+
		\rev{\frac{\overline{c}_{\bm{A}}}{\gamma(\alpha)}}
	\right)^2
	\norm{
		u_0
	}^2_{H^1_0(\Omega)},
\end{equation}
thus verifying item (ii) in Definition~\ref{def:hardy_spaces}, and proving \eqref{eq:a_priori_1}.
\end{proof}

\subsection{Well-posedness in Sobolev Spaces involving time}
\label{sec:well_posed_time_space}
We establish well-posedness of the linear,
second-order parabolic problem, i.e.,
Problem~\ref{pbm:wave_equation}. Standard results make use of the
so-called Faedo-Galerkin approach, see e.g. \cite{evans2022partial}.
However, these are established over finite time intervals.
For the sake of completeness, here we provide a different proof that uses
the tools introduced in Section~\ref{sec:laplace_hardy}. 
A complete proof is included as an appendix. 

\begin{theorem}\label{eq:well_posedness}
Let $f  \in L^2_{\alpha}(\IR_+;L^2(\Omega;\mathbb{R}))$ for some $\alpha_0>0$,
and $u_0 \in  H^1_0(\Omega;\mathbb{R})$.
Then, there exists a unique 
${u} \in \mathcal{W}_\alpha\left(\IR_+;H^{1}_0(\Omega;\mathbb{R})\right)$
solution to Problem~\ref{pbm:wave_equation} satisfying
\begin{equation}\label{eq:a_priori_estimate}
\begin{aligned}
    \rev{
	\norm{u}_{\mathcal{W}_\alpha\left(\IR_+;H^{1}_0(\Omega;\mathbb{R})\right)}
	\lesssim
	\norm{f}_{L^2_\alpha\left(\IR_+;L^2(\Omega)\right)}
    +
	\norm{
		u_0
	}_{H^{1}_0(\Omega)}.
    }
\end{aligned}
\end{equation}
\rev{with a hidden constant depending only on $\alpha,\underline{c}_{\bm{A}},\overline{c}_{\bm{A}}$, 
and $C_P(\Omega)$.}
\end{theorem}

\begin{proof}
The proof of this result may be found in Section~\ref{sec:existence},
which uses the results of Section~\ref{sec:complexity_estimates}.
\end{proof}

\subsection{Estimates for the Semi-Discrete Problem}\label{sec:estimates_sd}
In this section, we extend the results obtained in Section~\ref{sec:complexity_estimates}
to the solution of Problem~\ref{pr:sdp}.

\begin{lemma}\label{lmm:laplace_transform_u_bounds_discrete}
Let $f \in L^2_\alpha(\IR_+;L^2(\Omega;\mathbb{R}))$
for some $\alpha>0$, and let $u_0 \in H^1_0(\Omega;\mathbb{R})$.
Then, for each $s\in \Pi_\alpha$, there exists a unique
$\widehat{u}_h(s) \in \IV^{\mathbb{C}}_h$ solution to Problem~\ref{pr:sdp}
satisfying
\begin{equation}
	\norm{
		\widehat{u}_h(s)
	}_{H^1_0(\Omega)}
	\rev{\lesssim}
	\frac{C_P(\Omega)}{\rev{\gamma(\alpha)}}
	\norm{
		\widehat{f}(s)
	}_{L^2(\Omega)}
	+
	\frac{1}{\snorm{s}}
	\left(
		1
		+
		\rev{\frac{\overline{c}_{\bm{A}}}{\gamma(\alpha)}}
	\right)
	\norm{
		u_{0,h}
	}_{H^{1}_0(\Omega)},
\end{equation}
where $u_{0,h} = \rev{\mathsf{P}_h} u_0 \in  \mathbb{V}_h$.
\end{lemma}

\begin{proof}
Exactly as in the proof of Lemma~\ref{lmm:laplace_transform_u_bounds}.
\end{proof}

\begin{lemma}\label{eq:lemma_Hardy_space_uh_discrete}
Let $f \in L^2_\alpha(\IR_+;L^2(\Omega;\mathbb{R}))$
for some $\alpha>0$, and let $u_0 \in H^1_0(\Omega;\mathbb{R})$.
Then, $\widehat{u}_h \in \mathscr{H}^2_\alpha(\IV^{\mathbb{C}}_h)$,
where $\widehat{u}_h(s) \in \IV^{\mathbb{C}}_h$ \rev{is} solution to Problem~\ref{pr:sdp}
for each $s \in \Pi_{\alpha}$, and
\begin{equation}\label{eq:a_priori_1}
\begin{aligned}
	\norm{
		\widehat{u}_h
	}_{\mathscr{H}^2_\alpha(H^1_0(\Omega))}
	\lesssim
    &
	\frac{C_P(\Omega)}{\rev{\gamma(\alpha)}}
	\norm{
		\widehat{f}
	}_{\mathscr{H}^2_\alpha(L^2(\Omega))}
    +
	\frac{1}{\rev{\sqrt{\alpha}}}
	\left(
		1
		+
		\rev{\frac{\overline{c}_{\bm{A}}}{\gamma(\alpha)}}
	\right)
	\norm{
		u_{0,h}
	}_{H^1_0(\Omega)}.
\end{aligned}
\end{equation}
where $u_{0,h} = \mathsf{P}_h u_0 \in  \mathbb{V}_h$.
\end{lemma}

\begin{proof}
Exactly as in the proof of Theorem~\ref{eq:lemma_Hardy_space_uh}.
\end{proof}

\begin{remark}[A-priori Estimates of the Semi-discrete Full-order and Reduced Problems]
\label{rmk:existence_semi_discrete_full}
Similarly to Theorem~\ref{eq:well_posedness}, one can show that for any $R \in \IN$,
any discretization parameter $h>0$, and for
\begin{equation}	
	u_h
	\in
	\mathcal{W}_\alpha\left(\IR_+;\mathbb{V}_h\right)
	\quad
	\text{and}
	\quad
	u^{\normalfont\text{(rb)}}_R 
	\in 
	\mathcal{W}_\alpha\left(\IR_+;\mathbb{V}^{\normalfont\text{(rb)}}_R\right)
\end{equation}
\rev{solutions to Problem~\ref{prbm:semi_discrete_problem}
and Problem~\ref{pr:sdpr}, respectively,
\emph{a priori} estimates as in \eqref{eq:a_priori_estimate} may be stated as well.}
Indeed, it holds
\begin{equation}
\begin{aligned}
	\norm{u_h}_{\mathcal{W}_\alpha\left(\IR_+;H^{1}_0(\Omega;\mathbb{R})\right)}
	\lesssim
	\norm{f}_{L^2_\alpha\left(\IR_+;L^2(\Omega)\right)}
    	+
	\rev{\norm{
		u_{0}
	}_{H^{1}_0(\Omega)}},
\end{aligned}
\end{equation}
\rev{with a hidden constant depending on $\alpha,\underline{c}_{\bm{A}},\overline{c}_{\bm{A}}$, 
and $C_P(\Omega)$.}
An equivalent bound holds for $u^{\normalfont\text{(rb)}}_R $
with the corresponding initial condition \rev{projected in the reduced space}.
\end{remark}

\subsection{\rev{Low-Rank Approximation}}
\label{sec:best_app_error}
For $\eta\rev{>0}$, define 
\begin{equation}
	\mathcal{D}_{\eta}  
	\coloneqq 
	\left\{
		z \in \IC: \snorm{z}<\eta
	\right\},
\end{equation}
and we set \rev{$\mathcal{D} \coloneqq \mathcal{D}_1$. }

\rev{For $\alpha \in \mathbb{R}$ and $\beta >0$,
consider the following M\"{o}bius transform and its inverse
\begin{equation}\label{eq:mobius}
	\mathcal{M}:\Pi_\alpha \rightarrow \mathcal{D}:s \mapsto \frac{s -\alpha -\beta}{s - \alpha + \beta}
	\quad
	\text{and}
	\quad
	\mathcal{M}^{-1}:\mathcal{D}\rightarrow \Pi_\alpha: z \mapsto 
	\alpha
	-
	\beta
	\frac{
		z+1
	}{
		z-1
	},
\end{equation}
respectively.
The latter maps the interior of the centered circle of radius $\eta>0$
to the exterior of the circle of center and radius
\begin{equation}\label{eq:center_radius_Mobius}
    c_{\eta,\alpha,\beta}
    =
    \alpha
    +
    \beta
    \frac{1+\eta^2}{1-\eta^2}
    \quad
    \text{and}
    \quad
    \rho_{\eta,\alpha,\beta}
    = 
    \frac{2\beta\eta}{\snorm{1-\eta^2}},
\end{equation}
and $\partial \mathcal{D}$ to the line $\Re\{z\} = \alpha$, 
see, e.g.,~see Ref.~\refcite[Lemma 2.2]{giunta1989more}.
In the following, we denote by $\mathcal{C}_{\eta,\alpha,\beta}$
the circle of center and radius $ c_{\eta}$ and $ \rho_{\eta}$, respectively, and recall
as well that the circles $\mathcal{C}_{\eta,\alpha,\beta}$ and $\mathcal{C}_{1/\eta,\alpha,\beta}$
are mirror images of each other with respect to the vertical line $\Re\{z\} = \alpha$.
}

We are interested in the \rev{low-rank} approximation of
the solution to Problem~\ref{prbm:semi_discrete_problem}.
To this end, we resort to Hardy spaces of analytic functions
and work under the assumptions stated below. 

\begin{assumption}[Data Regularity]\label{assump:data_f}
In the following, in addition to $f \in L_{\alpha_0}(\IR_+,L^2(\Omega;\mathbb{R}))$ for some
$\alpha_0>0$, we assume the following:
\begin{itemize}
	\item[(i)] $\partial_t f \in L_{\alpha_0}(\IR_+,L^2(\Omega;\mathbb{R}))$.
	\item[(ii)] There exists $C_f>0$ and an open set $\mathcal{O} \subset \mathbb{C}\backslash \overline{\Pi_{\alpha_0}}$,
	such that $\mathcal{L}\{\partial_t f\}(s)$ admits a holomorphic extension to 
	$\mathcal{O}^c$, i.e., the complement of $\mathcal{O}$.
	\item[(iii)] $u_0 \in H^1_0(\Omega) \cap H^2(\Omega)$ and $\bm{A} \in \mathscr{C}^{1}(\overline{\Omega};\mathbb{R}^{d \times d})$.
\end{itemize}
\end{assumption}

Consider the eigenvalue problem of finding the finitely many eigenpairs
$\{\left(\zeta_{h,i},\lambda_{h,i} \right)\}_{i=1}^{N_h} \subset \mathbb{V}_h \times \mathbb{C}$
with $\norm{\zeta_{h,i}}_{L^2(\Omega)}=1$ such that
\begin{equation}
    \mathsf{a}
    \dotp{\zeta_{h,i}}{v_h}
    =
    \lambda_{h,i}
    \dotp{\zeta_{h,i}}{v_h}_{L^2(\Omega)},
    \quad
    \forall v_h \in \mathbb{V}_h.
\end{equation}
In the following, we assume $\lambda_{h,1}\leq \dots \leq \lambda_{h,N_h}$.
Recall that $\norm{\zeta_{h,i}}_{\mathbb{V}_h} = \norm{\zeta_{h,i}}^{-\frac{1}{2}}_{\mathbb{V}_h} = \lambda^{-\frac{1}{2}}_{h,i} $
and that for any $v_h \in \mathbb{V}_h$ one has that
\begin{equation}
	\norm{v_h}^2_{H^1_0(\Omega)}
	=
	\sum_{i=1}^{N_h}
	\snorm{(v_h,\zeta_{h,i})_{L^2(\Omega)} }^2 \lambda_{h,i}.
\end{equation}

\rev{
\begin{lemma}\label{eq:approximation_Hardy_spaces}
Let Assumption~\ref{assump:data_f} be satisfied for some $\rev{\alpha_0 > 0}$ and
let $u \in \mathcal{W}_{\rev{\alpha_0}}(\IR_+;\mathbb{V}_h)$
be the solution to Problem~\ref{prbm:semi_discrete_problem}.

Then, for any $\alpha>\alpha_0$ and any $\beta > 0$,
there exists $\eta_{\alpha,\beta}>1$
such that for $\eta \in (1,\eta_{\alpha,\beta})$ and
$R \in \{1,\dots,N_h\}$ it holds that
\begin{equation}
\begin{aligned}
	\inf_{
	\substack{
		\mathbb{V}_R \subset \mathbb{V}_h
		\\
		\normalfont\text{dim}(\mathbb{V}_R)\leq R	
	}
	}	
	\norm{
		u_h
		-
		\mathsf{P}_{\mathbb{V}_R}
		u_h
	}_{L^2_\alpha(\IR_+;H^1_0(\Omega))}
	\lesssim
	\frac{\eta^2}{(\eta-1)\sqrt{\alpha \Lambda}}
	&
	\left(
	\sup_{ s \in \partial \mathcal{C}_{\eta,\alpha, \beta}}
	\norm{
		\Pi_h (\mathcal{L}\{\partial_t f\}(s))
	}_{H^1_0(\Omega)}
	\right.
	\\
	&
	+
	\norm{\Pi_h (f(0))}_{H^1_0(\Omega)}
	\\
	&
	\left.
	+
	\norm{\Pi_h ( \nabla \cdot( \bm{A} \nabla u_{0}))}_{H^1_0 (\Omega)}
	\right)
	\eta^{-R}.
\end{aligned}
\end{equation}
where $\mathsf{P}_{\mathbb{V}_R}: \mathbb{V}_h \rightarrow \mathbb{V}_R$
denotes the orthogonal projection onto $\mathbb{V}_R$, $\Pi_h: L^2(\Omega)
\rightarrow \mathbb{V}_h$ signifies the
$L^2(\Omega)$ projection onto $\mathbb{V}_h$, and
\begin{equation}\label{eq:lambda}
	\Lambda
	\coloneqq
	\min
	\left\{
	\left(
		\alpha+\lambda_{h,N_h}- \beta\frac{\eta - 1}{ \eta+1}
	\right)^2
	,
	\left(
		\alpha+\lambda_{h,1}- \beta\frac{\eta + 1}{ \eta-1}
	\right)^2
	\right\}.
\end{equation}
\end{lemma}

\begin{proof}
Let us consider $w_h \coloneqq u_h - u_{0,h}$, which is solution to
\begin{equation}\label{eq:sol_w_h_exp}
\begin{aligned}
	\left(\partial_t w_h(t), v_h\right)_{L^2(\Omega)}
	+
	\mathsf{a}\dotp{w_h}{v_h}
	=
	&
	\dotp{
		f(t)
	}{v_h}_{L^2(\Omega)}
	-
	\mathsf{a}
	\left(
		u_{0,h}
		,
		v_h
	\right)
	\\
	=
	&
	\dotp{
		{f}(t)
	}{v_h}_{L^2(\Omega)}
	+
	\left(
		\nabla \cdot( \bm{A} \nabla u_{0})
		,
		v_h
	\right)_{L^2(\Omega)},
\end{aligned}
\end{equation}
for all $v_h\in \mathbb{V}_h$.
%
For $v_h \in L^2_{\alpha}(\IR_+;\mathbb{V}_h)$ to be specified one can conclude
similarly as in \eqref{eq:minimization_problem_2} that 
\begin{equation}
\begin{aligned}
	\norm{
		u_h - v_h
	}^2_{L^2_\alpha(\IR_+;H^1_0(\Omega))}
	=
	\int_{-\infty}^{+\infty}
	\,
	\norm{
		\widehat{u}_h(\alpha + \imath \tau) - \widehat{v}_h(\alpha + \imath \tau)
	}^2_{H^1_0(\Omega)}
	\frac{\text{d}\tau}{2\pi},
\end{aligned}
\end{equation}
where, as is customary, the hat indicates
the application of the Laplace transform to the underlying function.
In particular, if we set $\widehat{v}_h(s) = \frac{u_{0,h}}{s}  - \widehat{z}_h(s) $ 
we obtain
\begin{equation}
\begin{aligned}
	\norm{
		u_h - v_h
	}^2_{L^2_\alpha(\IR_+;H^1_0(\Omega))}
	=
	\int_{-\infty}^{+\infty}
	\,
	\norm{
		\widehat{w}_h(\alpha + \imath \tau) - \widehat{z}_h(\alpha + \imath \tau)
	}^2_{H^1_0(\Omega)}
	\frac{\text{d}\tau}{2\pi}.
\end{aligned}
\end{equation}

Set $s(\theta) = \alpha +\imath \beta \cot\left(\frac{\theta}{2}\right)$ 
and $\tau = \beta \cot\left(\frac{\theta}{2}\right)$, $\theta \in (0,2\pi)$,
with $\beta>0$. Then, one has the following
\begin{equation}
\begin{aligned}\label{eq:hardy_norm_exp_conv}
	\int_{-\infty}^{+\infty}
	\,
	\norm{
		\widehat{u}
		(\alpha +\imath \tau)
		-
		\widehat{v}
		(\alpha +\imath \tau)
	}^2_{H^1_0(\Omega)}
	\text{d}
	\tau
	&
	\\
	&\hspace{-4cm}
	=
	\int_{0}^{2\pi}
	\frac{
		\beta
	}{
		2\sin^2
		\left(
			\frac{\theta}{2}
		\right)
	}
	\norm{
		\widehat{u}
		\left(s(\theta))\right)
		-
		\widehat{v}
		\left(s(\theta))\right)
	}^2_{H^1_0(\Omega)}
	\text{d}
	\theta
	\\
	&\hspace{-4cm}
	=
	\int_{0}^{2\pi}
	\frac{
		\beta
	}{
		2\sin^2
		\left(
			\frac{\theta}{2}
		\right)
		\snorm{s(\theta)}^2
	}
	\norm{
		s(\theta)
		\left(
            \widehat{u}\left(s(\theta))\right)
		      -
		      \widehat{v}\left(s(\theta))\right)
        \right)
	}^2_{H^1_0(\Omega)}
	\text{d}
	\theta.
\end{aligned}
\end{equation}
Recall that with $z = \exp(\imath\theta)$ one has
$\cot\left(\frac{\theta}{2}\right)
	=
	\imath
	\frac{
		z+1
	}{
		z-1
	}
$, thus
	$s(\theta) = s(z) =	\alpha
	-
	\beta
	\frac{
		z+1
	}{
		z-1
	}
$.

Furthermore, the solution to \eqref{eq:sol_w_h_exp}
admits for each $ s \in \Pi_\alpha$ the following expression 
\begin{equation}\label{eq:exp_u}
\begin{aligned}
    \widehat{w}_h(s)
    =
    &
    \sum_{i=1}^{N_h}
    \frac{
        \dotp{\mathcal{L}\{f\}(s)}{\zeta_{h,i}}_{L^2(\Omega)}
        +
        \dotp{\nabla \cdot( \bm{A} \nabla u_{0})}{\zeta_{h,i}}_{L^2(\Omega)}
    }{s+\lambda_{h,i}}
    \zeta_{h,i}
    \\
    =
    &
    \sum_{i=1}^{N_h}
    \left(
    \frac{
        \dotp{\mathcal{L}\{\partial_t f\}(s)}{\zeta_{h,i}}_{L^2(\Omega)}
        +
         \dotp{f(0)}{\zeta_{h,i}}_{L^2(\Omega)}
    }{s(s+\lambda_{h,i})}
    +
    \frac{
        \dotp{\nabla \cdot( \bm{A} \nabla u_{0})}{\zeta_{h,i}}_{L^2(\Omega)}
    }{s(s+\lambda_{h,i})}
    \right)
    \zeta_{h,i},
\end{aligned}
\end{equation}
where the second equality follows from item (i) in
Assumption~\ref{assump:data_f}.
It follows from item (ii) in Assumption~\ref{assump:data_f}
that this representation admits a unique holomorphic extension to
$s \in \mathcal{O}^c \cap \left(\mathbb{C} \backslash \{0,-\lambda_{h,1},\dots,-\lambda_{h,N_h}\}\right)$.

Let us set
\begin{equation}
	g(z)
	=
	\left(
    \alpha
	-
	\beta
	\frac{
		z+1
	}{
		z-1
	}
	\right)
	\widehat{u}
	\left(
	\alpha
	-
	\beta
	\frac{
		z+1
	}{
		z-1
	}
	\right),
	\quad
	z \in \mathcal{D}.
\end{equation}
As pointed out in, e.g., the proof of Lemma 2.2 in Ref.~\refcite{giunta1989more}
and in Section 2 of Ref.~\refcite{weideman2023fully}, provided that $\beta\geq\alpha-\alpha_0$,
the half plane $\Pi_{\rev{\alpha}}$
is mapped \rev{through $\mathcal{M}$} to a disk of center $\delta_{\alpha,\beta}$
and radius $1-\delta_{\alpha,\beta}$ for some $\delta_{\alpha,\beta}<0$
\rev{that depends on $\alpha$ and $\beta$}.
Observe that $g(z) \coloneqq
\left(\mathcal{M}^{-1}(z)\right)\widehat{u}\left(\mathcal{M}^{-1}(z)\right)$,
$z \in \overline{\mathcal{D}}$, is well-defined
and, furthermore, is analytic in $\overline{\mathcal{D}}$.
On the other hand, provided that $\beta<\alpha-\alpha_0$, the
half plane $\Pi_{\rev{\alpha}}$ is mapped to the exterior of a disk of center
$\delta_{\alpha,\beta}>1$ and radius $\delta_{\alpha,\beta}-1$.
Hence $g(z)$ for $z \in \overline{\mathcal{D}}$ is well-defined, and,
as in the previous case, is analytic in $\overline{\mathcal{D}}$.
In either case, there exists $\eta_{\alpha,\beta}>1$ (depending on $\alpha$ and $\beta$)
such that $g(z)$ is analytic in $\mathcal{D}_{\eta_{\alpha,\beta}}$
\rev{and such that $\mathcal{O} \cup \{0,-\lambda_{h,1},\dots,-\lambda_{h,N_h}\} \subset \mathcal{C}_{\eta,\alpha,\beta}$
for any $\eta \in (1,\eta_{\alpha,\beta})$.}

We consider the Taylor series expansion of $g(z)$ centered at the origin
of the complex plane, i.e.~for $z \in \mathcal{D}_{\eta}$ and
$\eta \in (1,\eta_{\alpha,\beta})$
\begin{equation}\label{eq:laurent_series}
	g(z)
	=
	\sum_{j =0}^{\infty}
	c_j
	z^j
	\quad
	\text{with}
	\quad
	c_j
	=
	\frac{1}{2\pi \imath}
	\int_{q \in \partial \mathcal{D}_\eta}
	\frac{g(q)}{q^{j+1}}
	\text{d}q
	\in
	\mathbb{V}^\mathbb{C}_h,
	\quad
	j\in \mathbb{N}_0.
\end{equation}


Consequently, for any $\eta \in (1,\eta_{\alpha,\beta,u})$, we obtain
\begin{equation}
	\norm{c_j}_{H^1_0(\Omega)}
	\leq
	\eta^{-{j}}
	\left(
		\int_{0}^{2\pi}
      	\norm{g(\eta \exp(\imath \theta))}^2_{H^1_0(\Omega)}
		\text{d}\theta
	\right)^{\frac{1}{2}},
	\quad
	j \in \mathbb{N}_0.
\end{equation}
For each $s \in \mathcal{O}^c \cap \left(\mathbb{C} \backslash \{-\lambda_{1,h},\dots,-\lambda_{N_h,h}\}\right)$, we calculate
\begin{equation}
\begin{aligned}
	\norm{\widehat{u}_h(s)}^2_{H^1_0(\Omega)}
	\lesssim
    \sum_{i=1}^{N_h}
    &
    \left(
    \frac{
        \snorm{\dotp{\mathcal{L}\{\partial_t f\}(s)}{\zeta_{h,i}}_{L^2(\Omega)}}^2
        +
        \snorm{\dotp{f(0)}{\zeta_{h,i}}_{L^2(\Omega)}}^2
    }{\snorm{s}^2\snorm{s+\lambda_{h,i}}^2}
    \right.
    \\
    &
    \left.
    +
    \frac{
        \snorm{ \dotp{\nabla \cdot( \bm{A} \nabla u_{0})}{\zeta_{h,i}}_{L^2(\Omega)}}^2
    }{\snorm{s}^2\snorm{s+\lambda_{h,i}}^2}
    \right)
    \lambda_{h,i},
\end{aligned}
\end{equation}
Observe that for each $i \in \{1,\dots,N_h\}$
\begin{equation*}
\begin{aligned}
	\snorm{\mathcal{M}^{-1}(\eta \exp(\imath \theta))+\lambda_{h,i}}^2
	=
	&
	\left(
		\alpha+\lambda_{h,i}- \beta\frac{\eta^2 - 1}{ \eta^2 +1-2\eta \cos \theta}
	\right)^2
	\\
	&
	+
	\beta^2
	\left(
		\frac{2 \eta \sin(\theta)}{\eta^2 +1-2\eta \cos \theta}
	\right)^2
	\\
	\geq
	&
	\min
	\left\{
	\left(
		\alpha+\lambda_{h,N_h}- \beta\frac{\eta - 1}{ \eta+1}
	\right)^2
	,
	\left(
		\alpha+\lambda_{h,1}- \beta\frac{\eta + 1}{ \eta-1}
	\right)^2
	\right\},
\end{aligned}
\end{equation*}
thus
\begin{equation}
	\int_{0}^{2\pi}
	\frac{
		\text{d}\theta
	}{
		\snorm{\mathcal{M}^{-1}(\eta \exp(\imath \theta))+\lambda_{h,i}}^2
	}
	\lesssim
	\frac{1}{\Lambda},
\end{equation}
where $\Lambda$ is as in \eqref{eq:lambda}, and
\begin{equation}
\begin{aligned}
	\int_{0}^{2\pi}
	\norm{g(\eta \exp(\imath \theta))}^2_{\mathbb{V}^\mathbb{C}_h}
	\text{d}\theta
	\lesssim
	&
	\frac{1}{\Lambda}
	\left(
	\sup_{ s \in \mathcal{C}_{\eta,\alpha, \beta}}
	\norm{
		\Pi_h (\mathcal{L}\{\partial_t f\}(s))
	}^2_{H^1_0(\Omega)}
	\right.
	\\
	&
	\left.
	+
	\norm{\Pi_h (f(0))}^2_{H^1_0(\Omega)}
	+
	\norm{\Pi_h (\nabla \cdot( \bm{A} \nabla u_{0}))}^2_{H^1_0 (\Omega)}
	\right).
\end{aligned}
\end{equation}

Let us set for each $R\in \IN$
\begin{equation}\label{eq:taylor_expansion}
	{g}_\rev{R}(z)
	=
	\sum_{j=0}^{R-2}
	c_j
	z^j
	\quad
	\text{for }
	z \in \overline{\mathcal{D}},
\end{equation}
thus for $z \in \overline{\mathcal{D}}$ and any $\eta \in (1,\eta_{\alpha,\beta})$
\begin{equation}\label{eq:error_bound_X}
\begin{aligned}
	\norm{
		g(z)
		-
		{g}_\rev{R}(z)
	}_{H^1_0(\Omega)}
	\lesssim
	\frac{\eta^2}{(\eta-1)\sqrt{\Lambda}}
	&
	\left(
	\sup_{ s \in \mathcal{C}_{\eta,\alpha, \beta}}
	\norm{
		\Pi_h (\mathcal{L}\{\partial_t f\}(s))
	}_{H^1_0(\Omega)}
	\right.
	\\
	&
	+
	\norm{\Pi_h (f(0))}_{H^1_0(\Omega)}
	\\
	&
	\left.
	+
	\norm{\Pi_h  (\nabla \cdot( \bm{A} \nabla u_{0}))}_{H^1_0 (\Omega)}
	\right)
	\eta^{-R}.
\end{aligned}
\end{equation}
For each $R \in \mathbb{N}$, set
\begin{equation}
	\widehat{v}_R(s)
	=
	\frac{u_{0,h}}{s}
	+
	\frac{1}{s}
	g_R
	\left(
		\frac{\alpha -s  + \beta}{\alpha -s  - \beta}
	\right),
	\quad
	s\in\Pi_\alpha.
\end{equation}
and $V_{R} \coloneqq\text{span}\left\{u_{0,h},c_{0},\dots,c_{R-2}\right\} \subset \mathbb{V}_h$,
therefore $\text{dim}(V_R) \rev{\leq} R$.

Observe that for each $s \in \Pi_\alpha$ one has
\begin{equation}\label{eq:expasion_u_Mobius}
	\widehat{v}_R(s)
	=
	\frac{u_{0,h}}{s}
	+
	\sum_{j=0}^{R-2}
	\underbrace{
	\frac{1}{s}
	\left(
		\frac{\alpha -s+ \beta}{\alpha -s - \beta}
	\right)^j
	}_{\eqcolon\widehat{\omega}_{j}(s)}	
	c_j
	\in V_R.
\end{equation}
For each $j \in \mathbb{N}_0$,
the function $\widehat{\omega}_{j}: \Pi_\alpha\rightarrow\IC$ is holomorphic
as it is a rational function with a pole of 
multiplicity $j$, which is located outside $\Pi_\alpha$, and a pole of multiplicity one
located at the origin of the complex plane.
Therefore, it verifies item (i) in Definition~\ref{def:hardy_spaces}.
Recalling the definition of the $\mathscr{H}^2_\alpha$-norm
and Proposition~\ref{prop:properties_hardy}, we obtain
\begin{equation}
	\norm{
		\widehat{\omega}_{j}
	}^2_{\mathscr{H}^2_\alpha}
	=
	\int_{-\infty}^{+\infty}
	\snorm{
		\widehat{\omega}_{j}
		(\alpha  + \imath \tau)
	}^2
	\,
	\frac{\text{d}\tau}{2\pi}
	\leq
	\frac{1}{2\pi}
	\int_{-\infty}^{+\infty}
	\frac{	\text{d}
	\tau}{\snorm{\alpha + \imath \tau}^2}.
\end{equation}
Again, recalling that
\begin{equation}
	\frac{1}{2\pi}
	\int_{-\infty}^{+\infty}
	\frac{	\text{d}
	\tau}{\snorm{\alpha + \imath \tau}^2}
	=
	\frac{1}{2\pi}
	\int_{-\infty}^{+\infty}
	\frac{	\text{d}
	\tau}{{\alpha^2+\tau^2}}
	=
	\frac{1}{2\alpha}
\end{equation}
we obtain
\begin{equation}
	\norm{\widehat{\omega}_{j}}_{\mathscr{H}^2_\alpha}
	\leq 
	\frac{1}{\sqrt{2\alpha}},
	\quad
	j\in \mathbb{N}_0,
\end{equation}
thus verifying item (ii) in Definition~\ref{def:hardy_spaces}.
This, in turn, according to Theorem~(\ref{eq:laplace_transform_bijective}), implies that
$v_R = \mathcal{L}^{-1}\left \{\widehat{v}_R\right\} \in L^2_\alpha(\IR_+) \otimes V_R \subset L^2_\alpha(\IR_+;H^1_0(\Omega))$.

By replacing $v$ by $w_R$ in \eqref{eq:hardy_norm_exp_conv} we obtain
for any $\eta \in (1,\eta_{\alpha,\beta,u})$
\begin{equation}
\begin{aligned}
	\inf_{
	\substack{
		\mathbb{V}_R \subset \mathbb{V}_h
		\\
		\normalfont\text{dim}(\mathbb{V}_R)\leq R	
	}
	}	
	\norm{
		u_h
		-
		\mathsf{P}_{\mathbb{V}_R}
		u_h
	}^2_{L^2_\alpha(\IR_+;H^1_0(\Omega))}
	\leq
	&
	\norm{
		u_h
		-
		v_R
	}^2_{L^2_\alpha(\IR_+;X)}
	\\
	\stackrel{(\ref{eq:hardy_norm_exp_conv})}{\leq}
	&
	\frac{\beta}{4\pi}
	\int_{0}^{2\pi}
	\frac{
		\norm{
			g(\exp(\imath \theta))
			-
			g_R(\exp(\imath \theta))
		}^2_{H^1_0(\Omega)}
	}{
		\sin^2
		\left(
			\frac{\theta}{2}
		\right)
		\snorm{s(\theta)}^2
	}
	\text{d}
	\theta
	\\
	\lesssim
	&
	\sup_{z\in\overline{\mathcal{D}}}
	\norm{
		g(z)
		-
		g_R(z)
	}^2_{H^1_0(\Omega)}
	\int_{0}^{2\pi}
	\frac{
		\beta
		\text{d}
		\theta
	}{
		\sin^2
		\left(
			\frac{\theta}{2}
		\right)
		\snorm{s(\theta)}^2
	}.
\end{aligned}
\end{equation}
We calculate
\begin{equation}
	\int_{0}^{2\pi}
	\frac{
		\beta
		\text{d}
		\theta
	}{
		\sin^2
		\left(
			\frac{\theta}{2}
		\right)
		\snorm{s(\theta)}^2
	}
	\lesssim
	 \norm{\frac{1}{s}}^2_{\mathscr{H}^2_\alpha} 
	 \lesssim
	 \frac{1}{\alpha}.
\end{equation}
Recalling \eqref{eq:error_bound_X}, we conclude the proof. 
\end{proof}

}

\subsection{Convergence Analysis of the \rev{LT-MOR} Method}
\label{sec:convergence_lt_rb}
In this section, we establish the exponential convergence of the \rev{LT-MOR} method.
Let $u_h \in L^2_\alpha\left(\IR_+;\mathbb{V}_h\right)$ 
be the solution to Problem~\ref{prbm:semi_discrete_problem}.
Clearly, it holds that
\begin{equation}\label{eq:l2_bound}
	\int_{0}^{\infty}
	\norm{{u}_h(t)}^2_{H^1_0(\Omega)}
	\exp(-2\alpha t)
    \rev{\text{d}
    t
    }
	<
	\infty,
\end{equation}
thus rendering ${u}_h(t)$ a Hilbert-Schmidt kernel.
\rev{(For further details on Hilbert-Schmidt kernel
and operators we refer to Ch.~IV in Ref.~\refcite{gohberg1996traces}).}
Define the operator
$\mathsf{T}:L^2_\alpha(\IR_+) \rightarrow \mathbb{V}_{h} $ as
\begin{equation}
	\mathsf{T} g
	\coloneqq
	\int_{0}^{\infty}
	u_h(t)
	g(t)
	\exp(-2\alpha t)
	\text{d} t,
	\quad 
	\forall g \in L^2_\alpha(\IR_+).
\end{equation}
Its adjoint $\mathsf{T}^\star: \IV_{h} \rightarrow L^2_\alpha(\IR_+)$
is defined as
\begin{equation}
	\dotp{
		g
	}{
		\mathsf{T}^\star 
		v_h
	}_{ L^2_\alpha(\IR_+)}
	=
	\dotp{
		\mathsf{T} g
	}{
		 v_h
	}_{H^1_0(\Omega)},
	\quad 
	\forall g \in L^2_\alpha(\IR_+), 
	\quad
	\forall
	v_h \in \mathbb{V}_h,
\end{equation}
thus yielding for a.e. $t>0$
\begin{equation}
	\left(
		\mathsf{T}^\star 
		v_h
	\right)(t)
	=
	\dotp{
		{u}_h(t)
	}{
		v_h
	}_{H^1_0(\Omega)}
	,
	\quad 
	\forall 
	v_h
	\in \mathbb{V}_h.
\end{equation}
It follows from \eqref{eq:l2_bound} that
$\mathsf{T}$ is a Hilbert-Schmidt operator,
which in turn is compact.
The Hilbert-Schmidt norm of $\mathsf{T}$ admits the following expression
\begin{equation}
	\norm{\mathsf{T}}_{\text{HS}} 
	= 
	\norm{u_h}_{L^2_\alpha\left(\IR_+;H^1_0(\Omega)\right)}.
\end{equation}
Since $\mathsf{T}$ is compact, 
$\mathsf{K}= \mathsf{T} \mathsf{T}^\star:\mathbb{V}_h \rightarrow \mathbb{V}_h$ 
and 
$\mathsf{C}= \mathsf{T}^\star \mathsf{T}: L^2_\alpha(\IR_+) \rightarrow L^2_\alpha(\IR_+)$ 
are self-adjoint, non-negative, compact operators, given by
\begin{equation}
\begin{aligned}
	\left(\mathsf{C} g\right)(t)
	&
	=
	\int_{0}^\infty
	\dotp{
		{u}_h(t)
	}{
		{u}_h(\tau)
	}_{H^1_0(\Omega)}
	\,
	g(\tau)
	\exp(-2\alpha \tau)
	\text{d} 
	\tau,
	\quad
	\forall g \in L^2(\IR_+,\alpha),
	\\
	\mathsf{K} 
	v_h
	&
	=
	\int_{0}^\infty 
	{u}_h(\tau)
	\dotp{
		{u}_h(\tau)
	}{
		v_h
	}_{H^1_0(\Omega)}
	\exp(-2\alpha \tau)
	\text{d} \tau,
	\quad
	\forall 
	v_h
	\in \mathbb{V}_h.
\end{aligned}
\end{equation}
The operator $\mathsf{K}$ can be represented by the $N_h \times N_h$ 
symmetric, positive definite matrix
\begin{equation}
	\left(
		{\bm{\mathsf{K}}}
	\right)_{i,j}
	=
	\int_{0}^{\infty} 
	\dotp{
		{u}_h(\tau)
	}{
		\varphi_i
	}_{H^1_0(\Omega)}
	\dotp{
		\varphi_j
	}{
		{u}_h(\tau)
	}_{H^1_0(\Omega)}
	\exp(-2\alpha \tau)
	\text{d}
	\tau,
\end{equation}
whose eigenvalues $\sigma_1^2 \geq \cdots \geq \sigma_r^2 \geq 0$,
being $r$ the rank of the matrix ${\bm{\mathsf{K}}}$,
and associated eigenvectors $\zeta_i \in$ $\IC^{N_h}$
satisfy
\begin{equation}\label{eq:eig_K}
	{\bm{\mathsf{K}}}
	\boldsymbol\zeta_i
	=
	\sigma_i^2 \boldsymbol\zeta_i, \quad i=1, \ldots, r.
\end{equation}
Set
\begin{align}
	\zeta_{i}
	=
	\sum_{j=1}^{N_h}
	\left(
	\boldsymbol\zeta_{i}
	\right)_j
	\varphi_{j}
	\in 
	\mathbb{V}_h.
\end{align}
The functions $\psi_i \in L^2_\alpha(\IR_+)$ defined as
\begin{equation}
	\psi_i
	=
	\frac{1}{\sigma_i} 
	\mathsf{T}^\star \zeta_i, \quad i=1, \ldots, r,
\end{equation}
are the eigenvectors of $\mathsf{C}$. 
Finally, $u_h(t) \in \mathbb{V}_h$ admits the expansion for a.e. $t>0$
\begin{equation}\label{eq:decom_T_HS}
	u_h(t)
	=
	\sum_{i=1}^r 
	{\sigma}_i {\zeta}_i \psi_i(t)
	=
	\sum_{i=1}^r 
	{\zeta}_i
	\dotp{u_h(t)}{{\zeta}_i}_{H^1_0(\Omega)},
\end{equation}
and the following decomposition holds for $\mathsf{T}$
\begin{equation}
	\mathsf{T}
	=
	\sum_{i=1}^r 
	\sigma_i 
	\zeta_i
	\left(
		\psi_i
		,
		\cdot
	\right)_{L^2(\IR_+,\alpha)}.
\end{equation}
Next, for $R\leq r$ we set
\begin{equation}\label{eq:v_rb_optimal}
	\mathbb{V}^{\text{(rb)}}_R
	=
	\text{span}
	\left\{	
		\zeta_1
		,
		\dots
		,
		\zeta_R
	\right\}
	\subset
	\mathbb{V}_h.
\end{equation}
Let $\mathsf{P}^{\text{(rb)}}_R: H^1_0(\Omega) \rightarrow \mathbb{V}^{\text{(rb)}}_R$
be the projection operator onto $\mathbb{V}^{\text{(rb)}}_R$ and, for any 
finite dimensional $X_R \subset H^1_0(\Omega)$, let
$\mathsf{P}_{X_R}: H^1_0(\Omega) \rightarrow X_R$ be the projection operator
onto a finite dimensional subspace ${X_R}$ as defined in \eqref{eq:projection_H1}.
Then, it holds (see, e.g., Ref.~\refcite{quarteroni2015reduced}, Section 6.4)
\begin{equation}\label{eq:min_time_domain}
	\norm{
		u_h
		-
		\mathsf{P}^{\text{(rb)}}_R
		u_h
	}^2_{L^2_\alpha(\IR_+;H^1_0(\Omega))}
	=
	\min_{
	\substack{
		\IV_R \subset \mathbb{V}_h
		\\
		\text{dim}(\IV_R)\leq R	
	}
	}	
	\norm{
		u_h
		-
		\mathsf{P}_{\IV_R}
		u_h
	}^2_{L^2_\alpha(\IR_+;H^1_0(\Omega))},
\end{equation}
thus with $\mathbb{V}^{\text{(rb)}}_R$ as in \eqref{eq:v_rb_optimal}
\begin{equation}\label{eq:min_time_domain_2}
	\mathbb{V}^{\text{(rb)}}_R
	=
	\argmin_{
	\substack{
		\IV_R \subset \mathbb{V}_h
		\\
		\text{dim}(X_R)\leq R	
	}
	}	
	\norm{
		u_h
		-
		\mathsf{P}_{\IV_R}
		u_h
	}^2_{L^2_\alpha(\IR_+;H^1_0(\Omega))}.
\end{equation}
However, the computation of the solution to the minimization
problem stated in \eqref{eq:min_time_domain} is not feasible.
Even if we consider a time discretization of the right-hand side
in \eqref{eq:min_time_domain} as discussed in Section~\ref{sec:reduced_time_dependent_problem_traditional},
this would entail the computation
of the solution itself, which we eventually would like to approximate
using the reduced basis method, thus defeating the purpose of the
algorithm (cf. Remark~\ref{rmk:compute_rb_time}).

The key insight of the \rev{LT-MOR} method consists in identifying the
norm equivalence stated in Theorem~\ref{eq:laplace_transform_bijective}
and casting the minimization problem \eqref{eq:min_time_domain}
in terms of the Laplace transform
$\widehat{u}_h = \mathcal{L}\{u_h\} \in \mathscr{H}^2_\alpha(\mathbb{V}^\mathbb{C}_h)$
of $u_h \in L^2_\alpha(\IR_+;\mathbb{V}_h)$ as follows
\begin{equation}\label{eq:min_laplace_domain}
	\mathbb{V}^{\text{(rb)}}_R
    =	
	\argmin_{
	\substack{
		X_R \subset \mathbb{V}^\mathbb{C}_h
		\\
		\text{dim}(X_R)\leq R	
	}
	}	
	\norm{
		\widehat{u}_h
		-
		\mathsf{P}_{X_R}
		\widehat{u}_h
	}^2_{\mathscr{H}^2_\alpha(H^1_0(\Omega))}.
\end{equation}


\begin{remark}
Even though at a continuous level \eqref{eq:min_time_domain_2} and \eqref{eq:min_laplace_domain}
yield the same reduced space $\mathbb{V}^{\normalfont\text{(rb)}}_R$,
at a discrete level the offline
computations required to construct $\mathbb{V}^{\normalfont\text{(rb)}}_R$
differ vastly, as described in Section~\ref{sec:FRB}.
This and other computational aspects of the \rev{LT-MOR} method
are extensively discussed in Section~\ref{sec:computation_rb} ahead.
\end{remark}

\begin{remark}[Inverse Inequality]\label{rmk:inv_inequality}
Being $\mathbb{V}_h$ a \rev{finite-dimensional} subspace of $H^1_0(\Omega)$, 
all norms are equivalent. 
This implies the existence of $C^{\normalfont\text{inv}}_h>0$, which depends
on the discretization parameter $h>0$, such that 
\begin{equation}
	\norm{v_h}_{H^1_{0}(\Omega)} \leq C^{\normalfont\text{inv}}_h \norm{v_h}_{L^2(\Omega)},
	\quad
	\forall v_h \in \mathbb{V}_h. 
\end{equation}
In general, under the assumption that as $h\rightarrow 0^{+}$
the sequence of spaces $\{\mathbb{V}_h\}_{h>0}$
provides a more precise approximation of $H^1_0(\Omega)$, we can 
reasonably expect $C^{\normalfont\text{inv}}_h \rightarrow +\infty$ as
$h\rightarrow 0^{+}$.
\end{remark}

Equipped with these results, we are now in a position  
of stating and proving the main result concerning the  
exponential convergence of the \rev{LT-MOR} method.
We prove first the following lemma.

\rev{
\begin{theorem}[Exponential Convergence of the LT-MOR]
\label{eq:convergence_lt_rb_method}
Let Assumption~\ref{assump:data_f} be satisfied with $\alpha_0>0$.
Furthermore, let $u_h \in \mathcal{W}_{\rev{\alpha_0}}(\IR_+;\mathbb{V}_h)$
be the solution to Problem~\ref{prbm:semi_discrete_problem},
and let $u^{\normalfont\text{(rb)}}_R\in \mathcal{W}_{\rev{\alpha_0}}(\IR_+;\IV^{\normalfont\text{(rb)}}_R)$
be the solution to Problem~\ref{pr:sdpr} with 
$\IV^{\normalfont\text{(rb)}}_R$ be as in \eqref{eq:min_laplace_domain}.

Then, for any $\alpha>\alpha_0$ and any $\beta > 0$
there exists $\eta_{\alpha,\beta}>1$
such that for $\eta \in (1,\eta_{\alpha,\beta})$ and
$R \in \{1,\dots,N_h\}$ it holds that
\begin{equation}
\begin{aligned}
	\norm{
		u_h - u^{\normalfont\text{(rb)}}_R
	}_{L^2_\alpha\left(\IR_+;H^1_0(\Omega)\right)}
	\lesssim
    &
	\frac{ C^{\normalfont\text{inv}}_h \eta^2}{(\eta-1)\sqrt{\alpha \Lambda}}
	\left(
	\sup_{ s \in  \partial\mathcal{C}_{\eta,\alpha, \beta}}
	\norm{
		\Pi_h (\mathcal{L}\{\partial_t f\}(s))
	}_{H^1_0(\Omega)}
	\right.
	\\
	&
	\left.
	+
	\norm{\Pi_h (f(0))}_{H^1_0(\Omega)}
	+
	\norm{\Pi_h ( \nabla \cdot( \bm{A} \nabla u_{0}))}_{H^1_0 (\Omega)}
	\right)
	\eta^{-R}
    \\
    &
    +
    C^{\normalfont\text{inv}}_h
    \norm{
    \left(
		\normalfont\text{Id}
		-
		\mathsf{P}^{\text{(rb)}}_R
	\right)
    u_{0,h}
    }_{H^1_0(\Omega)},
\end{aligned}
\end{equation}
where the implicit constant is independent of the discretization parameter $h>0$
and $\Lambda>0$ is as in \eqref{eq:lambda}.
\end{theorem}

\begin{proof}
For a.e. $t \in \IR_{+}$ define 
\begin{equation}\label{eq:def_eta}
	\eta^{\text{(rb)}}_R(t) 
	\coloneqq 
	u^{\normalfont\text{(rb)}}_R(t)
	   -
	\left(\mathsf{P}^{\text{(rb)}}_R u_h\right)(t)
	\in
	\IV^{\text{(rb)}}_R.
\end{equation} 
By subtracting \eqref{eq:semi_discrete_time} and \eqref{eq:semi_discrete_rom},
and recalling that $\IV^{\text{(rb)}}_R\subset\mathbb{V}_h$ we obtain
that for all $v^{\text{(rb)}}_R\in \IV^{\text{(rb)}}_R$ it holds
\begin{equation}\label{eq:subs_rb}
	\dotp{
		\partial_t
		\left(
			u^{\normalfont\text{(rb)}}_R\rev{(t)}
			-
			u_h\rev{(t)}
		\right)
	}{
		v^{\text{(rb)}}_R
	}_{L^2(\Omega) }	
	+
	\mathsf{a}
	\dotp{
		u^{\normalfont\text{(rb)}}_R\rev{(t)}
		-
		u_h\rev{(t)}
	}{
		v^{\text{(rb)}}_R
	}
	=
	0.
\end{equation}
Using \eqref{eq:subs_rb} we may conclude that
$\eta^{\text{(rb)}}_R(t)$ satisfies for a.e.~$t \in \IR_+$ and for each $v^{\text{(rb)}}_R \in \IV^{\text{(rb)}}_R$
the following problem
\begin{equation}
\begin{aligned}
	\dotp{	
		\partial_t
		\eta^{\text{(rb)}}_R
	}{
		v^{\text{(rb)}}_R
	}_{L^2(\Omega) }	
    +
	\mathsf{a}
	\dotp{
		\eta^{\text{(rb)}}_R
	}{
		v^{\text{(rb)}}_R
	}
	=
    &
	\dotp{
		\left(
			\text{Id}
			-
			\mathsf{P}^{\text{(rb)}}_R
		\right)
		\partial_t
		 u_h 
	}{
		v^{\text{(rb)}}_R
	}_{L^2(\Omega)}
	\\
	&
    +
	\mathsf{a}
	\dotp{
		\left(
			\text{Id}
			-
			\mathsf{P}^{\text{(rb)}}_R
		\right)
		u_h
	}{
		v^{\text{(rb)}}_R
	},
\end{aligned}
\end{equation}
equipped with the initial condition
\begin{equation}
	\eta^{\text{(rb)}}_R(0) 
	= 
	u^{\normalfont\text{(rb)}}_R(0)
	 -
	\left(\mathsf{P}^{\text{(rb)}}_R u_h\right)(0)
	= 
	0.
\end{equation}
Furthermore, let $\mathsf{Q}^{\text{(rb)}}_R: \mathbb{V}_h \rightarrow \mathbb{V}^{\text{(rb)}}_R$ be the 
$L^2(\Omega)$-based projection operator onto $\mathbb{V}^{\text{(rb)}}_R$,
i.e.,~defined for $w_h \in \mathbb{V}_h$ as
\begin{equation}
    	\dotp{
		\left(
			\text{Id}
			-
			\mathsf{Q}^{\text{(rb)}}_R
		\right)
        w_h
	}{
		v^{\text{(rb)}}_R
	}_{L^2(\Omega)}
    =
    0,
    \quad
    \forall
    v^{\text{(rb)}}_R \in \mathbb{V}^{\text{(rb)}}_R,
\end{equation}
and define 
\begin{equation}
    \zeta^{\text{(rb)}}_R(t)
    = 
    \eta^{\text{(rb)}}_R(t) 
    - 
    \mathsf{Q}^{\text{(rb)}}_R
    \left(
			\text{Id}
			-
			\mathsf{P}^{\text{(rb)}}_R
	\right)
    u_h(t).
\end{equation}
Therefore, $\zeta^{\text{(rb)}}_R$ is solution to
\begin{equation}
\begin{aligned}
	\dotp{	
		\partial_t
		\zeta^{\text{(rb)}}_R
	}{
		v^{\text{(rb)}}_R
	}_{L^2(\Omega) }	
    +
	\mathsf{a}
	\dotp{
		\zeta^{\text{(rb)}}_R
	}{
		v^{\text{(rb)}}_R
	}
	=
    &
    \mathsf{a}
	\dotp{
        \mathsf{Q}^{\text{(rb)}}_R
		\left(
			\text{Id}
			-
			\mathsf{P}^{\text{(rb)}}_R
		\right)
		  u_h
	}{
		v^{\text{(rb)}}_R
	}
	\\
	&
    	+
	\mathsf{a}
	\dotp{
		\left(
			\text{Id}
			-
			\mathsf{P}^{\text{(rb)}}_R
		\right)
		  u_h(t) 
	}{
		v^{\text{(rb)}}_R
	},
\end{aligned}
\end{equation}
with initial condition
\begin{equation}
    \zeta^{\text{(rb)}}_R(0)
    =
    -
    \mathsf{Q}^{\text{(rb)}}_R
    \left(
			\text{Id}
			-
			\mathsf{P}^{\text{(rb)}}_R
	\right)
    u_{0,h}.
\end{equation}
Recalling Theorem~\ref{eq:well_posedness} and Remark~\ref{rmk:inv_inequality}
\begin{equation}\label{eq:error_including_IC}
\begin{aligned}
    \norm{\zeta^{\text{(rb)}}_R}_{L^2_\alpha\left(\IR_+;H^{1}_0(\Omega)\right)}
    \lesssim
    &
    C^{\normalfont\text{inv}}_h
    \norm{
		\left(
			\text{Id}
			-
			\mathsf{P}^{\text{(rb)}}_R
		\right)
		  u_h
    }_{L^2_\alpha(\mathbb{R}_+;H^1_0(\Omega))}
    \\
    &
    +
    C^{\normalfont\text{inv}}_h
    \norm{
    \left(
		\text{Id}
		-
		\mathsf{P}^{\text{(rb)}}_R
	\right)
    u_{0,h}
    }_{H^1_0(\Omega)},
\end{aligned}
\end{equation}
where we have used that $\norm{\mathsf{Q}^{\text{(rb)}}_R w_h}_{L^2(\Omega)} \leq \norm{w_h}_{L^2(\Omega)}$
for any $w \in \mathbb{V}_h$.

Therefore, by combining \eqref{eq:def_eta} and \eqref{eq:error_including_IC},
we obtain
\begin{equation}
\begin{aligned}
    \norm{
		u_h - u^{\normalfont\text{(rb)}}_R
	}_{L^2_\alpha\left(\IR_+;H^1_0(\Omega)\right)}
	\lesssim
    &
    (1+C^{\normalfont\text{inv}}_h)
    \norm{
		\left(
			\text{Id}
			-
			\mathsf{P}^{\text{(rb)}}_R
		\right)
		  u_h
    }_{L^2_\alpha(\mathbb{R}_+;H^1_0(\Omega))}
    \\
    &
    +
    C^{\normalfont\text{inv}}_h
    \norm{
    \left(
		\text{Id}
		-
		\mathsf{P}^{\text{(rb)}}_R
	\right)
    u_{0,h}
    }_{H^1_0(\Omega)}.
\end{aligned}
\end{equation}
This bound together with Lemma~\ref{eq:approximation_Hardy_spaces} yields the final result.
\end{proof}
}

\begin{remark}[Non-Vanishing Initial Condition]
Observe that Theorem~\ref{eq:convergence_lt_rb_method}
states exponential convergence of the \rev{LT-MOR} method
up to the approximation of the initial condition $u_{0,h}$ in $\mathbb{V}^{\text{(rb)}}_R$.
Consider $w_h(t) =  u_h(t)- u_{0,h} $,
where $w_h(t)$ is solution to 
\begin{equation}
\begin{aligned}
	\dotp{\partial_{t}w_h(t)}{v_h}_{L^2(\Omega)} 
	+ 
	\mathsf{a}
	\dotp{w_h(t)}{v_h}
	= 
	\dotp{f(t)}{v_h}_{L^2(\Omega)}
	-
	\mathsf{a}
	\dotp{u_{0,h}}{v_h},
	\quad
	\forall v_h \in H^1_0(\Omega,\mathbb{R}),
\end{aligned}
\end{equation}
equipped with vanishing initial conditions, i.e., $w_h(0) = 0$.
One can apply the \rev{LT-MOR} method as described in Section~\ref{sec:FRB}
to $w_h$ and reconstruct the solution through $u_h(t) = u_{0,h}+w_h(t)$.
\end{remark}

\begin{remark}[Optimal value of $\beta$]
So far, $\beta > 0$ has not been fixed.
Following Ref.~\cite{weideman1999algorithms}, and assuming that
$f(t) = 0$, the optimal values of $\beta$ and $\eta$ (in the sense that $\eta$
is maximized) are
\begin{equation}\label{eq:optimal_values}
	\beta_{\normalfont\text{opt}}
	=
	\sqrt{(\alpha+\lambda_{h,1})(\alpha+\lambda_{h,N_h})}
	\quad
	\text{and}
	\quad
	\eta_{\alpha,\beta_{\normalfont\text{opt}}}
	=
	\snorm{
	\frac{
		-\lambda_{h,N_h}
            -
            \alpha-\beta_{\normalfont\text{opt}}
	}{
		-\lambda_{h,N_h}
            -
            \alpha+\beta_{\normalfont\text{opt}}
	}}
	>1,
\end{equation}
The dependence of $\eta_{\alpha,\beta_{\normalfont\text{opt}},u_h}$ upon
$\alpha,\beta_{\normalfont\text{opt}}$ and $u_h$
comes through the eigenvalues of $\mathsf{a}(\cdot,\cdot)$.
Furthermore, for the case $f(t) \neq 0$, 
if all the poles of $\mathcal{L}\{\partial_t f\}$ are located in the interior of the 
circle $\mathcal{C}_{\eta_{\normalfont\text{opt}},\alpha,\beta_{\normalfont\text{opt}}}$,
then the result still is valid. 

This makes $\eta_{\alpha,\beta_{\normalfont\text{opt}}}$ dependent on the
discretization parameter $h>0$ of $\mathbb{V}_h$, and, in general, one has that
$\eta_{\alpha,\beta_{\normalfont\text{opt}}} \rightarrow 1^+$ as $h\rightarrow 0^+$.
\end{remark}

\section{Computational Aspects of the \rev{LT-MOR} Method}
\label{sec:computation_rb}
In Section~\ref{sec:FRB}, we have introduced the \rev{LT-MOR} method, whereas
in Section~\ref{sec:Basis} a convergence analysis of the aforementioned method
is performed. We \rev{show} that if the reduced space is chosen as in \eqref{eq:min_laplace_domain},
exponential convergence toward the high-fidelity solution is expected.
However, very little practical indications are provided as to how this reduced space is computed.
In this section, we elaborate on certain computational aspects of the
\rev{LT-MOR} method that are of importance for implementation purposes. 

\subsection{Real vs Complex Basis}
\label{sec:real_complex_basis}
The \rev{LT-MOR} method delivers a reduced basis built 
using \rev{snapshots} computed in the Laplace domain, therefore
these belong to a complex Hilbert space. However, Problem~\ref{pr:sdpr}
is posed in a real-valued setting.
In the following, we argue that the following statements are true:
\begin{itemize}
	\item[(i)]
	The computation of the reduced basis through the \rev{LT-MOR} method
	by using only the real part of the snapshots produces the same
	reduced space as if the basis were to be computed using a complex snapshot matrix.
	\item[(ii)]
	The computation of the basis through the \rev{LT-MOR} method using the complex-valued
	snapshots and a complex SVD delivers a reduced basis with
	vanishing imaginary part, i.e.~each element of the basis is real-valued.
\end{itemize}

To prove these statements, we introduce the Hilbert transform.
For $\varphi \in L^2(\IR)$ it is defined as 
\begin{equation}
	\left(
		\mathcal{H}
		\varphi
	\right)
	(x)
	\coloneqq
	\frac{1}{\pi}
	\text{p.v.}
	\int_{-\infty}^{+\infty}
	\frac{ 
		\varphi(y)
	}{
		x-y
	}
	\text{d}y,
	\quad
	x \in \mathbb{R},
\end{equation}
where ``p.v.'' signifies that the integral is understood 
in the sense of Cauchy's principal value.
The Hilbert transform $\mathcal{H}:L^p(\IR) \rightarrow L^p(\IR)$
defines a bounded linear operator for $p\in (0,\infty)$
In particular, for $p=2$ it holds (see, e.g., Ref.~\refcite{laeng2010simple})
\begin{equation}\label{eq:properties_hilbert_transform}
	\norm{\mathcal{H} \varphi}_{L^2(\IR)}=\norm{\varphi}_{L^2(\IR)}
	\quad
	\text{for any }\varphi \in L^2(\IR).
\end{equation}

In the following proposition, we recall an important result
stated at the beginning of Section 4.22 in Ref.~\refcite{king2009hilbert}
which is usually known as Titchmarsh’s theorem.

\begin{proposition}\label{prop:titchmarsh}
Let $F \in L^2(\IR)$ be complex-valued and denote by $\Re\{F(x)\}$ and $\Im\{F(x)\}$
the real and imaginary part of $F$, respectively, for a.e. $x\in \mathbb{R}$.
The following statements are equivalent
\begin{itemize}
	\item[(i)] $\Im\{F(x)\} = \left(\mathcal{H}\Re\{F\}\right)(x)$  for a.e. $x \in \IR$.
	\item[(ii)]  $\Re\{F(x)\} = -\left(\mathcal{H}\Im\{F\}\right)(x)$ for a.e. $x \in \IR$.
	\item[(iii)] If $f(t)$ denotes the inverse Fourier transform of $F(x)$, then $f(t)=0$ for $t<0$.
	\item[(iv)] $F(x+\imath y)$ is an analytic function in the upper half plane and for a.e. $x \in \IR$
	\begin{equation*}
		F(x) = \lim_{y\rightarrow 0^+}F(x+\imath y)
		\quad
		\text{and}
		\quad
		\int_{-\infty}^{\infty}\snorm{F(x+\imath y)}^2 \normalfont\text{d}x
		<
		\infty,
		\quad
		\text{for }y>0.
	\end{equation*}
\end{itemize}
\end{proposition}

The Hilbert transform together with Proposition~\ref{prop:titchmarsh} are
important \rev{ingredients} in the proof of the following result.
\begin{lemma}\label{lmm:real_basis}
Let $\widehat{u}_h = \mathcal{L}\{u_h\} \in \rev{\mathscr{H}^2_\alpha}(\mathbb{V}^\mathbb{C}_h)$
be the Laplace transform of $u_h \in L^2(\IR_+;\mathbb{V}_h)$, and let
$\mathbb{V}^{\normalfont\text{(rb)}}_R$ be as in \eqref{eq:min_laplace_domain}.
Then, 
\begin{itemize}
	\item[(i)]
	It holds that
	\begin{equation*}
	\mathbb{V}^{\normalfont\text{(rb)}}_R
	=
	\argmin_{
	\substack{
		\IV_R \subset \mathbb{V}_h
		\\
		\normalfont\text{dim}\left(\mathbb{V}_R\right)\leq R	
	}
	}
	\int_{-\infty}^{\infty}
	\,
	\norm{
		\Re\left\{\widehat{u}_h(\alpha+\imath \tau)\right\}
		-
		\mathsf{P}_{{\IV}_R}
		\Re\left\{\widehat{u}_h (\alpha+\imath \tau)\right\}
	}^2_{H^1_0(\Omega)}
	\normalfont\text{d}\tau.	
	\end{equation*}
	\item[(ii)]
	$\mathbb{V}^{\normalfont\text{(rb)}}_R\subset \mathbb{V}_h$ when
	$\mathbb{V}_h$ is viewed as a vector space over the field of real numbers. 
\end{itemize}
\end{lemma}

\begin{proof}
According to Lemma~\ref{eq:lemma_Hardy_space_uh} the map 
$\Pi_\alpha \ni s\mapsto \widehat{u}_h(s) \in \mathbb{V}^{\mathbb{C}}_h$ is holomorphic. 
Set $F_h(z) = \widehat{u}_h(\mathcal{T}_\alpha(z)) \in \mathbb{V}^\mathbb{C}_h$, for $z \in \Pi^+$,
where $\mathcal{T}_\alpha: \Pi^+ \rightarrow \Pi_\alpha: z \mapsto -\imath z+\alpha$,
and $\Pi^+\coloneqq \left\{z\in \IC:\; \Im\{z\}>0\right\}$.

We verify item (iv) in Proposition~\ref{prop:titchmarsh}.
The map $\Pi^+ \ni z \mapsto F_h(z) \in \mathbb{V}^\mathbb{C}_h \subset H^1_0(\Omega)$
is holomorphic, and according to item (iii) in Proposition~\ref{prop:properties_hardy} for any $y>0$
\begin{equation}
\begin{aligned}
	\int_{-\infty}^{\infty}\norm{F_h(x+\imath y)}^2_{H^1_0(\Omega)} \normalfont\text{d}x
	&
	=
	\int_{-\infty}^{\infty}\norm{ \widehat{u}_h(\mathcal{T}_\alpha(x+\imath y))}^2_{H^1_0(\Omega)} \normalfont\text{d}x
	\\
	&
	=
	\int_{-\infty}^{\infty}\norm{ \widehat{u}_h(-\imath x+y+\alpha)}^2_{H^1_0(\Omega)} \normalfont\text{d}x
	\\
	&
	=
	\int_{-\infty}^{\infty}\norm{ \widehat{u}_h(\alpha+y+\imath \tau)}^2_{H^1_0(\Omega)} \normalfont\text{d}\tau
	<
	\infty.
\end{aligned}
\end{equation}
In addition, from item (i) in Proposition~\ref{prop:properties_hardy} we conclude that
$F_h(x) = \displaystyle\lim_{y\rightarrow 0^+}F_h(x+\imath y)$.

Therefore, we can use Proposition~\ref{prop:titchmarsh}. In particular item (i) 
asserts that 
\begin{equation}
	\Im\{F_h(x)\} = \left(\mathcal{H}\Re\{F_h\}\right)(x)
	\quad
	\text{for a.e. } x \in \IR.
\end{equation}
Notice that $F_h(-\tau) = \widehat{u}_h(\mathcal{T}_\alpha(-\tau)) = \widehat{u}_h(\alpha+\imath \tau)$,
therefore
\begin{equation}\label{eq:real_imag_parts}
\begin{aligned}
	\norm{
		\widehat{u}_h(\alpha+\imath \tau)
	}^2_{H^1_0(\Omega)}
	&
	=
	\norm{
		\Re\{F_h(-\tau) \}
		+
		\imath
		\Im\{F_h(-\tau) \}
	}^2_{H^1_0(\Omega)}
	\\
	&
	=
	\norm{
		\Re\{F_h(-\tau) \}
		+
		\imath
		\left(\mathcal{H}\Re\{F_h\}\right)(-\tau)
	}^2_{H^1_0(\Omega)}
	\\
	&
	=
	\norm{
		\Re\{F_h(-\tau) \}
	}^2_{H^1_0(\Omega)}
	+
	\norm{
		\left(\mathcal{H}\Re\{F_h\}\right)(-\tau)
	}^2_{H^1_0(\Omega)}.
\end{aligned}
\end{equation}
The application of the Hilbert transform in \eqref{eq:real_imag_parts} is understood to be coefficient-wise in
the expansion of $F_h(x) \in \mathbb{V}^\mathbb{C}_h$, as $\mathbb{V}^\mathbb{C}_h$ is a finite
dimensional subspace of $H^1_0(\Omega)$. 
More precisely, one has 
$$\widehat{u}_h(s) = \sum_{j=1}^{N_h}(\widehat{\bm{\mathsf{u}}}_h(s))_j \varphi_j\in  \mathbb{V}^\mathbb{C}_h,$$
with $\widehat{\bm{\mathsf{u}}}_h \in \mathscr{H}^2_\alpha(\IC^{N_h})$. Thus,
\begin{equation}\label{eq:Hilbert_transform_finite_dimension}
	\left(\mathcal{H}\Re\{F_h\}\right)(-\tau)
	=
	\sum_{j=1}^{N_h}
	\mathcal{H}
	\left(
	\Re\left\{\left(\widehat{\bm{\mathsf{u}}}_h(\alpha+\imath \cdot)\right)_j\right\} 
	\right)(-\tau)\varphi_j,
\end{equation}
where the dot in the expression on the right-hand side of \eqref{eq:Hilbert_transform_finite_dimension}
denotes integration with respect to that variable in the definition of the Hilbert transform.

Next, we calculate
\begin{equation}
	\norm{
		\left(\mathcal{H}\Re\{F_h\}\right)(-\tau)
	}_{H^1_0(\Omega)}
	=
	\norm{
		\bm{\mathsf{v}}_h	
		(-\tau)
	}_{\IC^{N_h}},
\end{equation}
where, for each $\tau \in \IR$, 
\begin{equation}
	\bm{\mathsf{v}}_h	
	(-\tau)
	=
	\begin{pmatrix}
		\mathcal{H} \left( \Re\left\{\left( \bm{\mathsf{R}}_h \widehat{\bm{\mathsf{u}}}_h(\alpha+\imath \cdot)\right)_1\right\} \right)(-\tau)\\
		\vdots\\
		\mathcal{H} \left( \Re\left\{\left(  \bm{\mathsf{R}}_h \widehat{\bm{\mathsf{u}}}_h(\alpha+\imath \cdot)\right)_{N_h}\right\} \right)(-\tau)
	\end{pmatrix}
\end{equation}
and $\bm{\mathsf{B}}_h =  \bm{\mathsf{R}}^\top_h  \bm{\mathsf{R}}_h$ is the
Cholesky decomposition of  $\bm{\mathsf{B}}_h$ with $ \bm{\mathsf{R}}_h$ an upper triangular
matrix.

Then, for any $y>0$
\begin{equation}
\begin{aligned}\label{eq:usage_of_ht_1}
	\int_{-\infty}^{\infty}
	\norm{
		\widehat{u}_h(\alpha+y+\imath \tau)}^2_{H^1_0(\Omega)
	} \normalfont\text{d}\tau
	=
	&
	\int_{-\infty}^{\infty}
	\norm{
		\Re\{F_h(-\tau) \}
	}^2_{H^1_0(\Omega)}
	\normalfont\text{d}\tau
	\\
	&
	+
	\int_{-\infty}^{\infty}
	\norm{
		\left(\mathcal{H}\Re\{F_h\}\right)(-\tau)
	}^2_{H^1_0(\Omega)}
	\normalfont\text{d}\tau,
\end{aligned}
\end{equation}
and
\begin{equation}\label{eq:usage_of_ht_2}
\begin{aligned}
	\int_{-\infty}^{\infty}
	\norm{
		\left(\mathcal{H}\Re\{F_h\}\right)(-\tau)
	}^2_{H^1_0(\Omega)}
	\normalfont\text{d}\tau
	=
	&
	\sum_{j=1}^{N_h}
	\int_{-\infty}^{\infty}
	\snorm{
		\mathcal{H} \left( \Re\left\{\left( \bm{\mathsf{R}}_h \widehat{\bm{\mathsf{u}}}_h(\alpha+\imath \cdot)\right)_j\right\} \right)(-\tau)
	}^2
	\normalfont\text{d}\tau
	\\
	=
	&
	\sum_{j=1}^{N_h}
	\norm{
		\mathcal{H} \left( \Re\left\{\left( \bm{\mathsf{R}}_h \widehat{\bm{\mathsf{u}}}_h(\alpha+\imath \cdot)\right)_j\right\} \right)
	}^2_{L^2(\IR)}
	\\
	=
	&
	\sum_{j=1}^{N_h}
	\norm{
		\Re\left\{\left( \bm{\mathsf{R}}_h \widehat{\bm{\mathsf{u}}}_h(\alpha+\imath \cdot)\right)_j\right\}
	}^2_{L^2(\IR)}
	\\
	=
	&
	\sum_{j=1}^{N_h}
	\int_{-\infty}^{\infty}
	\snorm{
		\left( 
	 		\bm{\mathsf{R}}_h
			\Re\left\{\widehat{\bm{\mathsf{u}}}_h(\alpha+\imath \tau)\right\}
		\right)_j
	}^2
	\normalfont\text{d}\tau
	\\
	=
	&
	\int_{-\infty}^{\infty}
	\norm{
	 	\bm{\mathsf{R}}_h
		\Re\left\{\widehat{\bm{\mathsf{u}}}_h(\alpha+\imath \tau)\right\}
	}^2_{\IC^{N_h}}
	\normalfont\text{d}\tau,
\end{aligned}
\end{equation}
where we have used \eqref{eq:properties_hilbert_transform}.

It follows from \eqref{eq:usage_of_ht_1} and \eqref{eq:usage_of_ht_2}
that for any $y>0$
\begin{equation}
	\int_{-\infty}^{\infty}
	\norm{
		\widehat{u}_h(\alpha+y+\imath \tau)}^2_{H^1_0(\Omega)
	} \normalfont\text{d}\tau
	=
	2
	\int_{-\infty}^{\infty}
	\norm{
		\Re\left\{\widehat{u}_h(\alpha+y+\imath \tau)\right\}
	}^2_{H^1_0(\Omega)}
	\normalfont\text{d}\tau.
\end{equation}
Hence, recalling Proposition~\ref{prop:properties_hardy}
\begin{equation}
\begin{aligned}
	\norm{
		\widehat{u}_h
	}^2_{\mathscr{H}^2_\alpha(H^1_0(\Omega))}
	=
	&
	\int_{-\infty}^{\infty}
	\norm{
		\widehat{u}_h(\alpha+\imath \tau)}^2_{H^1_0(\Omega)
	} \normalfont\text{d}\tau
	\\
	=
	&
	2
	\int_{-\infty}^{\infty}
	\norm{
		\Re\left\{\widehat{u}_h(\alpha+\imath \tau)\right\}
	}^2_{H^1_0(\Omega)}
	\normalfont\text{d}\tau.
\end{aligned}
\end{equation}
The exact same analysis applies to 
$\widehat{u}_h - \mathsf{P}_{\IV_R} \widehat{u}_h$ instead of $\widehat{u}_h$, thus proving
the first claim. The second one follows straightforwardly 
from the first one.
\end{proof}

\subsection{Snapshot Selection}
\label{eq:snap_shot_selection}
As discussed in Section~\ref{sec:FRB}, the implementation
of the \rev{LT-MOR} method relies on solving Problem~\ref{pr:sdp} on a set
$\mathcal{P}_s$ of \rev{points} in the Laplace domain.
Indeed, in practice we construct the reduced space $\mathbb{V}^{\text{(rb)}}_{R,M}$
by solving\footnote{As opposed to $\mathbb{V}^{\text{(rb)}}_{R}$ in \eqref{eq:fPOD},
we use the notation $\mathbb{V}^{\text{(rb)}}_{R,M}$ to highlight the dependence on the number of snapshots
$M$.}
\begin{equation}\label{eq:compuation_RB_space_quad}
	\mathbb{V}^{\text{(rb)}}_{R,M}
	=
	\argmin_{
	\substack{
		\mathbb{V}_R \subset \mathbb{V}^\mathbb{C}_h
		\\
		\text{dim}(\mathbb{V}_R)\leq R	
	}
	}
    	\sum_{j=1}^{M} 
	\omega_j
	\norm{
		\Re\left\{
			\widehat{u}_h(s_j) 
		\right\}
		- 
		\mathsf{P}_{\mathbb{V}_R}
		\Re\left\{
			\widehat{u}_h(s_j) 
		\right\}
	}^2_{H^1_0(\Omega)},
\end{equation}
which is a computable approximation of \eqref{eq:min_laplace_domain}.

For the selection of the snapshots and weights,
we propose the following choice:
Given $\beta>0$ and $M \in \IN$ we set for $i = 1,\dots,M$
\begin{equation}\label{eq:snapshots}
	\omega_i
	=
	\frac{
		\pi \beta
	}{
		M\sin^2
		\left(
			\frac{
				\theta_i
			}{2}
		\right)
	},
	\quad
	s_i
	=
	\alpha
	+
	\imath
	\beta
	\cot
	\left(
		\frac{
			\theta_i
		}{2}
	\right),
	\quad
	\text{and}
	\quad
	\theta_i
	=
	\frac{2\pi}{M} i.
\end{equation}

	One can readily observe that the snapshot computed at
	$s_M$ produces computational issues as its imaginary
	part diverges. Herein, we discuss how to tackle this.
	To this end, set
	\begin{equation}
		\phi(\theta)
		=
		\frac{
			\beta
		}{
			2\sin^2
		\left(
			\frac{
				\theta
			}{2}
		\right)
		}
		\quad
		\text{and}
		\quad
		s(\theta)
		=
		\alpha
		+
		\imath
		\beta
		\cot
		\left(
			\frac{
			\theta
		}{2}
		\right),
		\quad
		\theta \in (0,2\pi).
	\end{equation}
	Let us consider the quantity
	\begin{equation}
	\begin{aligned}
		\phi(\theta)
		\norm{
			\widehat{u}_h(s(\theta)) 
			- 
			\mathsf{P}_{X_R}
			\widehat{u}_h(s(\theta)) 
		}^2_{H^1_0(\Omega)}
		\\
		&
		\hspace{-2cm}
		=
		\frac{\phi(\theta)}{\snorm{s(\theta)}^2}
		\norm{
			s(\theta)
			\widehat{u}_h(s(\theta)) 
			- 
			\mathsf{P}_{X_R}
			s(\theta)
			\widehat{u}_h(s(\theta)) 
		}^2_{H^1_0(\Omega)},
	\end{aligned}
	\end{equation}
	and define $\psi_{h}(s(\theta)) = s(\theta) \widehat{u}_h(s(\theta)) \in \mathbb{V}_h$, which clearly satisfies 
	\begin{equation}
	\begin{aligned}
		\dotp{\psi_{h}(s(\theta))}{v_h}_{L^2(\Omega)} 
		+
		\mathsf{a}
		\dotp{\widehat{u}_h(s(\theta))}{v_h}
		=
		&
		\dotp{
			\widehat{f}(s(\theta))
		}{v_h}_{L^2(\Omega)}
		\\
		&
		+
		\dotp{
			u_{0,h}
		}{v_h}_{L^2(\Omega)},
	\end{aligned}
	\end{equation}
	for any $0\neq v_h \in \mathbb{V}_h$. Therefore
	\begin{equation}
	\begin{aligned}
		\frac{
			\snorm{\dotp{\psi_{h}(s(\theta))-u_{0,h}}{v_h}_{L^2(\Omega)} }
		}{
			\norm{v_h}_{H^1_0(\Omega)}
		}
		&
		\leq
		\frac{
		\snorm{
			\dotp{
			\widehat{f}(s(\theta))
			}{v_h}_{L^2(\Omega)}
		}}{\norm{v_h}_{H^1_0(\Omega)}}
		+
		\frac{
		\snorm{
			\mathsf{a}
			\dotp{\widehat{u}_h(s(\theta))}{v_h}
		}}{\norm{v_h}_{H^1_0(\Omega)}}
		\\
		&
		\leq
		C_P(\Omega)
		\norm{\widehat{f}(s(\theta))}_{L^2(\Omega)}
		+
		{\overline{c}_{\bm{A}}}
		\norm{\widehat{u}_h(s(\theta))}_{H^1_0(\Omega)}.
	\end{aligned}
	\end{equation}
	Recalling Lemma~\ref{lmm:laplace_transform_u_bounds}, we obtain
	\begin{equation}\label{eq:bund_S_inf}
	\begin{aligned}
		\frac{
			\snorm{\dotp{\psi_{h}(s(\theta))-u_{0,h}}{v_h}_{L^2(\Omega)}}
		}{
			\norm{v_h}_{H^1_0(\Omega)}
		}
		\leq
		&
		C_P(\Omega)
		\left(
			1
			+
			\frac{{\overline{c}_{\bm{A}}}}{\underline{c}_{\bm{A}}}
		\right)
		\norm{
		\widehat{f}(s(\theta))
		}_{L^2(\Omega)}
		\\
		&
		+
		\frac{{\overline{c}_{\bm{A}}}}{\snorm{s(\theta)}}
		\left(
			1
			+
			\frac{{\overline{c}_{\bm{A}}}}{\underline{c}_{\bm{A}}}
		\right)
		\norm{
			u_{0,h}
		}_{L^2(\Omega)}.
	\end{aligned}
	\end{equation}
	Since $\widehat{f} \in \mathscr{H}^2(\Pi_{\alpha_0};L^2(\Omega))$,
	one has $\snorm{\widehat{f}(s)} \rightarrow 0 $ as $\snorm{s}\rightarrow \infty$,
	and $\snorm{s(\theta)} \rightarrow \infty$ as $\theta \rightarrow 2\pi$. 
	Therefore, from \eqref{eq:bund_S_inf} we conclude that
	$\psi_{h}(s(\theta))\rightarrow u_{0,h} $ in $L^2(\Omega)$ as $\theta \rightarrow 2\pi$, 
	and since $u_{0,h} \in \mathbb{V}_h \subset H^1_0(\Omega)$, we have that
	$\psi_{h}(s(\theta))\rightarrow u_{0,h} $ in $H^1_0(\Omega)$.
	
	Finally, we calculate
	\begin{equation}
		\lim_{\theta \rightarrow 2\pi}
			\frac{
				\phi(\theta)
			}{
				\snorm{s(\theta)}^2
			}	
		=
		\lim_{\theta \rightarrow 2\pi}
		\frac{	
		\beta
		}{2\sin^2
		\left(
			\frac{\theta}{2}
		\right)\snorm{\alpha + \imath \beta \cot\left(\frac{\theta}{2}\right)}^2}	
		=
		\frac{1}{2\beta}.
	\end{equation}
	Therefore, in \eqref{eq:compuation_RB_space_quad} we replace
	the last term in the sum by
	\begin{equation}
		\frac{\pi}{M\beta}
		\norm{
			u_{0,h}
			- 
			\mathsf{P}_{X_R}
			u_{0,h}
		}^2_{H^1_0(\Omega)}.
	\end{equation}
	In other words, we set in \eqref{eq:compuation_RB_space_quad}
	$\omega_M = \frac{\pi}{M\beta}$ and $\widehat{u}_h(s_M) = u_{0,h} \in \mathbb{V}_h$.

For any finite dimensional subspace of $X_R \subset H^1_0(\Omega)$
of dimension $R\in \IN$, define 
\begin{align}\label{eq:discrete_error_measure}
	\varepsilon^{(M)}_{\alpha,\beta}
	\left(
		X_R
	\right)
	\coloneqq
    	\sum_{j=1}^{M} 
	\omega_j
	\norm{
		\Re\left\{\widehat{u}_h(s_j) \right\}
		- 
		\mathsf{P}_{X_R}
		\Re\left\{\widehat{u}_h(s_j) \right\}
	}^2_{H^1_0(\Omega)},
\end{align}
and
\begin{align}\label{eq:exact_error}
	\rev{\varepsilon_{\alpha}}
	\left(
		X_R
	\right)
	\coloneqq
	\int_{-\infty}^{+\infty}
	\,
	\norm{
		\Re\left\{\widehat{u}_h(\alpha + \imath \tau)\right\}
		- 
		\mathsf{P}_{X_R}
		\Re\left\{\widehat{u}_h(\alpha + \imath \tau) \right\}
	}^2_{H^1_0(\Omega)}
	\text{d}
	\tau.
\end{align}

Even though is not stated explicitly, the dependence of 
$\varepsilon^{(M)}_{\alpha,\beta}(X_R)$ and $\rev{\varepsilon_{\alpha}}(X_R)$
on $\alpha,\beta$ comes from the definition of $\left\{(\omega_i,s_i)\right\}_{i=1}^{M}$
in \eqref{eq:snapshots}. 

Let us define
\begin{equation}
	\mathcal{E}_{\eta,\alpha,\beta}
	\coloneqq
	\left \{
		\tau \in \mathbb{C}: \;
		\alpha + \imath \tau \in \mathcal{C}_{\eta,\alpha,\beta}
	\right\}
\end{equation}
and its complement $\mathcal{E}^c_{\eta,\alpha,\beta} \coloneqq \mathbb{C} \backslash \overline{\mathcal{E}_{\eta,\alpha,\beta}}$.
\rev{
\begin{lemma}\label{lmm:holomorphic_extensions}
Let Assumption~\ref{assump:data_f} be satisfied and
let $u_h \in \mathcal{W}_{\rev{\alpha_0}}(\IR_+;\mathbb{V}_h)$
for some $\alpha_0>0$ be the solution to Problem~\ref{prbm:semi_discrete_problem}.

Then, for any $\alpha>\alpha_0$ and any $\beta > 0$,
there exists $\eta_{\alpha,\beta}>1$
such that for any $\eta \in (1,\eta_{\alpha,\beta})$ the map
\begin{equation}
	\IR
	\ni
	\tau
	\mapsto
	\kappa_{\alpha}(\tau)
	\coloneqq
	\norm{
		(\alpha+\imath \tau)
		\left(
      		\widehat{u}_h
      		(\alpha + \imath \tau)
      		-
      		\mathsf{P}_{X_R} 
      		\widehat{u}_h
      		(\alpha + \imath \tau)
		\right)
	}^2_{H^1_0(\Omega)}
\end{equation}
admits a unique bounded analytic extension to
$\mathcal{E}^c_{1/\eta,\alpha,\beta} \cap \mathcal{E}^c_{\eta,\alpha,\beta}$
bounded according to
\begin{equation}\label{eq:bound_hol_ext}
	\snorm{
		\kappa_{\alpha}(\tau)
	}
	\leq
	\norm{(\alpha + \imath \tau)\widehat{u}_h(\alpha + \imath \tau)}_{H^1_0(\Omega)}
	\norm{(\alpha - \imath \tau)\widehat{u}_h(\alpha - \imath \tau)}_{H^1_0(\Omega)},
\end{equation}
for $\tau \in \mathcal{E}^c_{1/\eta,\alpha,\beta} \cap \mathcal{E}^c_{\eta,\alpha,\beta}$.
\end{lemma}

\begin{proof}
Since $\widehat{u}_h \in \mathscr{H}^2_{\alpha_0}(\mathbb{V}^\mathbb{C}_h)$ the map
$\Pi_\alpha \ni s \mapsto \widehat{u}_h(s)  \in H^1_0(\Omega)$ is analytic, therefore
straightforwardly the map $\Pi_\alpha \ni s \mapsto \widehat{u}_h (s) -\mathsf{P}_{X_R} \widehat{u}_h
(s) \in H^1_0(\Omega)$ is so as well.  
As in the proof of Lemma~\ref{eq:approximation_Hardy_spaces} and recalling
item (i) and (ii) in Assumption~\ref{assump:data_f},
we may conclude that for any $\alpha>\alpha_0$ and $\beta$ there exists
$\eta_{\alpha,\beta}>1$ such that for any $\eta \in (1,\eta_{\alpha,\beta})$
the map $\Pi_\alpha \ni s \mapsto \widehat{u}_h (s) -\mathsf{P}_{X_R} \widehat{u}_h
(s) \in H^1_0(\Omega)$ admits a holomorphic extension to $\mathcal{C}^c_{\eta,\alpha,\beta}$, i.e.,
the complement of $\mathcal{C}_{\eta,\alpha,\beta}$.

Next, for each $s \in \mathcal{C}_{\eta,\alpha,\beta}$
\begin{equation}
\begin{aligned}
	\norm{
		s
		\left(
  		\widehat{u}_h(s)
  		- 
  		\mathsf{P}_{X_R}
  		\widehat{u}_h(s)
		\right)
	}^2_{H^1_0(\Omega)}
	&
	=
	\dotp{
		s
		\left(
		\widehat{u}_h(s)
		- 
		\mathsf{P}_{X_R}
		\widehat{u}_h(s)
		\right)
	}{
		s
		\left(
		\widehat{u}_h(s)
		-
		\mathsf{P}_{X_R} 
		\widehat{u}_h(s)
		\right)
	}_{H^1_0(\Omega)}
	\\
	&
	=
	\dotp{
		s
		\left(
		\widehat{u}_h(s)
		-
		\mathsf{P}_{X_R} 
		\widehat{u}_h(s)
		\right)
	}{
		\overline{s}
		\left(
		\overline{\widehat{u}_h(\overline{s})}
		-
		\mathsf{P}_{X_R} 
		\overline{\widehat{u}_h(\overline{s})}
		\right)
	}_{H^1_0(\Omega)},
\end{aligned}
\end{equation}
where we have used that $\widehat{u}_h(s) = \overline{\widehat{u}_h(\overline{s})}$
since $\widehat{u}_h$ is the Laplace transform of ${u}_h \in L^2_{\alpha_0}(\mathbb{R}_+;\mathbb{V}_h)$,
which is real-valued provided that the problem's data is real-valued as is the case. 

For $s = \alpha+\imath \tau$
\begin{equation}
	\norm{
		s
		\left(
      		\widehat{u}_h(s)
      		-
      		\mathsf{P}_{X_R} 
      		\widehat{u}_h(s)
		\right)
	}^2_{H^1_0(\Omega)}
	=
	\dotp{
		g_{1,\alpha}(\tau)
	}{
		\overline{g_{2,\alpha}(\tau)}
	}_{H^1_0(\Omega)},
\end{equation}
where 
\begin{equation}
\begin{aligned}
	g_{1,\alpha}(\tau)
	&
	\coloneqq
	(\alpha+\imath \tau)
	\left(
      	\widehat{u}_h(\alpha+\imath \tau)
      	-
      	\mathsf{P}_{X_R} 
      	\widehat{u}_h(\alpha+\imath \tau)
	\right),
	\quad
	\text{and},\\
	g_{2,\alpha}(\tau)
	&
	\coloneqq
	(\alpha-\imath \tau)
	\left(
      	{\widehat{u}_h(\alpha-\imath \tau)}
      	-
      	\mathsf{P}_{X_R} 
      	{\widehat{u}_h(\alpha-\imath \tau)}
	\right).
\end{aligned}
\end{equation}
The maps $$\IR\ni \tau \mapsto g_{1,\alpha}(\tau),g_{1,\alpha}(\tau) \in H^1_0(\Omega)$$
admit bounded holomorphic extensions to $\mathcal{E}^c _{\eta,\alpha,\beta}$
and $\mathcal{E}^c_{1/\eta,\alpha,\beta}$, respectively, as a consequence of the fact
that $\Pi_\alpha \ni s \mapsto \widehat{u}_h (s) -\mathsf{P}_{X_R} \widehat{u}_h
(s) \in H^1_0(\Omega)$ admits a holomorphic extension to $\mathcal{C}^c_{\eta,\alpha,\beta}$, 
as argued previously, and the fact that $
	\mathcal{E}_{1/\eta,\alpha,\beta}
	=
	\left \{
		\tau \in \mathbb{C}: \;
		\alpha - \imath \tau \in \mathcal{C}_{\eta,\alpha,\beta}
	\right\}$, which follows from the fact that 
$\mathcal{C}_{\eta,\alpha,\beta}$ and $\mathcal{C}_{1/\eta,\alpha,\beta}$
are mirror images of each other with respect to the vertical line $\Re\{s\} = \alpha$. 

In addition, the map $H^1_0(\Omega) \times H^1_0(\Omega) \ni (u,v)
\mapsto (u,\overline{v})_{H^1_0(\Omega)} \in \mathbb{C}$
is bilinear, i.e., linear in each component, therefore holomorphic as well. 
Thus, we may conclude that
\begin{equation}
	\IR\ni \tau 
	\mapsto 
	\dotp{g_{1,\alpha}(\tau))}{\overline{g_{2,\alpha}(\tau)}}_{H^1_0(\Omega)}
	\in
	\mathbb{R_+}
\end{equation}
admits a (unique) holomorphic extension to $\mathcal{E}^c_{1/\eta,\alpha,\beta} \cap \mathcal{E}^c_{\eta,\alpha,\beta}$,
which is bounded according to \eqref{eq:bound_hol_ext}.
\end{proof}

\begin{remark}\label{eq:extension_holmorphic_Re}
Define for $\tau \in \mathbb{R}$
\begin{equation}
\begin{aligned}
	g_{1,\Re}(\tau)
	&
	\coloneqq
	\Re
	\left\{
	\widehat{u}_h(\alpha+\imath \tau)
	\right\}
	-
	\mathsf{P}_{X_R} 
	\Re
	\left\{
	\widehat{u}_h(\alpha+\imath \tau)
	\right\}
	\quad
	\text{and},\\
	g_{2,\Re}(\tau)
	&
	\coloneqq
	\Re
	\left\{
	\widehat{u}_h(\alpha-\imath \tau)
	\right\}
	-
	\mathsf{P}_{X_R} 
	\Re
	\left\{
	\widehat{u}_h(\alpha-\imath \tau)
	\right\}.
\end{aligned}
\end{equation}
The maps 
\begin{equation}
	\IR\ni \tau \mapsto g_{1,\alpha}(\tau),g_{1,\alpha}(\tau) \in \mathbb{V}_h
\end{equation}
also admit bounded holomorphic extensions to $\mathcal{E}^c _{\eta,\alpha,\beta}$
and $\mathcal{E}^c_{1/\eta,\alpha,\beta}$, respectively. 
Using $\widehat{u}_h(s) = \overline{\widehat{u}_h(\overline{s})}$, for $\tau \in \mathbb{R}$ one has
\begin{equation}
\begin{aligned}
	\Re
	\left\{
	\widehat{u}_h(\alpha \pm \imath \tau)
	\right\}
	&
	=
	\frac{\widehat{u}_h(\alpha\pm\imath \tau) +\overline{\widehat{u}_h(\alpha\pm\imath \tau)}}{2}
	\\
	&
	=
	\frac{\widehat{u}_h(\alpha\pm\imath \tau) +{\widehat{u}_h(\alpha\mp\imath \tau)}}{2},
\end{aligned}
\end{equation}
which admits a bounded holomorphic to $\mathcal{E}^c_{1/\eta,\alpha,\beta} \cap \mathcal{E}^c_{\eta,\alpha,\beta}$.
\end{remark}

For $\eta>1$, let us define
\begin{equation}
	\mathcal{A}_\eta
	\coloneqq
	\left\{
		z \in \mathbb{C}: \;
		\eta^{-1} < \snorm{z} <\eta
	\right\}.
\end{equation}

Equipped with the properties established in Remark~\ref{lmm:holomorphic_extensions}
we can state the following result.

\begin{lemma}\label{lmm:exponential_convergence_rule}
Let Assumption~\ref{assump:data_f} be satisfied and
let $u_h \in \mathcal{W}_{\rev{\alpha_0}}(\IR_+;\mathbb{V}_h)$
for some $\rev{\alpha_0 \geq 0}$
be the solution to Problem~\ref{prbm:semi_discrete_problem}.

Then, for any $\alpha>\alpha_0$ and $\beta > 0$ such that $\alpha \neq \beta$,
there exists $\eta_{\alpha,\beta}>1$ (depending on $\alpha,\beta$)
such that for $M \in \IN$ and any $\eta \in (1,\eta_{\alpha,\beta})$
it holds that
\begin{equation}
\begin{aligned}
	\snorm{
		\varepsilon^{(M)}_{\alpha,\beta}
		\left(
			X_R
		\right)
		-
		\varepsilon_{\alpha}
		\left(
			X_R
		\right)
	}
	\lesssim
	\frac{\Gamma }{\eta^{M}-1},
\end{aligned}
\end{equation}
where
\begin{equation}\label{eq:gamma}
\begin{aligned}
	\Gamma
	\coloneqq
	\Upsilon_{\eta,\alpha,\beta}
	&
	\left(
	\sup_{ s \in \mathcal{C}^c_{\eta,\alpha,\beta}\cap \mathcal{C}^c_{1/\eta,\alpha,\beta}}
	\norm{
		\Pi_h (\mathcal{L}\{\partial_t f\}(s))
	}^2_{H^1_0(\Omega)}
	\right.
	\\
	&
	\left.
	\sup_{ s \in \mathcal{C}^c_{\eta,\alpha,\beta} \cap \mathcal{C}^c_{1/\eta,\alpha,\beta} }
	\norm{
		\Pi_h (\mathcal{L}\{\partial_t f\}(\overline{s}))
	}^2_{H^1_0(\Omega)}
	\right.
	\\
	&
	\left.
	+
	\norm{u_{0}}^2_{H^1_0(\Omega)}
	+
	\norm{\Pi_h (f(0))}^2_{H^1_0(\Omega)}
	+
	\norm{\Pi_h ( \nabla \cdot( \bm{A} \nabla u_{0}))}^2_{H^1_0 (\Omega)}
	\right),
\end{aligned}
\end{equation}
where $\Lambda>0$ is as in \eqref{eq:lambda} and
\begin{equation}
	\Upsilon_{\eta,\alpha,\beta}
	\coloneqq
	\frac{\beta \eta}{\Lambda(\alpha-\beta)^2\left(1-\snorm{\frac{\alpha+\beta}{\alpha-\beta}}\right)^2}.
\end{equation}
\end{lemma}

 \begin{proof}
Set $\tau = \beta\cot(\theta)$, then
\begin{equation}
 \begin{aligned}
 	\varepsilon_{\alpha}
 	\left(
 		X_R
 	\right)
 	&
 	=
 	\int_{-\infty}^{+\infty}
 	\,
 	\norm{
 		\Re\left\{\widehat{u}_h(\alpha + \imath \tau)\right\}
 		- 
 		\mathsf{P}_{X_R}
 		\Re\left\{\widehat{u}_h(\alpha + \imath \tau) \right\}
 	}^2_{H^1_0(\Omega)}
 	\text{d}
 	\tau.
 	\\
 	&
 	=
 	\int_{0}^{2\pi}
 	q(\theta)
 	\text{d}
 	\theta,
 \end{aligned}
 \end{equation}
 where for $\theta \in (0,2\pi)$ and $s(\theta) = \alpha +\imath \beta \cot\left(\frac{\theta}{2}\right)$ 
 \begin{equation}
 \begin{aligned}
 	q(\theta)
 	=
 	\frac{\beta}{2\sin^2
 		\left(
 			\frac{\theta}{2}
 	\right)\snorm{s(\theta)}^2}
 	\norm{
 		s(\theta)
 		\Re\left\{
 		\widehat{u}_h
 		\left(s(\theta)\right)
 		\right\}
 		-
 		\mathsf{P}_{X_R} 
		s(\theta)
 		\Re\left\{
 		\widehat{u}_h
 		\left(s(\theta)\right)
 		\right\}
 	}^2_{H^1_0(\Omega)}
 \end{aligned}
 \end{equation}
 
 Let us define for each $z \in \mathcal{D}$
 \begin{equation}
 	h(z)
	\coloneqq
	w(z)
 	\norm{
		g(z)
 		-
 		\mathsf{P}_{X_R} 
		g(z)
 	}^2_{H^1_0(\Omega)},
\end{equation}
where $g(z)\coloneqq \mathcal{M}^{-1}(z)\Re\!\left\{\widehat{u}\!\left(\mathcal{M}^{-1}(z)\right)\right\}$
and $w(z) \coloneqq-\frac{\beta z}{(\alpha-\beta)^2\left(z - \frac{\alpha+\beta}{\alpha-\beta}\right)^2}$, $z \in \mathcal{D}$.

Observe that $w(z)$ admits a holomorphic extension to $\mathcal{D}_\eta$
for any $\eta \in \left(1,\snorm{\frac{\alpha+\beta}{\alpha-\beta}}\right)$.
Arguing as in the proof of Lemma~\ref{eq:approximation_Hardy_spaces},
together with Lemma~\ref{lmm:holomorphic_extensions} and Remark~\ref{eq:extension_holmorphic_Re},
we may conclude that there exist $\eta_{\alpha,\beta}>1$
(depending on $\alpha,\beta$) such that for any $\eta \in (1,\eta_{\alpha,\beta})$, $h(z)$
admits a holomorphic extension to $\mathcal{A}_{\eta}$, which is bounded according to
\begin{equation}
\begin{aligned}
 	\snorm{h(z)}
 	\leq
	&
	\frac{\beta \eta}{(\alpha-\beta)^2\left(1-\snorm{\frac{\alpha+\beta}{\alpha-\beta}}\right)^2}
	\left(
	\sup_{s \in  \mathcal{C}^c_{\eta,\alpha,\beta} \cap \mathcal{C}^c_{1/\eta,\alpha,\beta}}
	\norm{s\widehat{u}_h(s)}_{H^1_0(\Omega)}^2
	\right.
	\\
	&
	\left.
	+
	\sup_{s \in  \mathcal{C}^c_{\eta,\alpha,\beta} \cap \mathcal{C}^c_{1/\eta,\alpha,\beta}}
	\norm{\overline{s}\widehat{u}_h(\overline{s})}_{H^1_0(\Omega)}^2
	\right).
\end{aligned}
 \end{equation}

 Recalling Theorem 2.2 in Ref.~\refcite{trefethen2014exponentially}, for any $\eta \in (1,\eta_{\alpha,\beta})$
 it holds that
 \begin{equation}
 \begin{aligned}
 	\snorm{
 		\varepsilon^{(M)}_{\alpha,\beta}
 		\left(
 			X_R
 		\right)
 		-
 		\varepsilon_{\alpha,\beta}
 		\left(
 			X_R
 		\right)
 	}
 	=
 	\snorm{
 		\frac{2\pi}{M}
 		\sum_{k=1}^{M}
 		h(z_k)
 		-
 		\int_{0}^{2\pi}
 		g(\theta)
 		\text{d}
 		\theta
 	}
 	\lesssim
 	\frac{\Gamma}{\eta^{M}-1},
 \end{aligned}
 \end{equation}
 with $z_k = \exp\left({2\pi\imath\frac{k}{M}}\right)$, $k=1,\dots,M$,
 and $\Gamma$ as in \eqref{eq:gamma}.
\end{proof}
}

\subsection{Fully Discrete Error Analysis}
The result stated in Theorem~\ref{eq:convergence_lt_rb_method_full} assumes that
$\IV^{\normalfont\text{(rb)}}_{R}$ can be computed. The next result, takes into 
account the snapshot sampling.

\begin{theorem}[Fully Discrete Exponential Convergence of the LT-MOR]
\label{eq:convergence_lt_rb_method_full}
Let Assumption~\ref{assump:data_f} be satisfied with $\alpha_0>0$.
Furthermore, let $u_h \in \mathcal{W}_{\rev{\alpha_0}}(\IR_+;\mathbb{V}_h)$
be the solution to Problem~\ref{prbm:semi_discrete_problem},
and let $u^{\normalfont\text{(rb)}}_{R,M}\in \mathcal{W}_{\rev{\alpha_0}}(\IR_+;\IV^{\normalfont\text{(rb)}}_{R,M})$
be the solution to Problem~\ref{pr:sdpr} with 
$\IV^{\normalfont\text{(rb)}}_{R,M}$ be as in \eqref{eq:compuation_RB_space_quad}.

Then, for any $\alpha>\alpha_0$ and any $\beta > 0$ such that $\beta \neq \alpha$
there exists $\eta_{\alpha,\beta}>1$
such that for $\eta \in (1,\eta_{\alpha,\beta})$,
$R \in \{1,\dots,N_h\}$ and $M \in \mathbb{N}$ it holds that
\begin{equation}
\begin{aligned}
	\norm{
		u_h - u^{\normalfont\text{(rb)}}_{R,M}
	}_{L^2_\alpha\left(\IR_+;H^1_0(\Omega)\right)}
	\lesssim
    &
	\frac{ C^{\normalfont\text{inv}}_h \eta^2}{(\eta-1)\sqrt{\alpha \Lambda}}
	\left(
	\sup_{ s \in \mathcal{C}_{\eta,\alpha, \beta}}
	\norm{
		\Pi_h (\mathcal{L}\{\partial_t f\}(s))
	}_{H^1_0(\Omega)}
	\right.
	\\
	&
	\left.
	+
	\norm{\Pi_h (f(0))}_{H^1_0(\Omega)}
	+
	\norm{\Pi_h ( \nabla \cdot( \bm{A} \nabla u_{0}))}_{H^1_0 (\Omega)}
	\right)
	\eta^{-R}
	\\
	&
	+
	\sqrt{\frac{\Gamma }{\eta^{M}-1}}
    +
    C^{\normalfont\text{inv}}_h
    \norm{
    \left(
		\normalfont\text{Id}
		-
		\mathsf{P}^{\text{(rb)}}_{R,M}
	\right)
    u_{0,h}
    }_{H^1_0(\Omega)},
\end{aligned}
\end{equation}
where the implicit constant is independent of the discretization parameter $h>0$
and $\Lambda>0$ is as in \eqref{eq:lambda} and
$\mathsf{P}^{\normalfont\text{(rb)}}_{R,M}: \mathbb{V}_h \rightarrow \IV^{\normalfont\text{(rb)}}_{R,M}$
corresponds to the $H^1_0(\Omega)$-projection operator onto $\IV^{\normalfont\text{(rb)}}_{R,M}$.
\end{theorem}

\begin{proof}
Exactly as in the proof of Theorem~\ref{eq:convergence_lt_rb_method},
we have that
\begin{equation}
\begin{aligned}
    \norm{
		u_h - u^{\normalfont\text{(rb)}}_{R,M}
	}_{L^2_\alpha\left(\IR_+;H^1_0(\Omega)\right)}
	\lesssim
    &
    (1+C^{\normalfont\text{inv}}_h)
    \norm{
		\left(
			\text{Id}
			-
			\mathsf{P}^{\text{(rb)}}_{R,M}
		\right)
		  u_h
    }_{L^2_\alpha(\mathbb{R}_+;H^1_0(\Omega))}
    \\
    &
    +
    C^{\normalfont\text{inv}}_h
    \norm{
    \left(
		\text{Id}
		-
		\mathsf{P}^{\text{(rb)}}_{R,M}
	\right)
    u_{0,h}
    }_{H^1_0(\Omega)}.
\end{aligned}
\end{equation}
Recalling Lemma~\ref{lmm:real_basis} and the Paley-Wiener theorem, we obtain
\begin{equation}
\begin{aligned}
    \norm{
		u_h - u^{\normalfont\text{(rb)}}_{R,M}
	}^2_{L^2_\alpha\left(\IR_+;H^1_0(\Omega)\right)}
	\lesssim
    &
    (1+C^{\normalfont\text{inv}}_h)^2
	\varepsilon_{\alpha}
	\left(
		\IV^{\normalfont\text{(rb)}}_{R,M}
	\right)
    \\
    &
    +
    \left(C^{\normalfont\text{inv}}_h\right)^2
    \norm{
    \left(
		\text{Id}
		-
		\mathsf{P}^{\text{(rb)}}_{R,M}
	\right)
    u_{0,h}
    }_{H^1_0(\Omega)}.
\end{aligned}
\end{equation}
Furthermore, recalling Lemma~\ref{lmm:exponential_convergence_rule}
\begin{equation}
\begin{aligned}
	\varepsilon_{\alpha}
	\left(
		\IV^{\normalfont\text{(rb)}}_{R,M}
	\right)
	\leq
	&
	\snorm{
      	\varepsilon_{\alpha}
      	\left(
      		\IV^{\normalfont\text{(rb)}}_{R,M}
      	\right)
		-
		\varepsilon^{(M)}_{\alpha,\beta}
		\left(
			\IV^{\normalfont\text{(rb)}}_{R,M}
		\right)
	}
	+
	\varepsilon^{(M)}_{\alpha,\beta}
	\left(
		\IV^{\normalfont\text{(rb)}}_{R,M}
	\right)
	\\
	\lesssim
	&
	\frac{\Gamma }{\eta^{M}-1}
	+
	\varepsilon^{(M)}_{\alpha,\beta}
	\left(
		\IV^{\normalfont\text{(rb)}}_{R,M}
	\right),
\end{aligned}
\end{equation}
where $\Gamma$ is as in \eqref{eq:gamma}.
Recalling that $\IV^{\normalfont\text{(rb)}}_{R,M}$ minimizes
$\varepsilon^{(M)}_{\alpha,\beta}\left(\cdot\right)$ defined in \eqref{eq:discrete_error_measure},
we have that
\begin{equation}
\begin{aligned}
	\varepsilon^{(M)}_{\alpha,\beta}\left(\IV^{\normalfont\text{(rb)}}_{R,M}\right)
	&
	\leq
	\varepsilon^{(M)}_{\alpha,\beta}\left(\IV^{\normalfont\text{(rb)}}_{R}\right)
	\\
	&
	\leq
	\snorm{
		\varepsilon^{(M)}_{\alpha,\beta}\left(\IV^{\normalfont\text{(rb)}}_{R}\right)
		-
		\varepsilon_{\alpha}\left(\IV^{\normalfont\text{(rb)}}_{R}\right)
	}
	+
	\varepsilon_{\alpha}\left(\IV^{\normalfont\text{(rb)}}_{R}\right)
	\\
	&
	\lesssim
	\frac{\Gamma }{\eta^{M}-1}
	+
	\varepsilon_{\alpha}
	\left(
		\IV^{\normalfont\text{(rb)}}_{R}
	\right),
\end{aligned}
\end{equation}
where we have used again Lemma~\ref{lmm:exponential_convergence_rule}.
Finally, we bound
$\varepsilon_{\alpha}\left(\IV^{\normalfont\text{(rb)}}_{R}\right)$ using Lemma~\ref{eq:approximation_Hardy_spaces}
and obtain the final result.
\end{proof}

\subsection{Halving the Number of Snapshots}
\label{sec:halving_snapshots}
As pointed out in Remark 1 of Ref.~\refcite{Bigoni2020a}, the total 
number of snapshots can be halved.
Let  $u_h \in L^2_\alpha(\IR_+,H^1_0(\Omega;\mathbb{R}))$ be the solution to 
Problem~\ref{pbm:wave_equation}, and let $\widehat{u}_h = \mathcal{L}\{{u}_h \} \in \mathscr{H}^2(\Pi_{\alpha};H^1_0(\Omega))$
be its Laplace transform.

For $i=1,\dots,M$ one can easily verify that $\overline{s_i} =s_{M+1-i}$.
The application of complex conjugation to Problem~\ref{pbrm:laplace_discrete}
yields the following problem for $i=1,\dots,M$
\begin{align}
	\overline{s_i}
	\dotp{\overline{\widehat{u}_h(s_i)}}{v_h}_{L^2(\Omega)} 
	+ 
	\mathsf{a}
	\left(
		\overline{\widehat{u}_h(s_i)},v_h
	\right) 
	= 
	\dotp{\overline{\widehat{f}(s_i)}}{v_h}_{L^2(\Omega)}
	+
	\dotp{u_{0,h}}{v_h}_{L^2(\Omega)} , 
\end{align}
where $\widehat{f}(s) =\mathcal{L}\{f\}(s)$ corresponds to the Laplace transform
of $f\in L^2_\alpha(\IR_+;L^2(\Omega;\mathbb{R}))$, and we have used that $u_0 \in L^2(\Omega;\mathbb{R})$.
Recalling the definition of the Laplace transform, one can readily verify that
$\overline{\widehat{f}(s)} = {\widehat{f}(\overline{s})}$.
Consequently, ${\widehat{u}_h(s_i)} =\overline{\widehat{u}_h(s_{M+1-i})} $
for $i=1,\dots,M$.
Assuming that $M \in \IN$ is even, we can readily 
notice that one only \rev{needs to compute}
the snapshots $u(s_i)$, $i=1,\dots,\frac{M}{2}$,
as $u({s_i}) = {u({s_{M+1-i}})}$, for $i=\frac{M}{2}+1,\dots,M$.

In addition, since $\Re\left\{u(s_i)\right\} = \Re\left\{\overline{u({s_{M+1-i}})}\right\}$,
and in view of Proposition~\ref{prop:titchmarsh}, we do not need to include 
the samples computed for $i=\frac{M+1}{2}+1,\dots,M$ in the construction 
of the snapshot matrix introduced in \eqref{eq:snapshot_M_LT_RB}.

\section{Numerical Results}
\label{sec:numerical_results}
We present numerical results validating our theoretical claims
and portraying the computational advantages of the \rev{LT-MOR} method
over traditional methods to solve the linear, second-order parabolic problem
in bounded domains as described in Section~\ref{sec:parabolic_problems}.

We are \rev{interested} in assessing the performance of the \rev{LT-MOR} method
in three aspects: 
\begin{itemize}
	\item[(i)]
	Accuracy with respect to the high-fidelity solution.
	To this end, we set $\mathfrak{J} = (0,T)$ and 
	consider the following metric
	\begin{equation}\label{sec:example_I}
	\begin{aligned}
		\text{Rel\_Error}^{\normalfont\text{(rb)}}_R(\mathfrak{J};X)
		&
		=
		\frac{
			\norm{
				u_h - u^{\normalfont\text{(rb)}}_R
			}_{L^2(\mathfrak{J};X)}
		}{
			\norm{
				u_h
			}_{L^2(\mathfrak{J};X)}
		}
		\\
		&
		\approx
		\frac{
			\left(
			\displaystyle\sum_{j=0}^{N_t}
			\norm{
				u_h(t_j) - u^{\normalfont\text{(rb)}}_R(t_j)
			}^2_{X}
			\right)^{\frac{1}{2}}
		}{
			\left(
			\displaystyle\sum_{j=0}^{N_t}
			\norm{
				u_h(t_j)
			}^2_{X}
			\right)^{\frac{1}{2}}
		},
	\end{aligned}
	\end{equation}
	where $X \in \{L^2(\Omega),H^1_0(\Omega)\}$, i.e.~we
	compute (an approximation of) the $L^2(\mathfrak{J};X)$-relative error for a number of
	reduced spaces of dimension ranging from $R \in \left\{1,\dots,R_{\text{max}}\right\}$.
	\item[(ii)]
	Accuracy with respect to the number of snapshots in the offline phase, i.e.~the 
	number of samples in the Laplace domain.
	\item[(iii)]
	Speed-up with respect to the high-fidelity solver.
\end{itemize}

\subsection{High-Fidelity Solver: Finite Element Discretization}
Let $\Omega \subset \IR^d$, $d\in \{2,3\}$ be a bounded Lipschitz polygon/polyhedron
with boundary $\partial \Omega$. We consider a conforming,
uniformly shape regular triangulation $\mathcal{T}_h$ of $\Omega$ with mesh-size $h>0$.
We consider the space $\mathcal{S}^{p,1}_0(\mathcal{T}_h)$ of $H^1_0(\Omega)$-conforming
Lagrangian finite element space of order $p$, which plays the role of $\mathbb{V}_h$ in \eqref{ssec:fe_problem}.
The finite element implementation is conducted in the {\tt MATLAB}
library {\tt Gypsilab}\cite{alouges2018fem}.

\subsection{Problem Set-up}
\label{sec:setting}
We consider the domain $\Omega = (-\frac{1}{2},\frac{1}{2})^d$, 
for $d \in \{2,3\}$, i.e.~the unit square or unit cube centred in the origin
of the cartesian coordinate system.

Two initial conditions are used
\begin{subequations}
\begin{align}
	u^{(1)}_{0}
	(\bm{x} )
	&
	=
	0,
	\quad
	\text{and},
	\label{eq:ic_1}\\
	u^{(2)}_{0}
	(\bm{x})
	&
	=
	\prod_{i=1}^{d}
	\sin\left(\zeta_i\pi\left(x_i-\frac{1}{2}\right)\right),
	\quad
	\bm{x} = (x_1,\dots,x_d) \in \Omega,
	\label{eq:ic_2}
\end{align}
\end{subequations}
where $\zeta = (\zeta_i)_{i=1}^{d}\subset \mathbb{R}^d$.
Observe that $u^{(1)}_{0}$ and $u^{(2)}_{0}$ belong to $H^1_0(\Omega)$.

We consider a separable forcing term
of the form $f(\bm{x},t) = b(t)g(\bm{x})$ with
\begin{equation}\label{eq:forcing_term}
\begin{aligned}
	b(t)
	&
	=	
	\left(
		\vartheta_1
		\sin(\omega t)
		+
		\vartheta_2
		\cos(\omega t)
	\right)
	\exp(\nu t)
	\\
	g(\bm{x})
	&
	=
	\cos(\lambda x_1)x_{d-1}(1+x_d)^2,
	\quad
	\bm{x} = (x_1,\dots,x_d) \in \Omega,
\end{aligned}
\end{equation}
where $\vartheta_1,\vartheta_2,\omega,\nu$ and $\lambda$ are
hyper-parameters of the problem.

Observe that for $s \in \Pi_\nu$
\begin{equation}
	\widehat{b}(s)
	=
	\mathcal{L}
	\left\{
		{b}(t)
	\right\}
	=
	\vartheta_1
	\frac{\omega}{(s-\nu)^2+\omega^2}
	+
	\vartheta_2
	\frac{s-\nu}{(s-\nu)^2+\omega^2}.
\end{equation}

\subsection{Results in the Unit Square}\label{sec:results_square}
We compute the solution to the linear, second-order parabolic
problem in a square with the following set-up.
\begin{itemize}
	\item[(i)] {\bf FE Discretization.}
	We consider a FE discretization using $\mathbb{P}^2$
	elements on a mesh $\mathcal{T}_h$ of $20000$ triangles, with a total number of degrees of freedom equal to $39601$,
	i.e. $\text{dim}\left(\mathcal{S}^{2,1}_0\left(\mathcal{T}_h\right)\right)=39601$
	and a mesh size $h = 1.41 \times 10^{-2}$.
	\item[(ii)] {\bf Construction of the Reduced Space.}
	The space $\mathbb{V}^{\text{(rb)}}_R \subset \mathbb{V}_h$ is computed following the computations described in 
	Section~\ref{sec:FRB}, in particular following the considerations of \eqref{rmk:comp_Vrb}, together
	with the choice of snapshots described in \eqref{eq:snap_shot_selection}.
	As per the setting for the computation of the snapshots, we set $\alpha = 1$ and $\beta = 2$ 
	in \eqref{eq:snapshots} and consider $M \in \{50,150,250,350,450\}$. However,
	in view of the insights of Section~\ref{sec:halving_snapshots} we only effectively compute $\frac{M}{2}$ samples
	as described therein. 
	\item[(iii)] {\bf Hyper-parameters Configuration}.
	We consider the following configurations of hyper-parameters: $\vartheta_1 = \vartheta_2 = 1$,
	$\nu = 0.5$, $\lambda =10$, and $\omega =10$ in \eqref{eq:forcing_term}, and we set $\zeta_1 = 4$
	and $\zeta_2 =1$ in \eqref{eq:ic_2}.
	\item[(iv)] {\bf Time-stepping Scheme.}
	For both the computation of the high-fidelity solution and the reduced basis solution, 
	i.e., the numerical approximation of \eqref{prbm:semi_discrete_problem} and Problem~\ref{pr:sdpr},
	respectively, we consider the backward Euler time-stepping scheme.
	We set the final time to $T=10$, and the total number of time steps to $N_t = 2.5\times 10^4$.
\end{itemize}


\subsubsection{Singular Values of the Snapshot Matrix}
Figure~\ref{fig:modes_pm} portrays the decay of the singular values of the snapshot matrix for 
the initial conditions $u^{(1)}_{0}$ and $u^{(2)}_{0}$ defined in \eqref{eq:ic_1} and \eqref{eq:ic_2},
respectively, and the set-up described in Sections~\ref{sec:setting} and \ref{sec:results_square}, where
the considerations of the latter section are particular for the problem in the unit square.

\begin{figure}[!ht]
	\centering
	\begin{subfigure}{0.49\linewidth}
		\includegraphics[width=\textwidth]
		{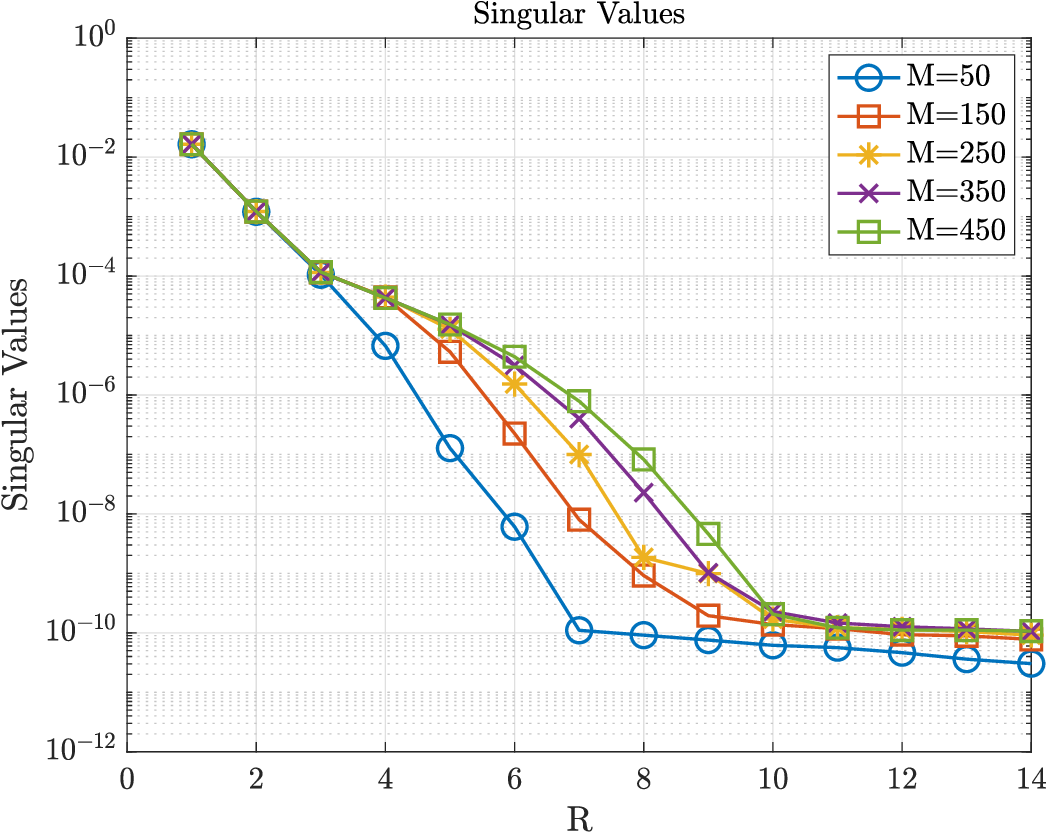}
		\subcaption{\label{fig:setting_1_square_decay_sing_values} Initial Condition $u^{(1)}_{0}$.}
	\end{subfigure}
	\begin{subfigure}{0.49\linewidth}
		\includegraphics[width=\textwidth]
		{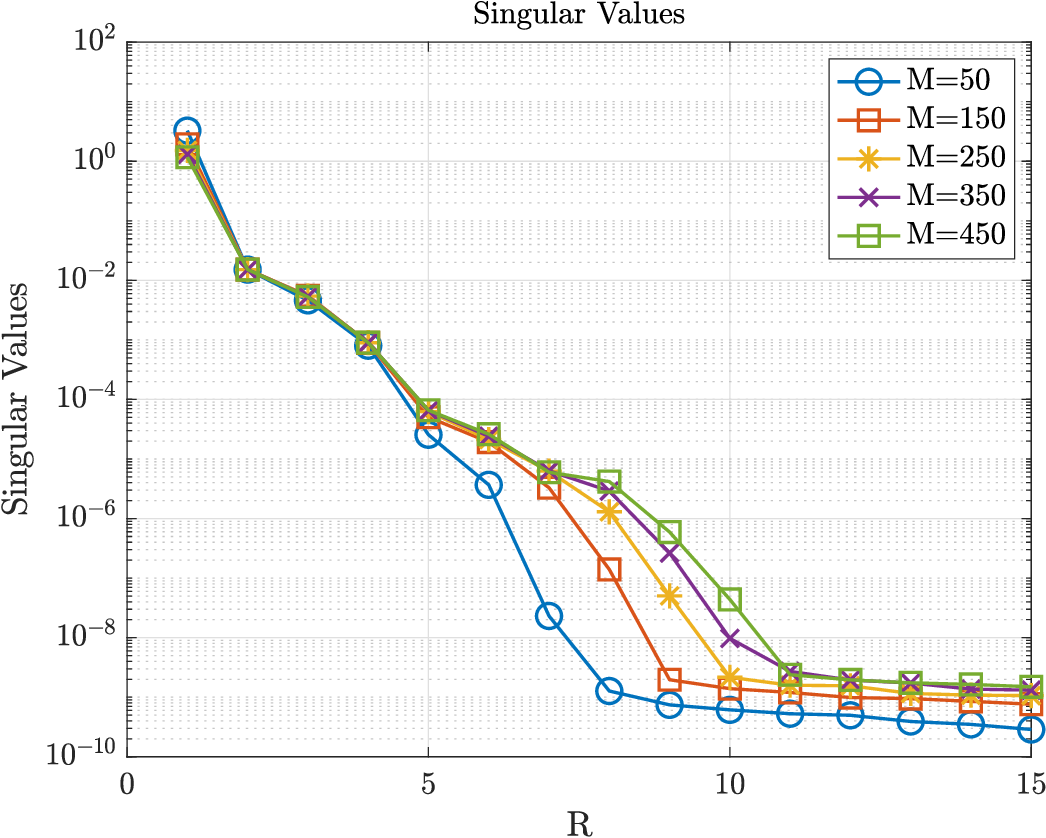}
		\subcaption{\label{fig:setting_2_square_decay_sing_values}  Initial Condition $u^{(2)}_{0}$.}
	\end{subfigure}
	\caption{\label{fig:modes_pm} 
	Singular values of the snapshot matrix for the setting considered in Section~\ref{sec:results_square}
	for the unit square in two dimensions, i.e. $\Omega = (-\frac{1}{2},\frac{1}{2})^2 \subset \mathbb{R}^2$,
	and for the initial conditions $u^{(1)}_{0}$ and $u^{(2)}_{0}$ defined in \eqref{eq:ic_1} and \eqref{eq:ic_2},
	respectively.
	}
\end{figure}


\subsubsection{Convergence of the Relative Error}

Figure~\ref{fig:error_rel_U1} and Figure~\ref{fig:error_rel_U2} portray the
convergence of the relative error as defined in Section~\ref{sec:example_I}
between the high-fidelity solution and the reduced solution as the dimension of the reduced space
increases for initial conditions $u^{(1)}_{0}$ and $u^{(2)}_{0}$, respectively.
More precisely, Figure~\ref{fig:plot_relative_error_U1_L2} and Figure~\ref{fig:plot_relative_error_U1_H1}
present the aforementioned error measure with $X = L^2(\Omega)$ and $X=H^1_0(\Omega)$
in Section~\ref{sec:example_I}, respectively, and for $M \in \{50,150,250,350,450\}$.
Again, we remark that under the considerations presented in Section~\ref{sec:halving_snapshots},
effectively only half, i.e. $\frac{M}{2}$, the number of snapshots are required. The same
holds for Figure~\ref{fig:plot_relative_error_U2_L2} and Figure~\ref{fig:plot_relative_error_U2_H1}.

\begin{figure}[!ht]
	\centering
	\begin{subfigure}{0.49\linewidth}
		\includegraphics[width=\textwidth]
		{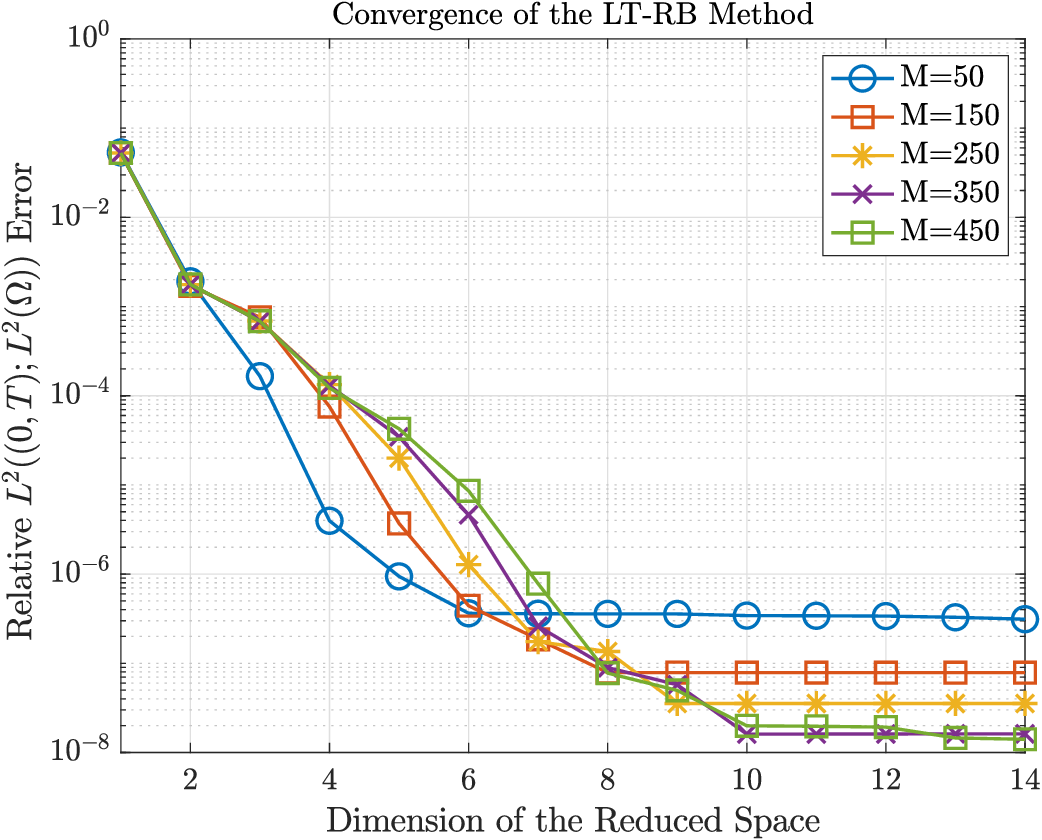}
		\subcaption{$\text{Rel\_Error}^{\normalfont\text{(rb)}}_R(\mathfrak{J};L^2(\Omega))$.}
		\label{fig:plot_relative_error_U1_L2}
	\end{subfigure}
	\begin{subfigure}{0.49\linewidth}
		\includegraphics[width=\textwidth]
		{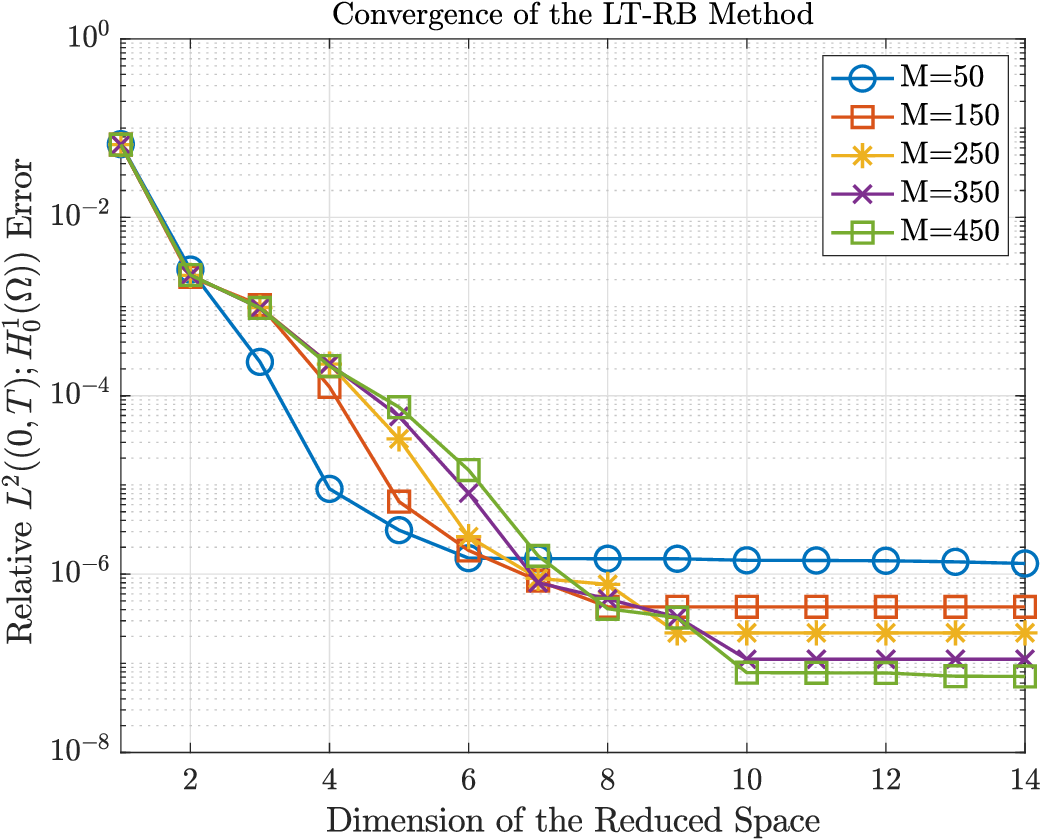}
		\subcaption{$\text{Rel\_Error}^{\normalfont\text{(rb)}}_R(\mathfrak{J};H^1_0(\Omega))$.}
		\label{fig:plot_relative_error_U1_H1}
	\end{subfigure}
	\caption{\label{fig:error_rel_U1} 
	Convergence of the relative error as defined in Section~\ref{sec:example_I} between the high-fidelity solution 
	and the reduced solution as the dimension of the reduced space
	increases from $R=1$ up to and including $R=14$, for $M \in \{50,150,250,350,450\}$.
	In Figure~\ref{fig:plot_relative_error_U1_L2} the relative error is computed in the $L^2(\Omega)$-norm
	and in Figure~\ref{fig:plot_relative_error_U1_L2} in the $H^1_0(\Omega)$-norm.
	The geometrical setting corresponds to the one described in Section~\ref{sec:results_square},
	i.e.~the unit square in two dimensions $\Omega = (-\frac{1}{2},\frac{1}{2})^2 \subset \mathbb{R}^2$,
	and for the initial conditions $u^{(1)}_{0}$ defined in \eqref{eq:ic_1}.
	}
\end{figure}

\begin{figure}[!ht]
	\centering
	\begin{subfigure}{0.49\linewidth}
		\includegraphics[width=\textwidth]
		{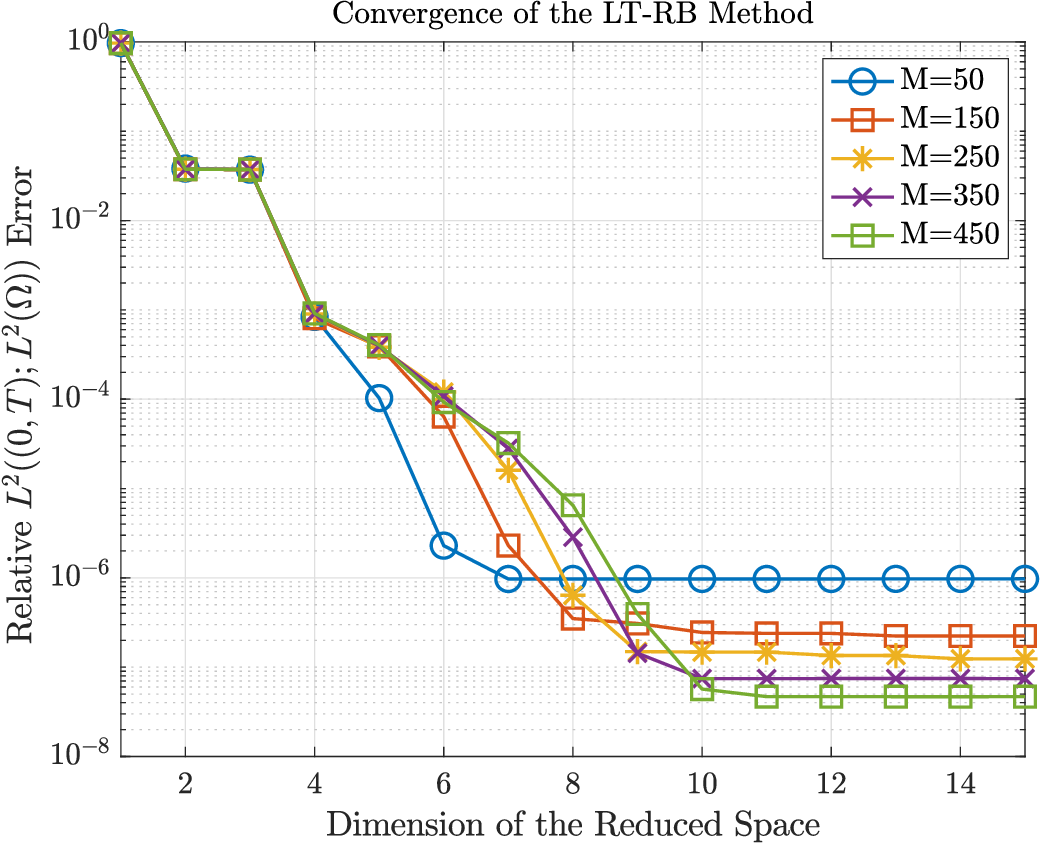}
		\subcaption{$\text{Rel\_Error}^{\normalfont\text{(rb)}}_R(\mathfrak{J};L^2(\Omega))$.}
		\label{fig:plot_relative_error_U2_L2}
	\end{subfigure}
	\begin{subfigure}{0.49\linewidth}
		\includegraphics[width=\textwidth]
		{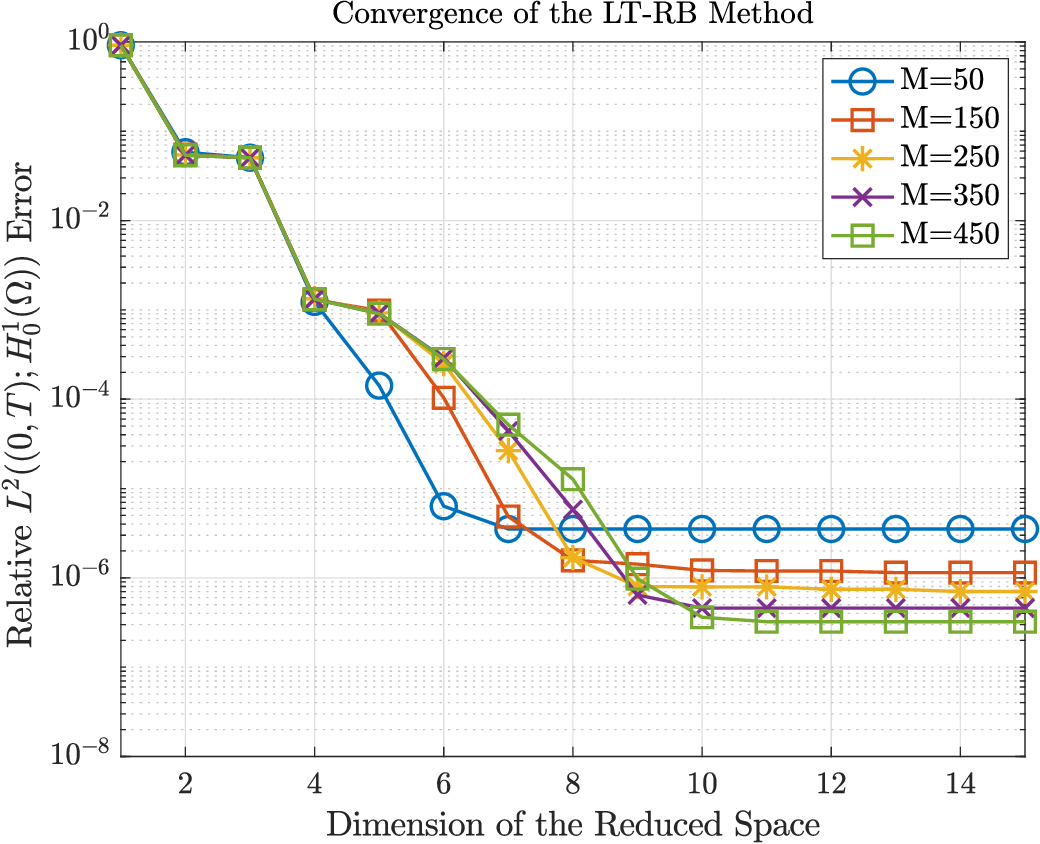}
		\subcaption{$\text{Rel\_Error}^{\normalfont\text{(rb)}}_R(\mathfrak{J};H^1_0(\Omega))$.}
		\label{fig:plot_relative_error_U2_H1}
	\end{subfigure}
	\caption{\label{fig:error_rel_U2} 
	Convergence of the relative error as defined in Section~\ref{sec:example_I} between the high-fidelity solution 
	and the reduced solution as the dimension of the reduced space
	increases from $R=1$ up to and including $R=15$, for $M \in \{50,150,250,350,450\}$.
	In Figure~\ref{fig:plot_relative_error_U2_L2} the relative error is computed in the $L^2(\Omega)$-norm
	and in Figure~\ref{fig:plot_relative_error_U2_L2} in the $H^1_0(\Omega)$-norm.
	The geometrical setting corresponds to the one described in Section~\ref{sec:results_square},
	i.e.~the unit square in two dimensions $\Omega = (-\frac{1}{2},\frac{1}{2})^2 \subset \mathbb{R}^2$,
	and for the initial conditions $u^{(2)}_{0}$ defined in \eqref{eq:ic_2}.
	}
\end{figure}

\subsubsection{Speed-up}

Figure~\ref{fig:plot_speed_U1_HF} presents the execution times
for the computation of the high-fidelity solution with initial condition $u^{(1)}_{0}$ split into two contributions: (1)
{\sf Assemble FEM}, which consists in the time-required to set up the FE linear system of 
equations, (2) {\sf Solve TD-HF}, which corresponds to the total time required to solve
the high-fidelity model using the backward Euler scheme.

Figure~\ref{fig:plot_speed_U1_RB_50} through Figure~\ref{fig:plot_speed_U1_RB_450}
show the execution times of the \rev{LT-MOR} method for $M \in \{50,150,250,350,450\}$.
In each of these plots, the total time is broken down into the following contributions:
(1) Assembling the FE discretization ({\sf Assemble FEM}), (2) computing the snapshots 
or high-fidelity solutions in the Laplace domain ({\sf LD-HF}), (3) build the reduced basis ({\sf Build RB}),
and (4) compute the reduce solution in the time domain ({\sf Solve TD-RB}). 

\begin{figure}[!ht]
	\centering
	\begin{subfigure}{0.49\linewidth}
		\includegraphics[width=\textwidth]
		{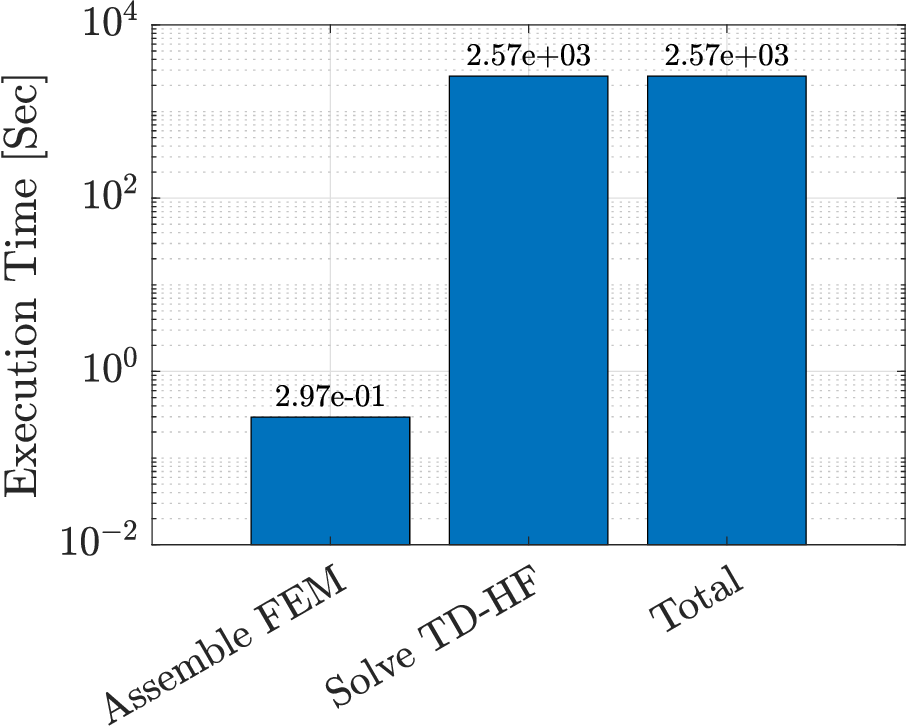}
		\subcaption{High-Fidelity Solution.}
		\label{fig:plot_speed_U1_HF}
	\end{subfigure}
	\begin{subfigure}{0.49\linewidth}
		\includegraphics[width=\textwidth]
		{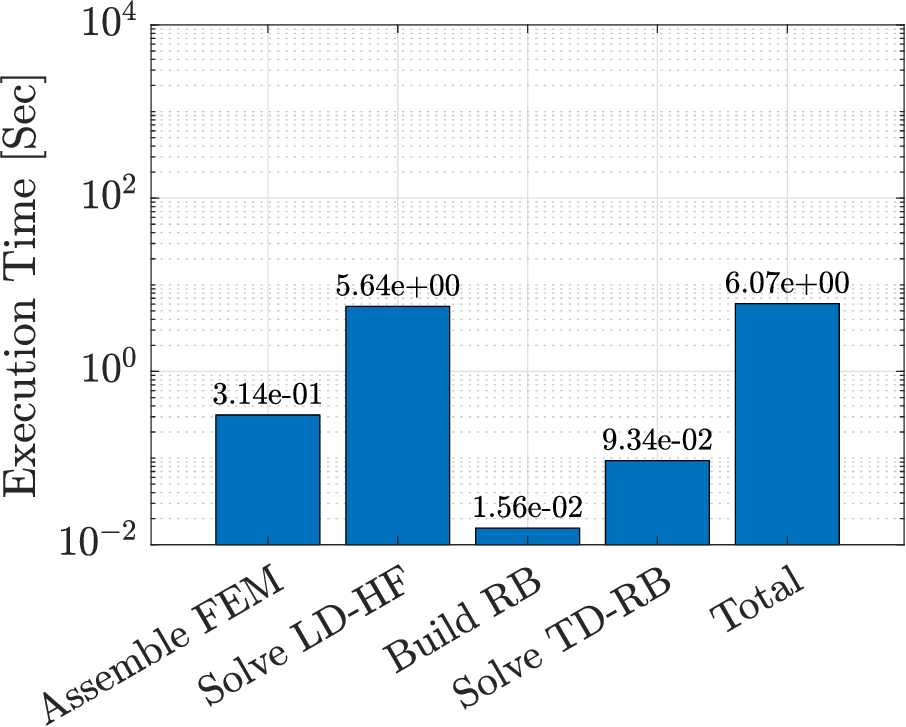}
		\subcaption{\rev{LT-MOR} with $M=50$ and $R=14$.}
		\label{fig:plot_speed_U1_RB_50}
	\end{subfigure}
	\centering
	\begin{subfigure}{0.49\linewidth}
		\includegraphics[width=\textwidth]
		{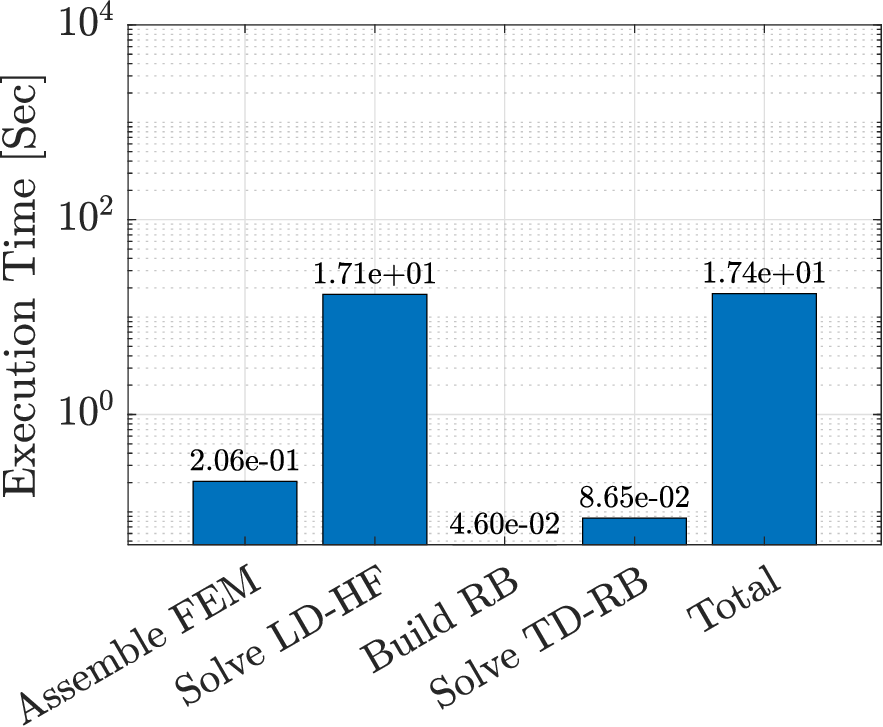}
		\subcaption{\rev{LT-MOR} with $M=150$ and $R=14$.}
		\label{fig:plot_speed_U1_RB_150}
	\end{subfigure}
	\begin{subfigure}{0.49\linewidth}
		\includegraphics[width=\textwidth]
		{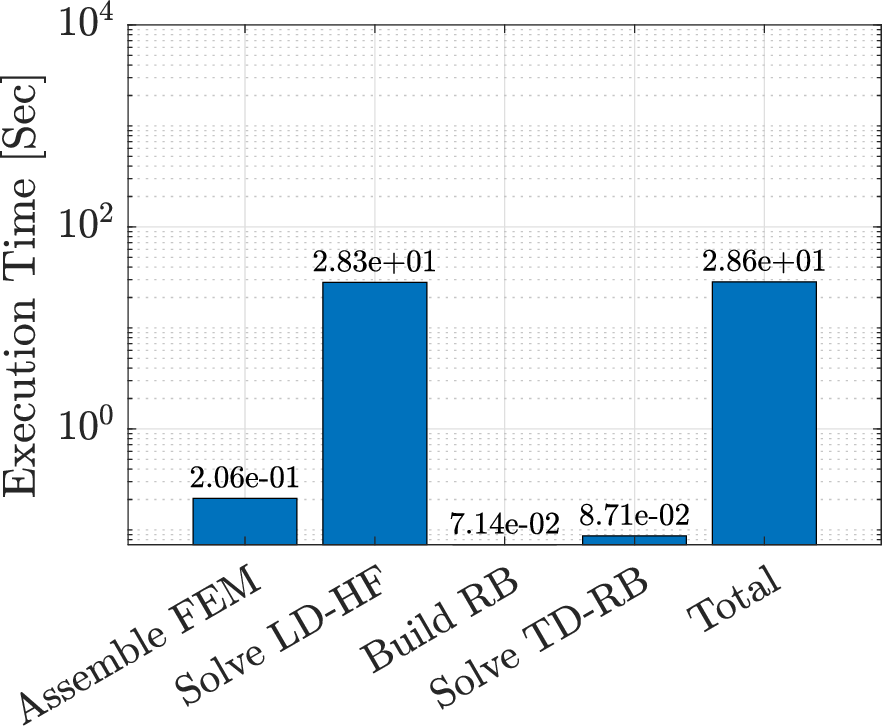}
		\subcaption{\rev{LT-MOR} with $M=250$ and $R=14$.}
		\label{fig:plot_speed_U1_RB_250}
	\end{subfigure}
	\centering
	\begin{subfigure}{0.49\linewidth}
		\includegraphics[width=\textwidth]
		{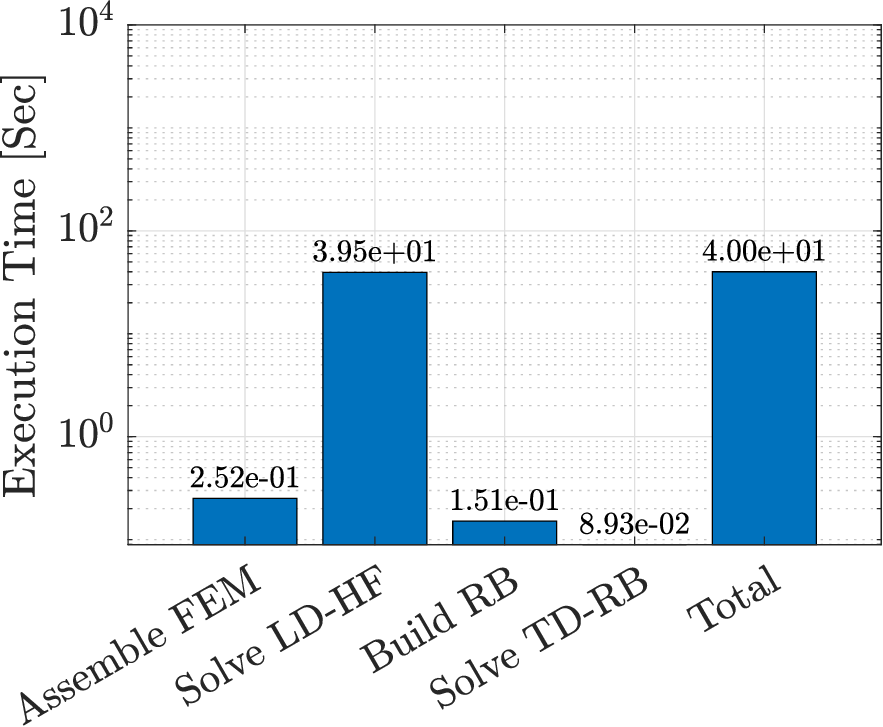}
		\subcaption{\rev{LT-MOR} with $M=350$ and $R=14$.}
		\label{fig:plot_speed_U1_RB_350}
	\end{subfigure}
	\begin{subfigure}{0.49\linewidth}
		\includegraphics[width=\textwidth]
		{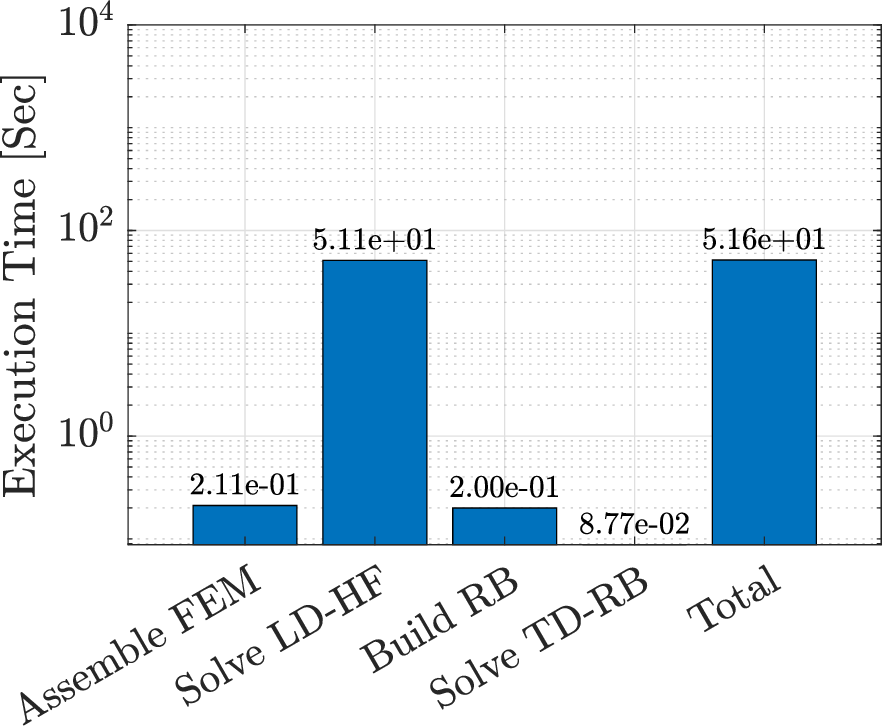}
		\subcaption{\rev{LT-MOR} with $M=450$ and $R=14$.}
		\label{fig:plot_speed_U1_RB_450}
	\end{subfigure}
	\caption{\label{fig:speed_square} 
	Execution times for the computation of the high-fidelity
	solution with initial condition $u^{(1)}_{0}$ and the 
	reduced one for the parabolic problem in the unit
	square as described in Section~\ref{sec:results_square}.
	Figure~\ref{fig:plot_speed_U1_HF} presents the execution time
	for the computation of the high-fidelity solution split into the two main contributions: (1)
	Assembling the FE discretization ({\sf Assemble FEM}), and (2)
	solving the high-fidelity model ({\sf Solve TD-HF}).
	Figure~\ref{fig:plot_speed_U1_RB_50} through Figure~\ref{fig:plot_speed_U1_RB_450}
	show the execution time of the \rev{LT-MOR} method for $M \in \{50,150,250,350,450\}$.
	In each of these plots, the total time is broken down into four main contributions:
	(1) Assembling the FE discretization ({\sf Assemble FEM}), (2) computing the snapshots 
	or high-fidelity solutions in the Laplace domain ({\sf LD-HF}),
	(3) \rev{building} the reduced basis ({\sf Build RB}),
	and (4) \rev{computing} the \rev{reduced} solution in the time domain ({\sf Solve TD-RB}). 
	}
\end{figure}

\subsubsection{Visualization of the Reduced Basis Space}
Let $\mathbb{V}^{\text{(rb)}}_R$ be as in \eqref{eq:fPOD} for some $R\in\IN$.
Then, $\left\{\varphi^{\textrm{(rb)}}_1,\cdots,\varphi^{\textrm{(rb)}}_R\right\}$
constitutes an orthonormal basis of $\IV_R^{(\textrm{rb})}$ with $\varphi^{\textrm{(rb)}}_j$ as
in \eqref{eq:def_basis_rb}. Indeed, provided that $\bm{\Phi}^{\textrm{(rb)}}_R$
solution to \eqref{eq:fPOD} has been computed, one can plot the corresponding representation in $\mathbb{V}_h$
by using the expression stated in \eqref{eq:def_basis_rb} (though originally introduced for time-dependent
approach for MOR, it is also valid for the \rev{LT-MOR} method).
In Figure~\ref{fig:reduced_basis_1} and Figure~\ref{fig:reduced_basis_2},
we visualize the basis $\left\{\varphi^{\textrm{(rb)}}_1,\cdots,\varphi^{\textrm{(rb)}}_R\right\}$
in the space $\mathbb{V}_h$ for the initial conditions \eqref{eq:ic_1} and \eqref{eq:ic_2}, respectively.
\begin{figure}[!ht]
	\centering
	\begin{subfigure}{0.49\linewidth}
		\includegraphics[width=\textwidth]
		{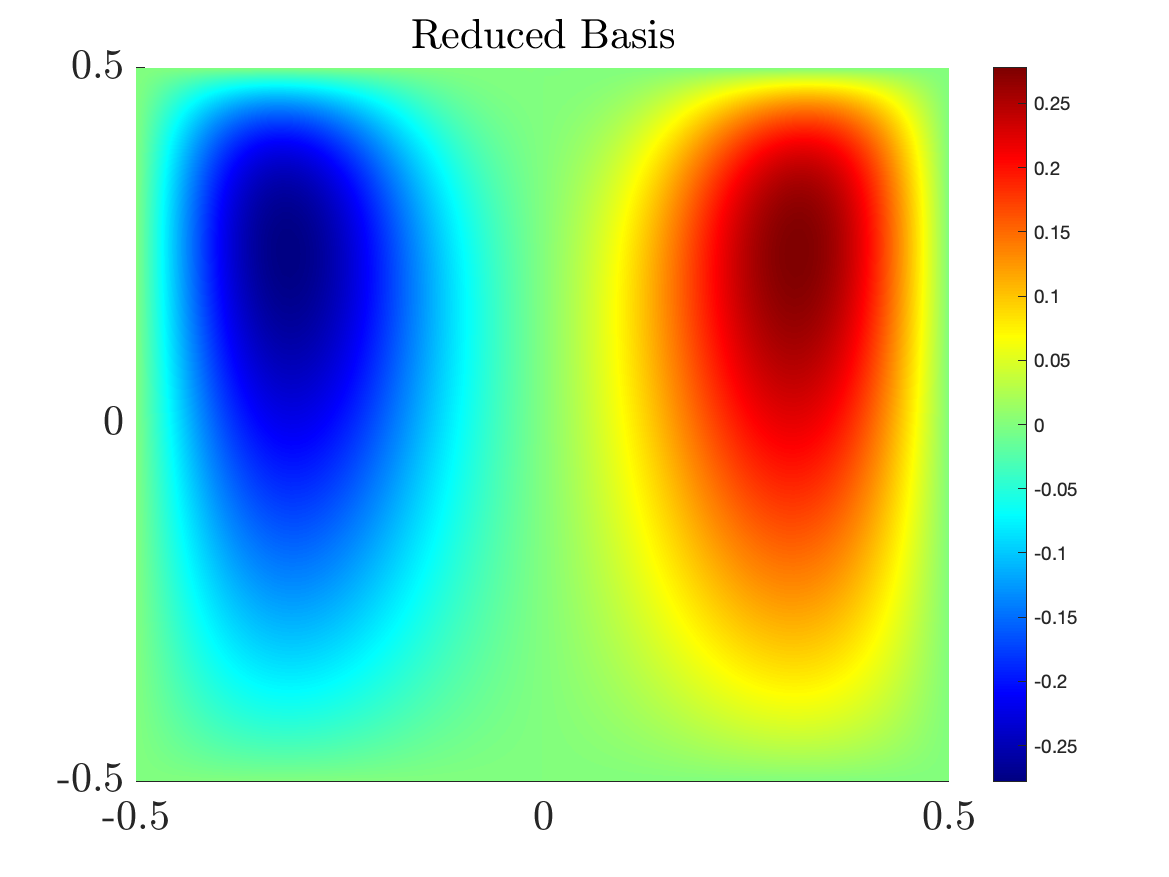}
		\subcaption{$\varphi^{\textrm{(rb)}}_1$}
		\label{fig:plot_times_speed_omega_x_10_omega_25_nu_2_M_45}
	\end{subfigure}
	\begin{subfigure}{0.49\linewidth}
		\includegraphics[width=\textwidth]
		{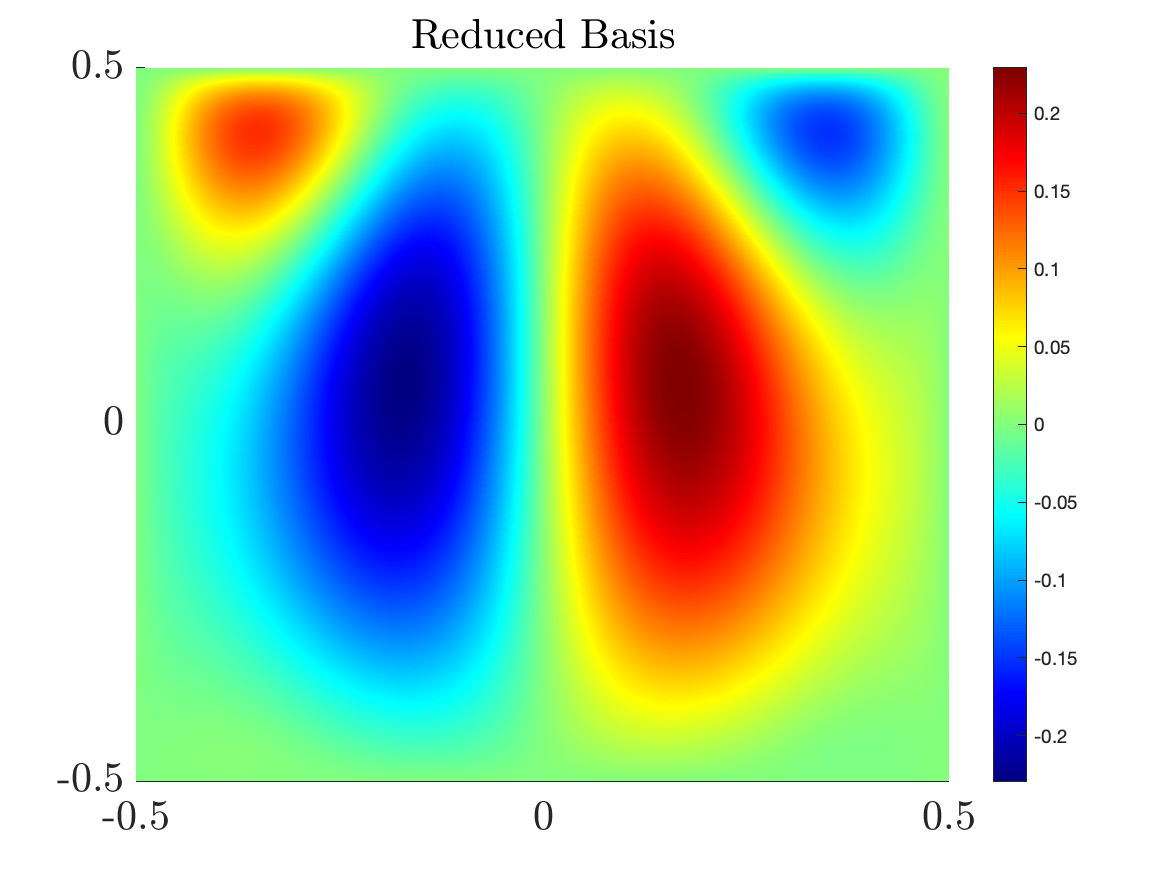}
		\subcaption{$\varphi^{\textrm{(rb)}}_2$}
		\label{fig:plot_times_speed_omega_x_10_omega_25_nu_2_M_125}
	\end{subfigure}
	\centering
	\begin{subfigure}{0.49\linewidth}
		\includegraphics[width=\textwidth]
		{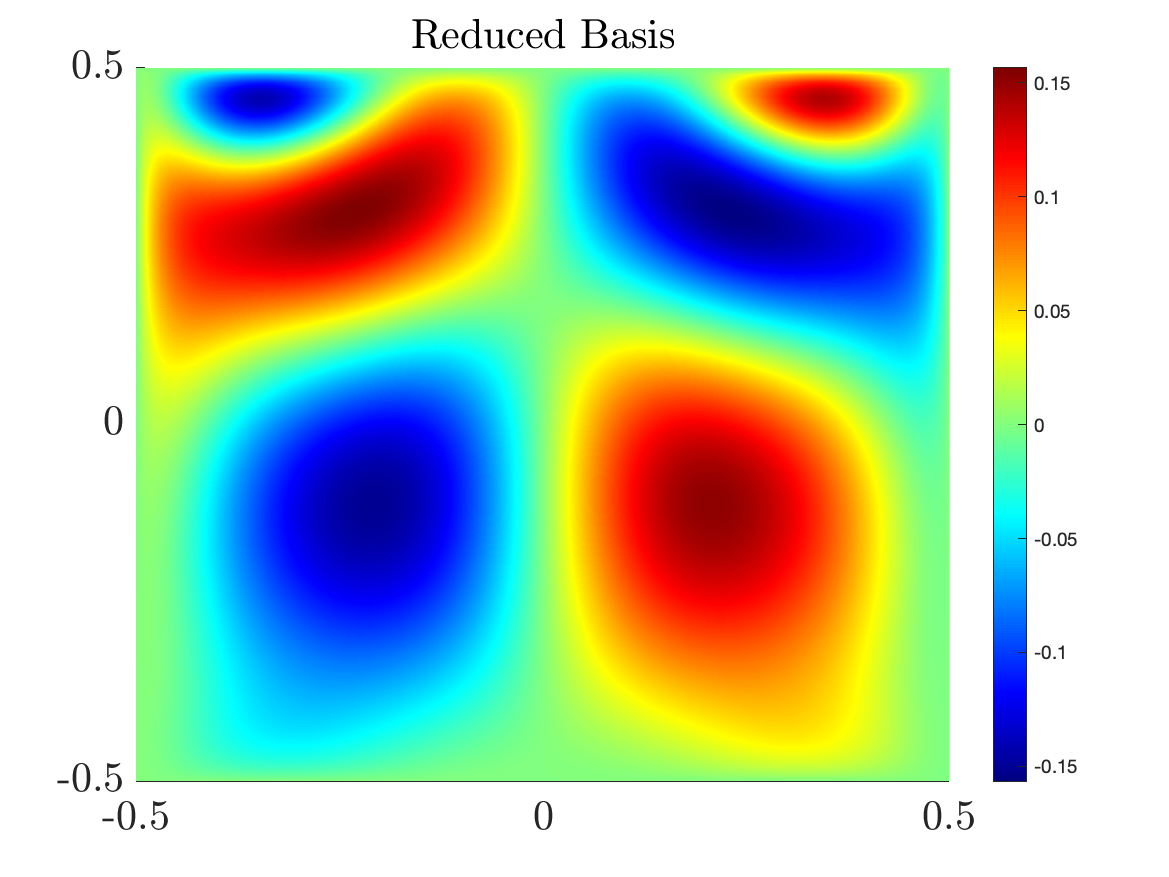}
		\subcaption{$\varphi^{\textrm{(rb)}}_3$}
		\label{fig:plot_times_speed_omega_x_10_omega_25_nu_2_M_45}
	\end{subfigure}
	\begin{subfigure}{0.49\linewidth}
		\includegraphics[width=\textwidth]
		{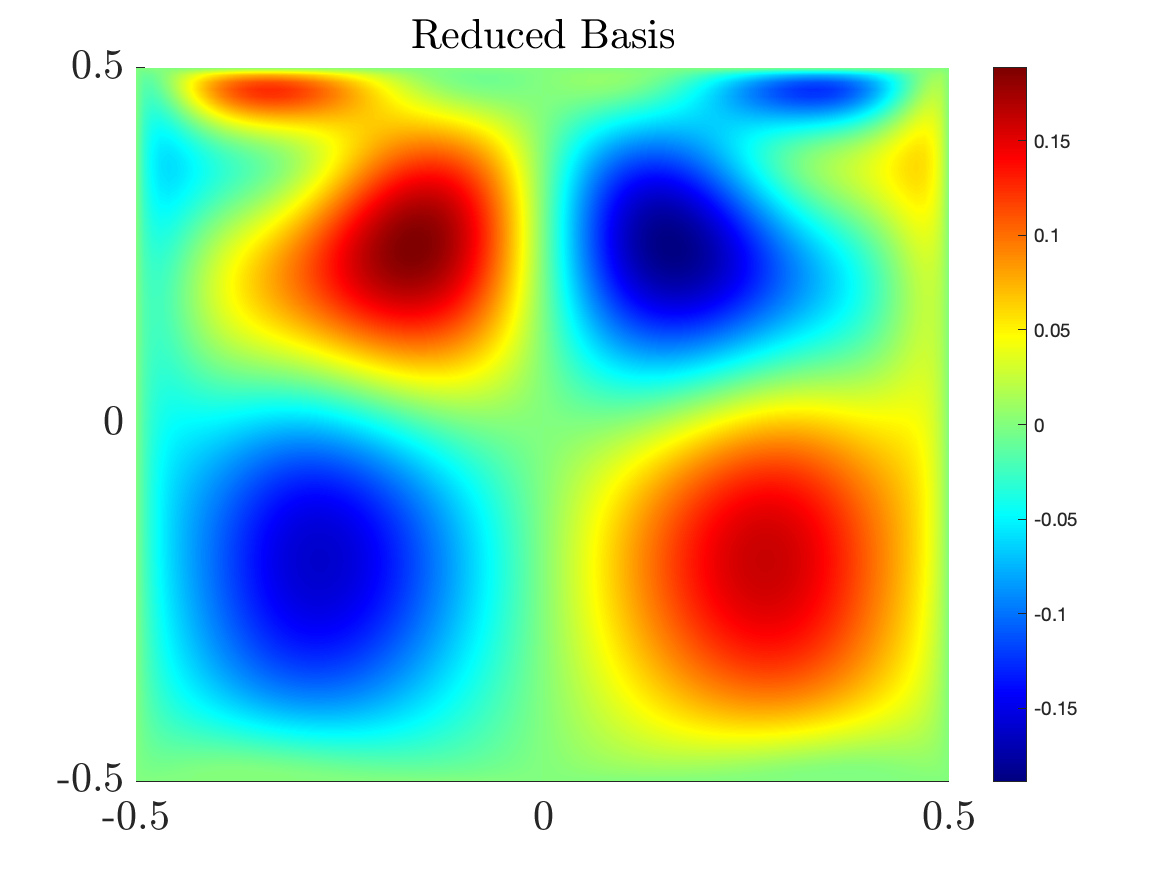}
		\subcaption{$\varphi^{\textrm{(rb)}}_4$}
		\label{fig:plot_times_speed_omega_x_10_omega_25_nu_2_M_125}
	\end{subfigure}
	\centering
	\begin{subfigure}{0.49\linewidth}
		\includegraphics[width=\textwidth]
		{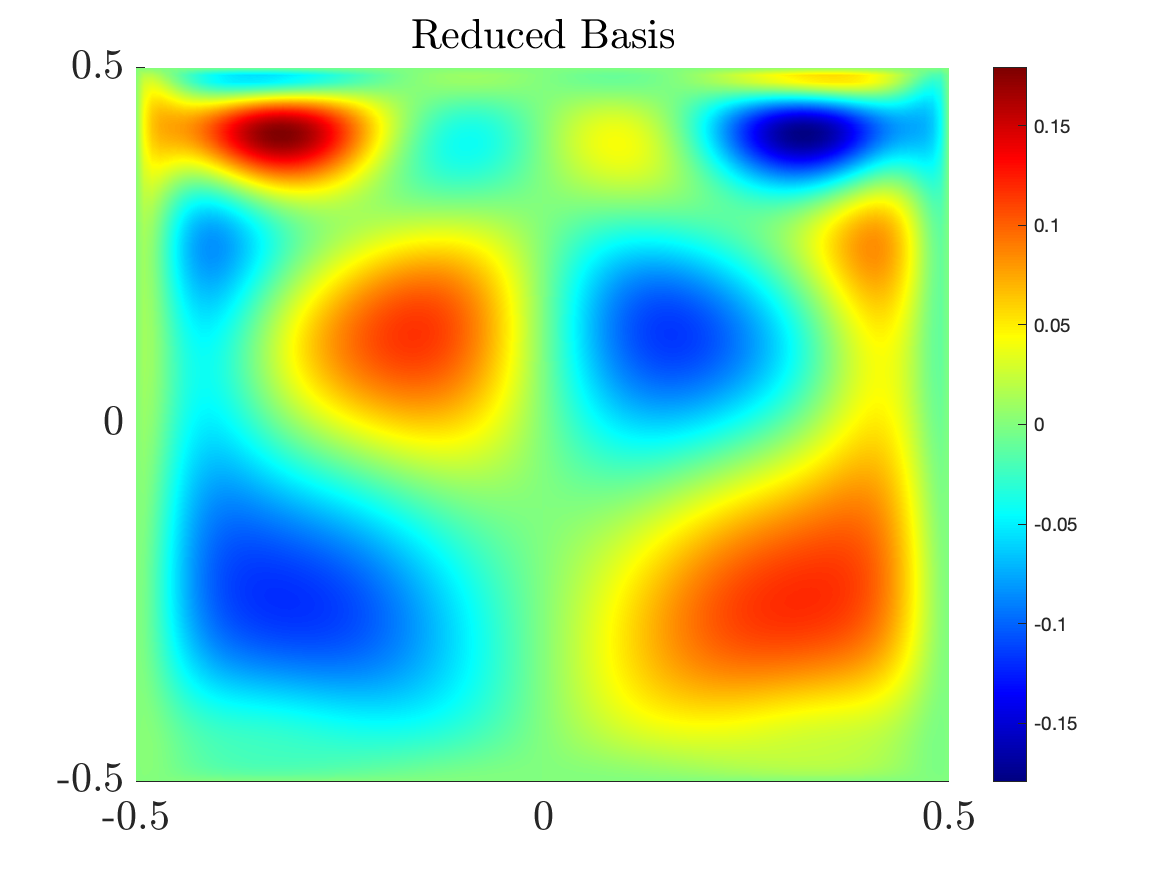}
		\subcaption{$\varphi^{\textrm{(rb)}}_5$}
		\label{fig:plot_times_speed_omega_x_10_omega_25_nu_2_M_45}
	\end{subfigure}
	\begin{subfigure}{0.49\linewidth}
		\includegraphics[width=\textwidth]
		{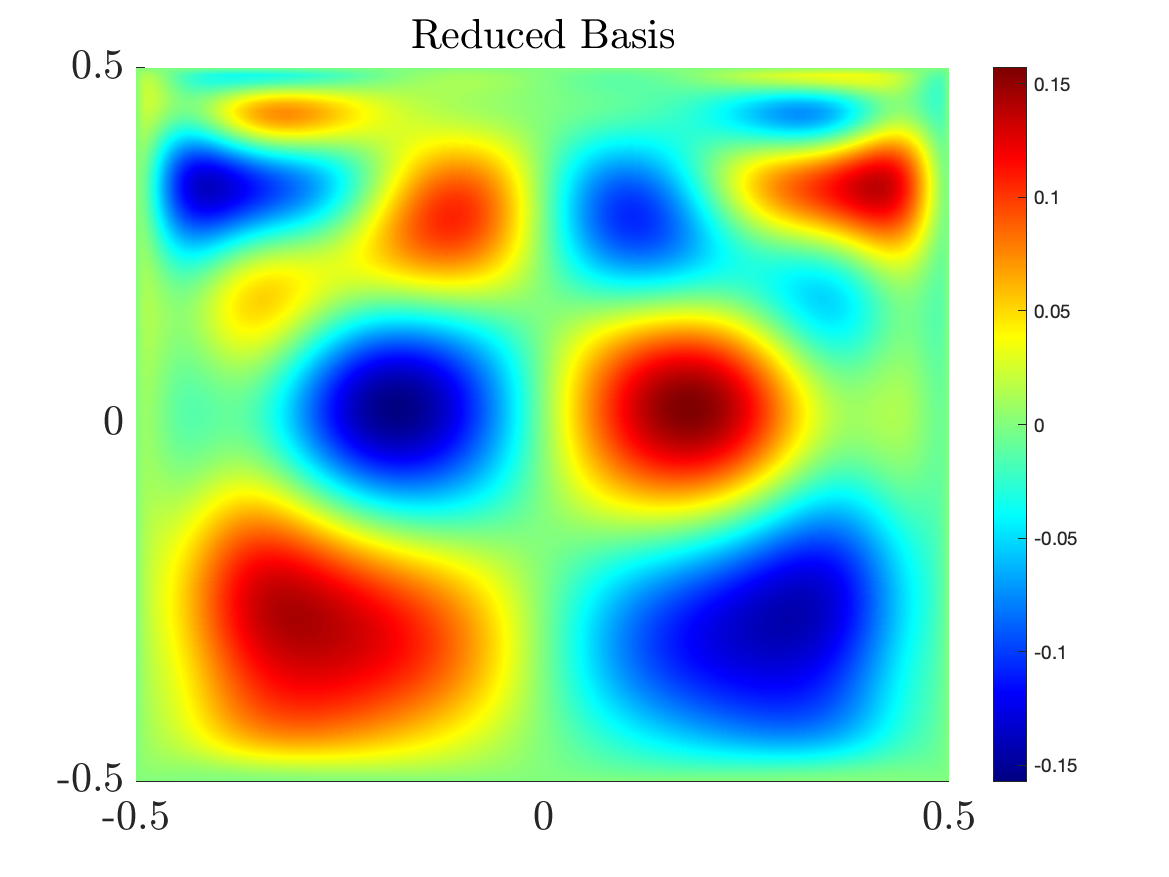}
		\subcaption{$\varphi^{\textrm{(rb)}}_6$}
		\label{fig:plot_times_speed_omega_x_10_omega_25_nu_2_M_125}
	\end{subfigure}
	\caption{\label{fig:reduced_basis_1} 
	Visualization of the first six elements of the reduced space \rev{$\mathbb{V}^{\text{(rb)}}_{R,M}$}
	for the initial condition $u^{(1)}_{0}$ in the square.
	}
\end{figure}

\begin{figure}[!ht]
	\centering
	\begin{subfigure}{0.49\linewidth}
		\includegraphics[width=\textwidth]
		{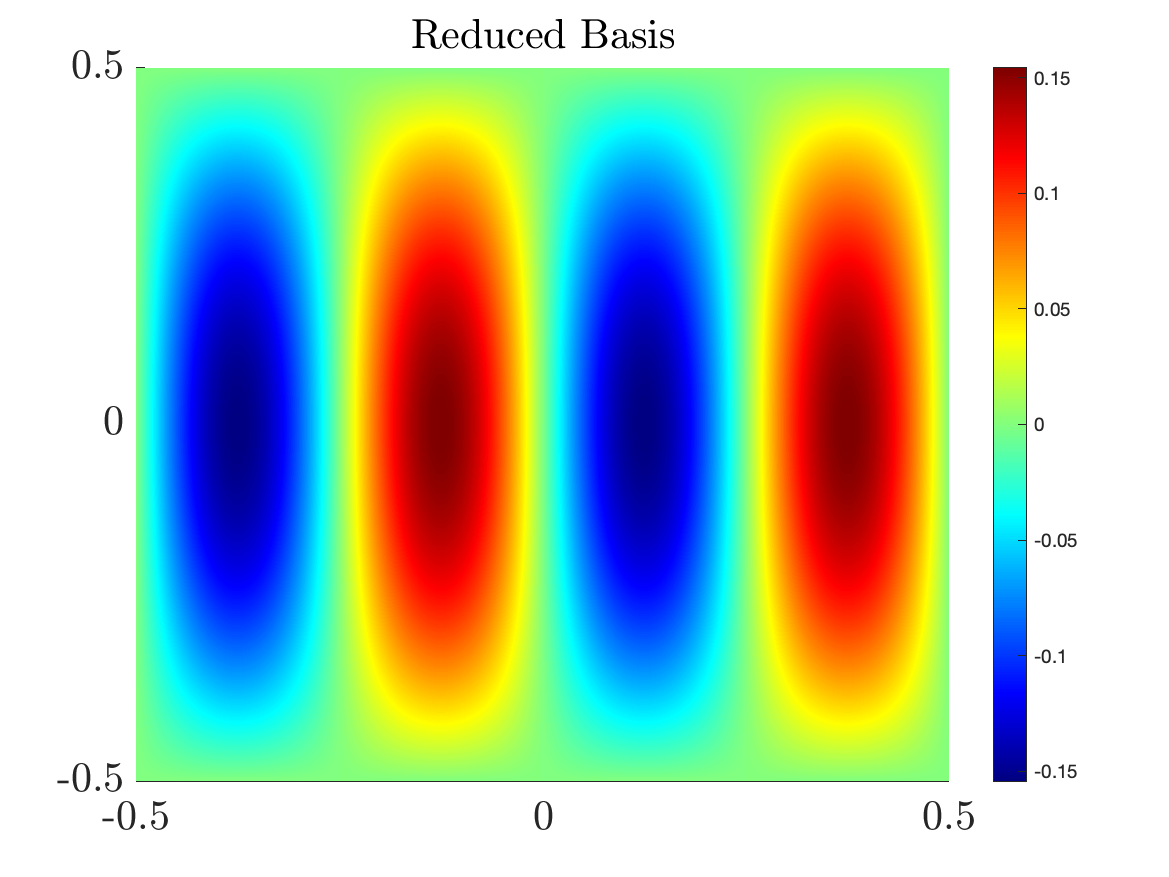}
		\subcaption{$\varphi^{\textrm{(rb)}}_1$}
	\end{subfigure}
	\begin{subfigure}{0.49\linewidth}
		\includegraphics[width=\textwidth]
		{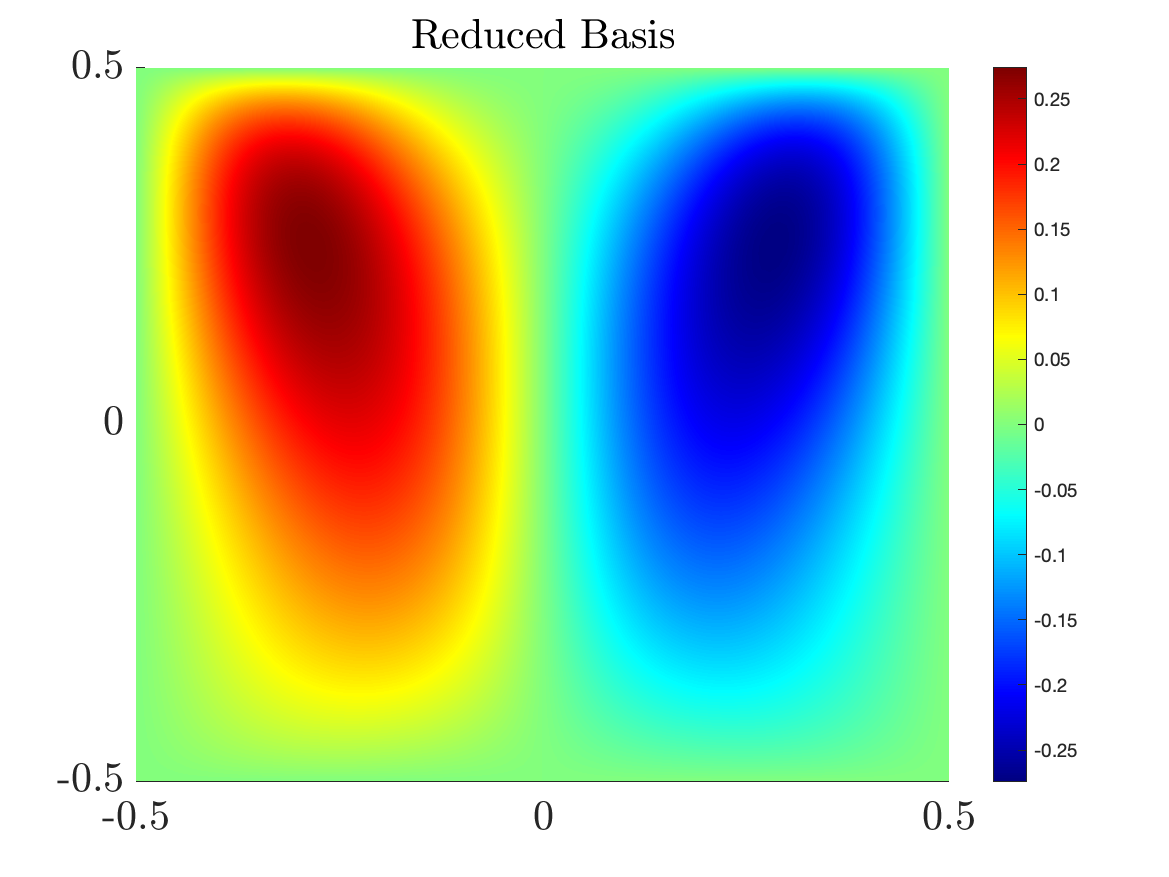}
		\subcaption{$\varphi^{\textrm{(rb)}}_2$}
	\end{subfigure}
	\centering
	\begin{subfigure}{0.49\linewidth}
		\includegraphics[width=\textwidth]
		{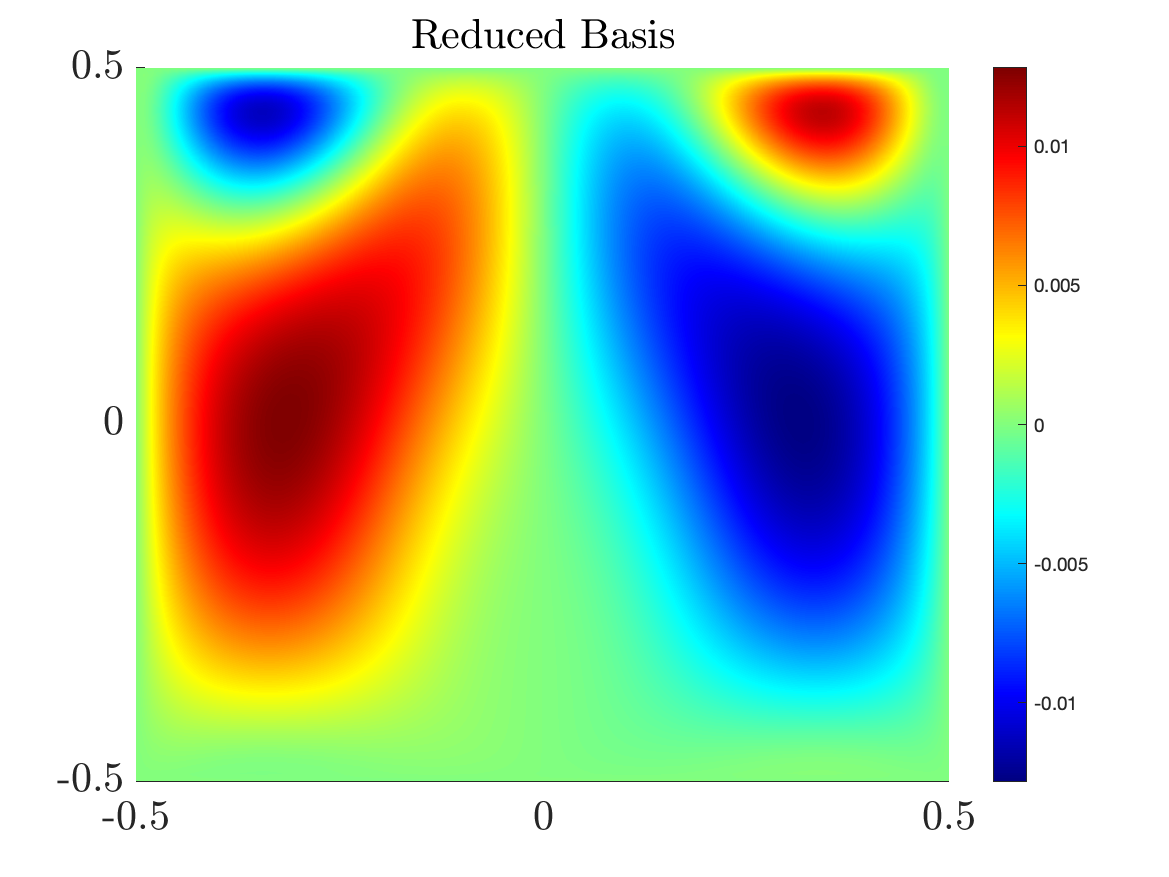}
		\subcaption{$\varphi^{\textrm{(rb)}}_3$}
	\end{subfigure}
	\begin{subfigure}{0.49\linewidth}
		\includegraphics[width=\textwidth]
		{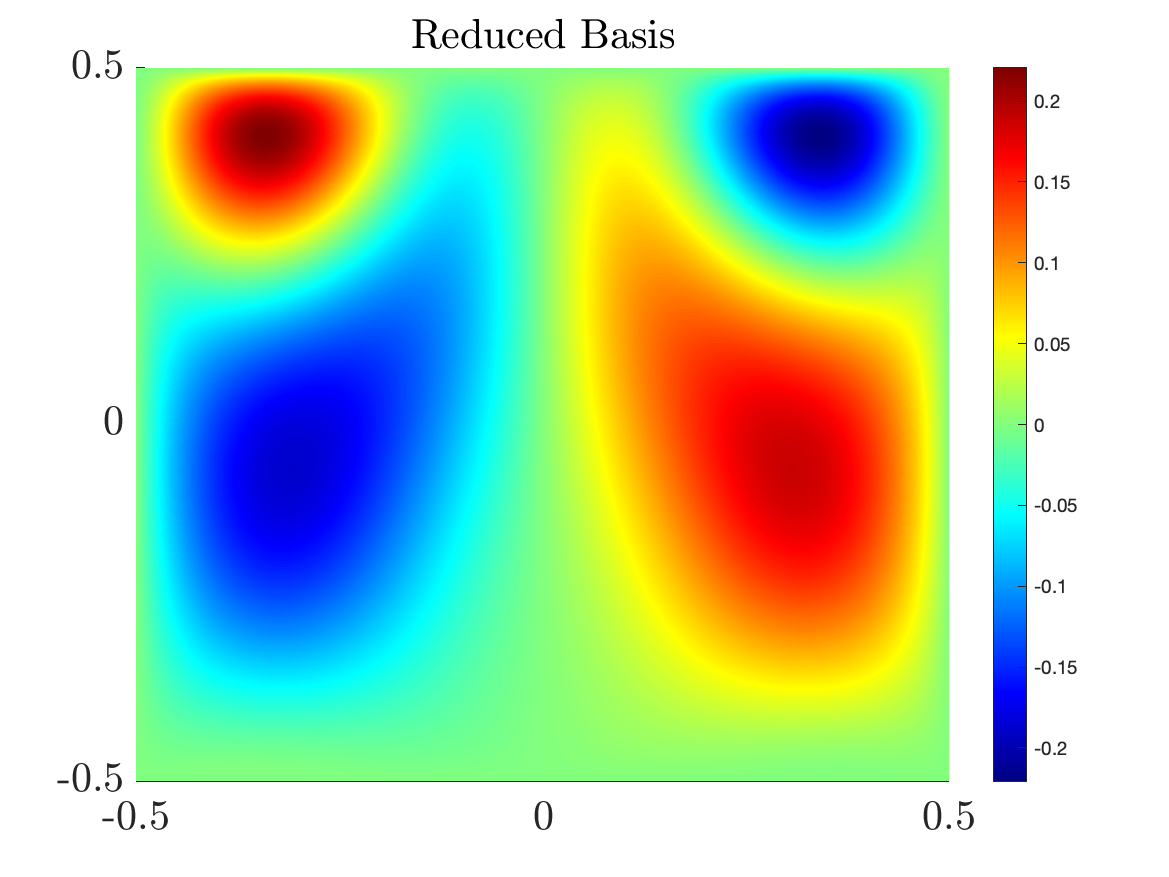}
		\subcaption{$\varphi^{\textrm{(rb)}}_4$}
	\end{subfigure}
	\centering
	\begin{subfigure}{0.49\linewidth}
		\includegraphics[width=\textwidth]
		{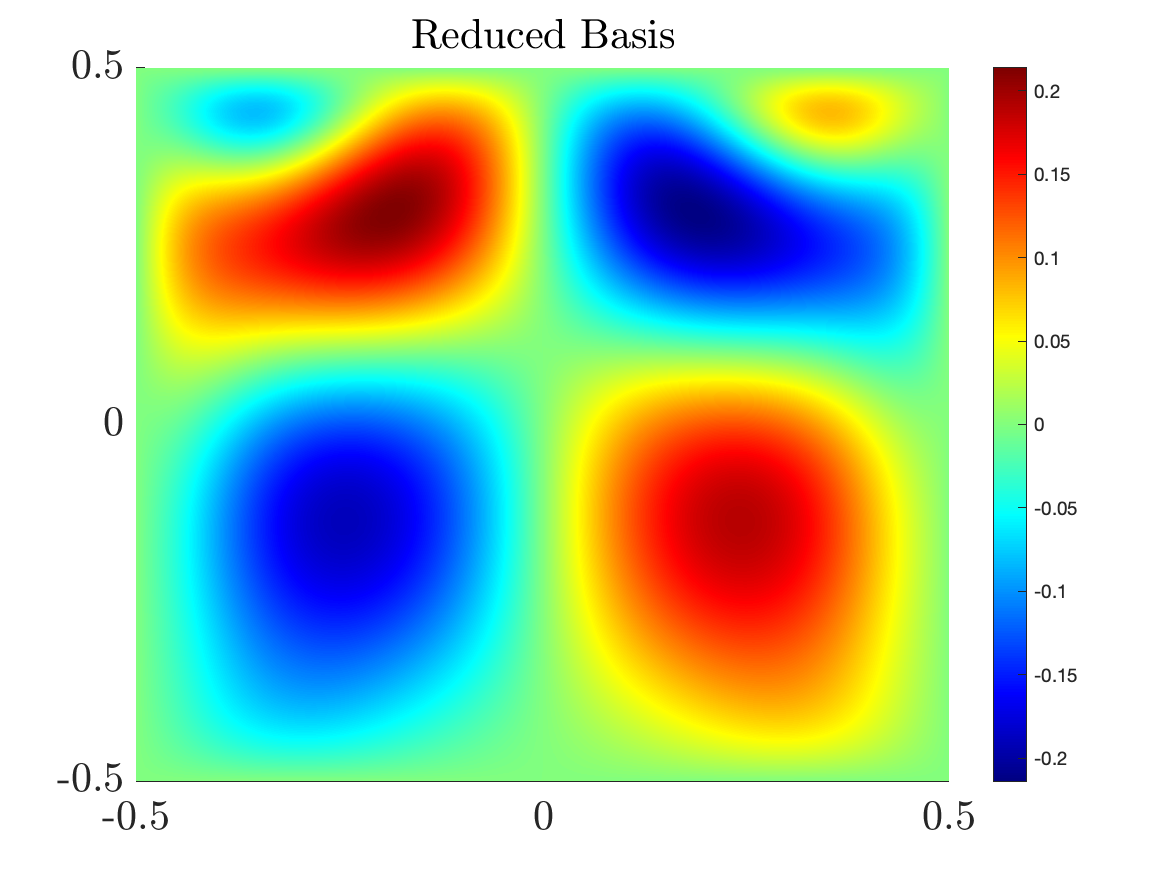}
		\subcaption{$\varphi^{\textrm{(rb)}}_5$}
	\end{subfigure}
	\begin{subfigure}{0.49\linewidth}
		\includegraphics[width=\textwidth]
		{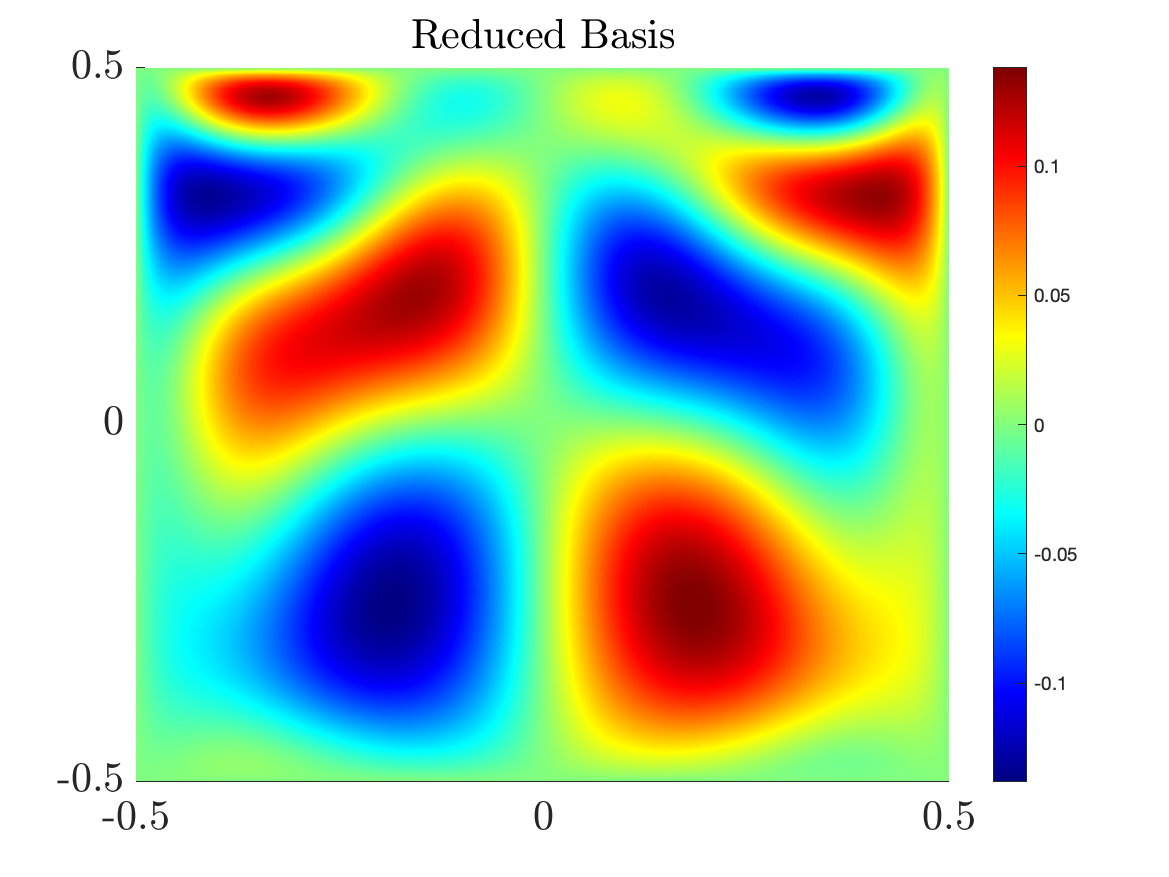}
		\subcaption{$\varphi^{\textrm{(rb)}}_6$}
	\end{subfigure}
	\caption{\label{fig:reduced_basis_2} 
	Visualization of the first six elements of the reduced space \rev{$\mathbb{V}^{\text{(rb)}}_{R,M}$}
	for the initial condition $u^{(2)}_{0}$ in the square.
	}
\end{figure}


\subsection{Results in the Cube}
\label{sec:results_cube}
We compute the solution to the linear, second-order parabolic
problem in a cube with the following set-up.
\begin{itemize}
	\item[(i)] {\bf FE Discretization.}
	We consider a FE discretization using $\mathbb{P}^1$
	elements on a mesh $\mathcal{T}_h$ of $303918$ tetrahedrons, with a total number of degrees of freedom equal to $46656$,
	i.e. $\text{dim}\left(\mathcal{S}^{1,1}_0\left(\mathcal{T}_h\right)\right)=46656$
	and a mesh size $h = 4.68 \times 10^{-2}$.
	\item[(ii)] {\bf Construction of the Reduced Space.}
	The space $\mathbb{V}^{\text{(rb)}}_R \subset \mathbb{V}_h$ is computed following the approach described in 
	Section~\ref{sec:FRB}, in particular following the considerations of \eqref{rmk:comp_Vrb}, together
	with the choice of snapshots described in \eqref{eq:snap_shot_selection}.
	As per the setting for the computation of the snapshots, we set $\alpha = 1$ and $\beta = 2$ 
	in \eqref{eq:snapshots} and consider $M \in \{50,150,250\}$. However,
	in view of the analysis of Section~\ref{sec:halving_snapshots} we only compute
	effectively $\frac{M}{2}$ samples.
	\item[(iii)] {\bf Hyper-parameters Configuration}.
	We consider the following configurations of hyper-parameters: $\vartheta_1 = \vartheta_2 = 1$,
	$\nu = 0.5$, $\lambda =5$, and $\omega =5$ in \eqref{eq:forcing_term}.	
	\item[(iv)] {\bf Time-stepping Scheme.}
	For both the computation of the high-fidelity solution and the reduced basis solution, 
	i.e., the numerical approximation of \eqref{prbm:semi_discrete_problem} and Problem~\ref{pr:sdpr},
	we consider the backward Euler time-stepping scheme.
	Again, we set the final time to $T=10$, and the total number of time steps to $N_t = 10^4$.
\end{itemize}
\subsubsection{Singular Values of the Snapshot Matrix}
Figure~\ref{fig:svd_cube} portrays the decay of the singular values of the snapshot matrix for 
the initial condition $u^{(1)}_{0}$ defined in \eqref{eq:ic_1}, and the set-up described in
Section~\ref{sec:setting} and Section~\ref{sec:results_square}, where
the considerations of the latter section are particular for the cube.

\begin{figure}[!ht]
	\centering
	\begin{subfigure}{0.6\linewidth}
		\includegraphics[width=\textwidth]
		{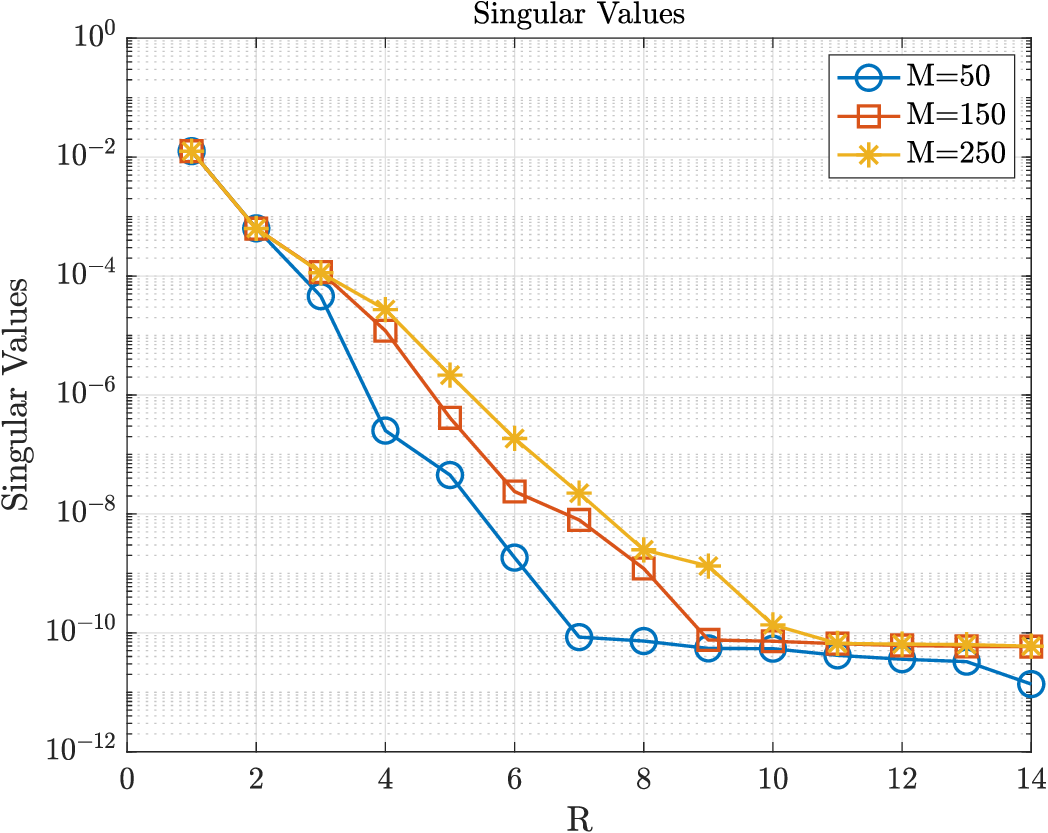}
		\subcaption{\label{fig:setting_1_cube_decay_sing_values} Initial Condition $u^{(1)}_{0}$}
	\end{subfigure}
	\caption{\label{fig:svd_cube} 
	Singular values of the snapshot matrix for the setting considered in Section~\ref{sec:results_cube}
	for the cube in three dimensions, i.e., $\Omega = (-\frac{1}{2},\frac{1}{2})^3 \subset \mathbb{R}^3$,
	and for the initial conditions $u^{(1)}_{0}$ defined in \eqref{eq:ic_1}.
	}
\end{figure}

\subsubsection{Convergence of the Relative Error}
Figure~\ref{fig:error_rel_cube} \rev{portrays} the
convergence of the relative error as defined in Section~\ref{sec:example_I}
between the high-fidelity solution and the reduced solution as the dimension of the reduced space
increases for the initial conditions $u^{(1)}_{0}$.
More precisely, Figure~\ref{fig:plot_relative_error_cube_L2_cube} and Figure~\ref{fig:plot_relative_error_cube_H1_cube}
present the aforementioned error measure with $X = L^2(\Omega)$ and $X=H^1_0(\Omega)$
in Section~\ref{sec:example_I}, respectively, and for $M \in \{50,150,250,350,450\}$.
Again, we remark that under the considerations presented in Section~\ref{sec:halving_snapshots},
effectively only half, i.e. $\frac{M}{2}$, the number of snapshots are required. The same
holds for Figure~\ref{fig:plot_relative_error_U2_L2} and Figure~\ref{fig:plot_relative_error_U2_H1}.

\begin{figure}[!ht]
	\centering
	\begin{subfigure}{0.49\linewidth}
		\includegraphics[width=\textwidth]
		{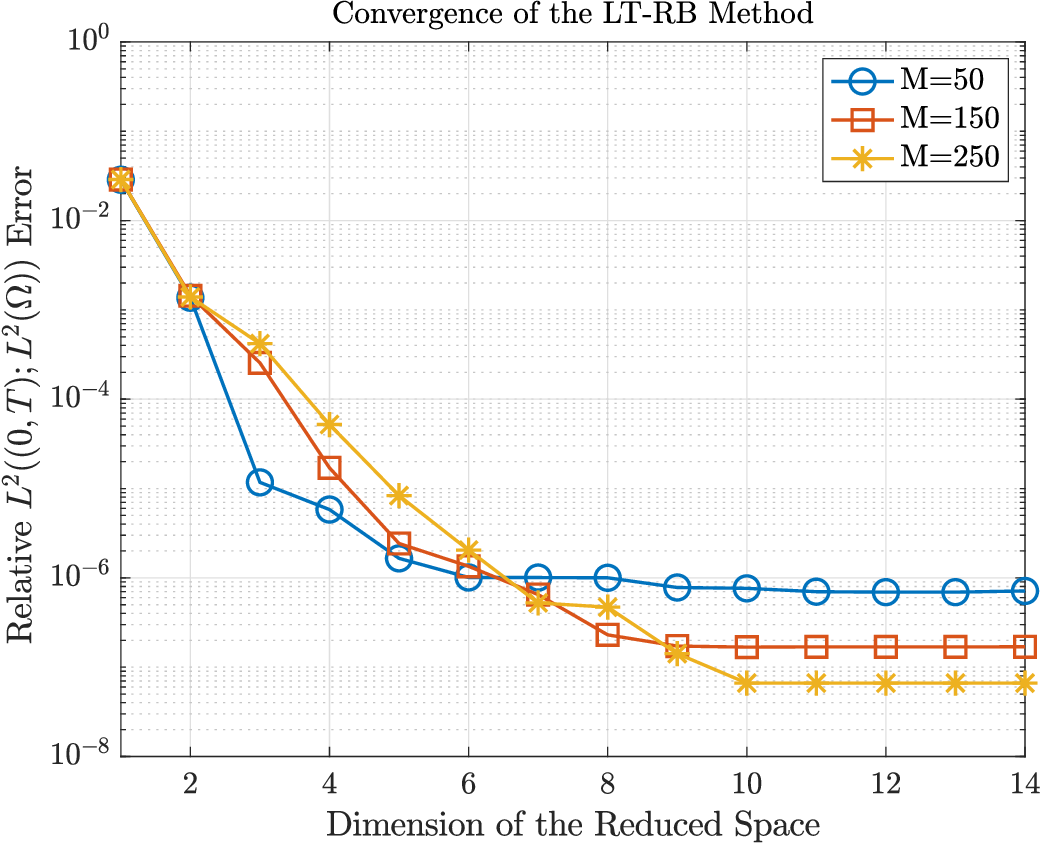}
		\subcaption{$\text{Rel\_Error}^{\normalfont\text{(rb)}}_R(\mathfrak{J};L^2(\Omega))$.}
		\label{fig:plot_relative_error_cube_L2_cube}
	\end{subfigure}
	\begin{subfigure}{0.49\linewidth}
		\includegraphics[width=\textwidth]
		{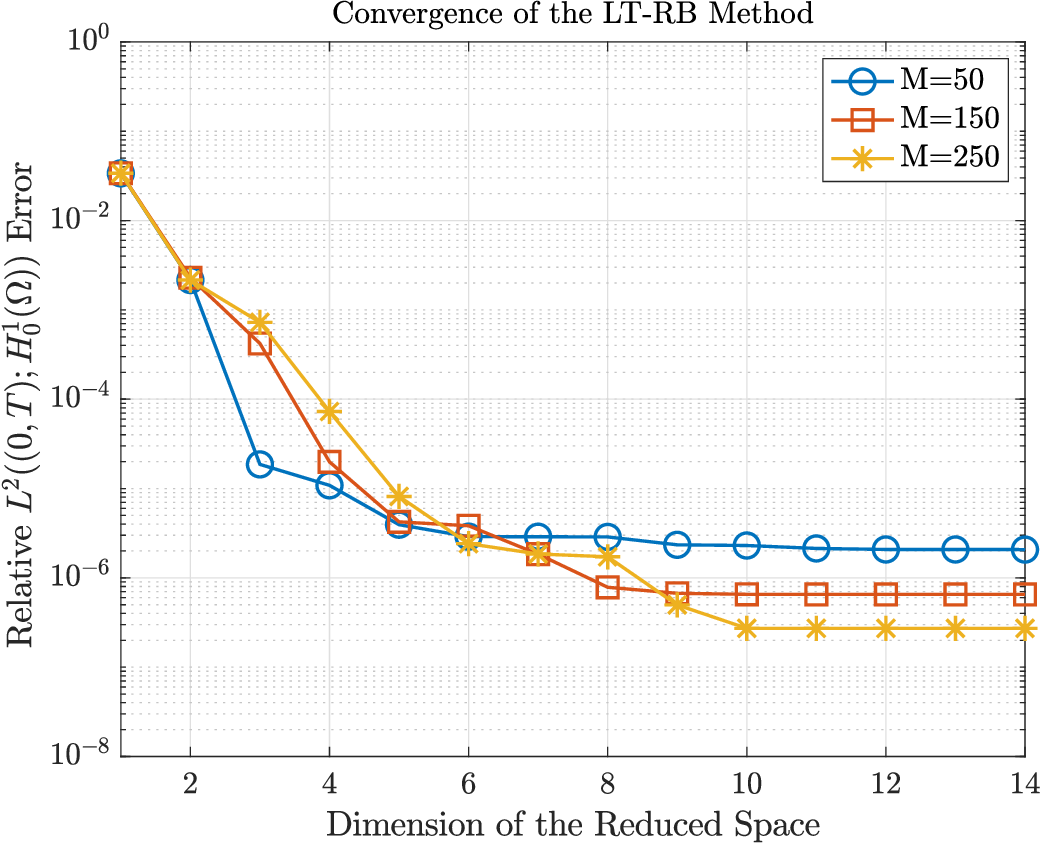}
		\subcaption{$\text{Rel\_Error}^{\normalfont\text{(rb)}}_R(\mathfrak{J};L^2(\Omega))$.}
		\label{fig:plot_relative_error_cube_H1_cube}
	\end{subfigure}
	\caption{\label{fig:error_rel_cube} 
	Relative error as defined in Section~\ref{sec:example_I}
	for the cube in three dimensions, i.e., $\Omega = (-\frac{1}{2},\frac{1}{2})^3 \subset \mathbb{R}^3$,
	and for the initial conditions $u^{(1)}_{0}$ defined in \eqref{eq:ic_1}.
	}
\end{figure}


\subsubsection{Speed-up of the \rev{LT-MOR} Method}
Figure~\ref{fig:speed_up_cube} presents the execution times for the \rev{LT-MOR} and 
for the computation of the high-fidelity solution with initial condition $u^{(1)}_{0}$. More precisely,
Figure~\ref{fig:plot_times_speed_omega_x_5_omega_5_nu_5_HF} displays the execution time
for the computation of the high-fidelity solution split into two main contributions: (1)
{\sf Assemble FEM}, which consists in the time required to set up the FE linear system of 
equations, and (2) {\sf Solve TD-HF}, which corresponds to the total time required to solve
the high-fidelity model using the backward Euler scheme.

Figure~\ref{fig:plot_times_speed_omega_x_5_omega_5_nu_5_M_50} through Figure~\ref{fig:plot_times_speed_omega_x_5_omega_5_nu_5_M_250}
show the execution \rev{times} of the \rev{LT-MOR} method for $M \in \{50,150,250\}$.
In each of this plots, the total time is broken down into the following contributions:
(1) Assembling the FE discretization ({\sf Assemble FEM}), (2) computing the snapshots 
or high-fidelity solutions in the Laplace domain ({\sf LD-HF}), (3) \rev{building} the reduced basis ({\sf Build RB}),
and (4) \rev{computing} the \rev{reduced} solution in the time domain ({\sf Solve TD-RB}).

\begin{figure}[!ht]
	\centering
	\begin{subfigure}{0.49\linewidth}
		\includegraphics[width=\textwidth]
		{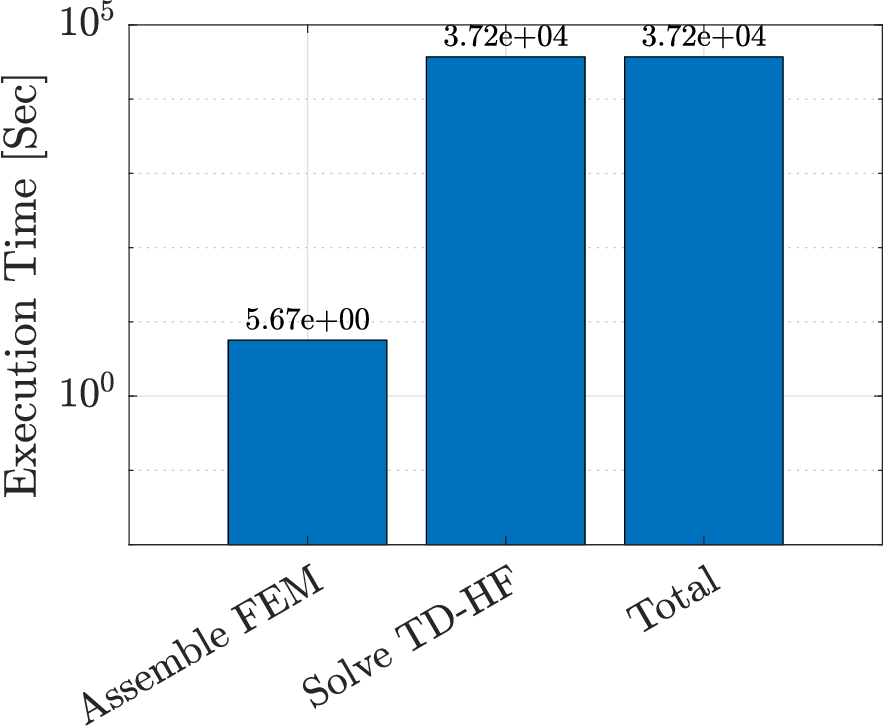}
		\subcaption{High-fidelity Solution.}
		\label{fig:plot_times_speed_omega_x_5_omega_5_nu_5_HF}
	\end{subfigure}
	\begin{subfigure}{0.49\linewidth}
		\includegraphics[width=\textwidth]
		{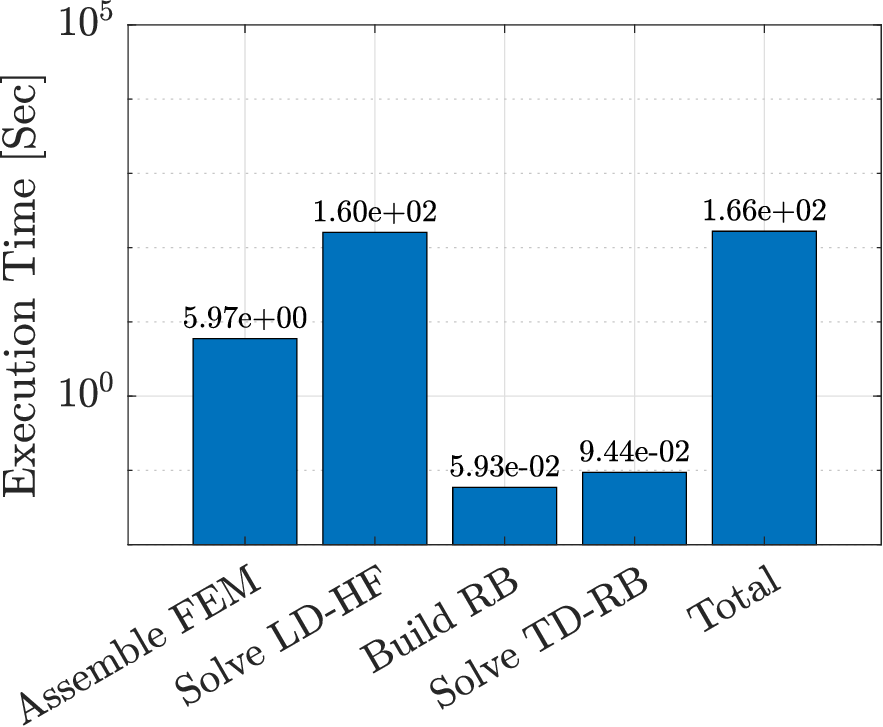}
		\subcaption{\rev{LT-MOR} with $M=50$ and $R=14$.}
		\label{fig:plot_times_speed_omega_x_5_omega_5_nu_5_M_50}
	\end{subfigure}
	\centering
	\begin{subfigure}{0.49\linewidth}
		\includegraphics[width=\textwidth]
		{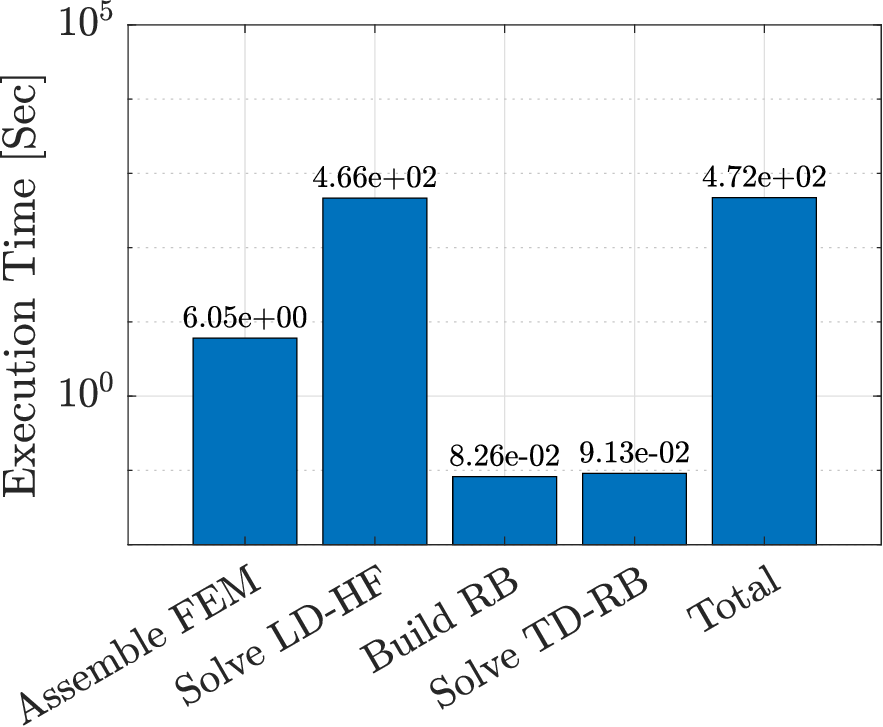}
		\subcaption{\rev{LT-MOR} with $M=150$ and $R=14$.}
		\label{fig:plot_times_speed_omega_x_5_omega_5_nu_5_M_150}
	\end{subfigure}
	\begin{subfigure}{0.49\linewidth}
		\includegraphics[width=\textwidth]
		{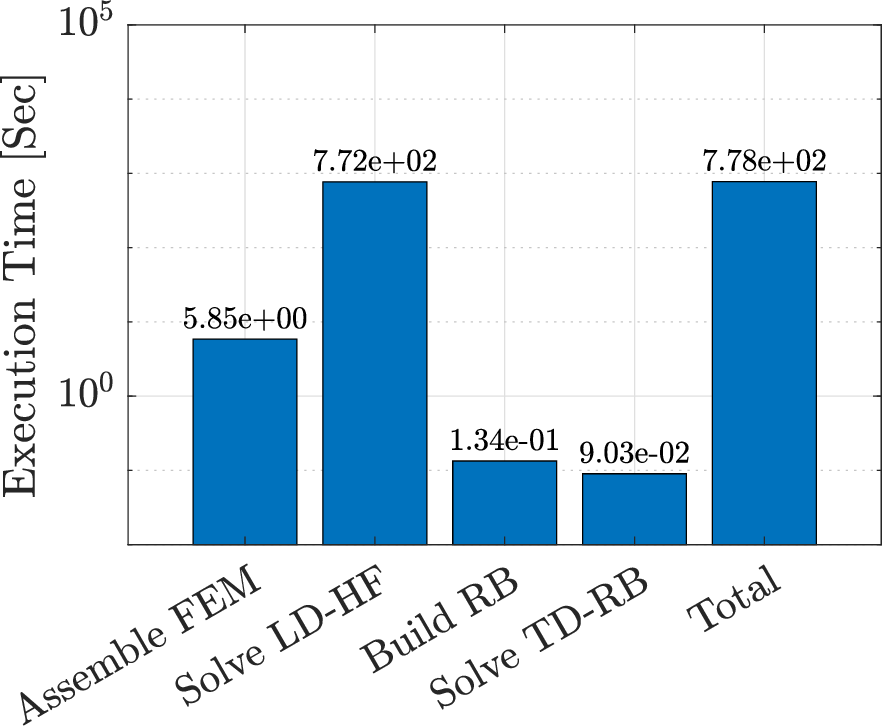}
		\subcaption{\rev{LT-MOR} with $M=250$ and $R=14$.}
		\label{fig:plot_times_speed_omega_x_5_omega_5_nu_5_M_250}
	\end{subfigure}
	\caption{\label{fig:speed_up_cube} 
	Execution times for the computation of the high-fidelity solution and the 
	reduced one for the parabolic problem in the unit cube as described in Section~\ref{sec:results_cube}.
	Figure~\ref{fig:plot_times_speed_omega_x_5_omega_5_nu_5_HF} presents the execution times
	for the computation of the high-fidelity solution split into the two main contributions: (1)
	Assembling the FE discretization ({\sf Assemble FEM}),
	and (2) solving the high-fidelity model ({\sf Solve TD-HF}).
	Figure~\ref{fig:plot_times_speed_omega_x_5_omega_5_nu_5_M_50} through Figure~\ref{fig:plot_times_speed_omega_x_5_omega_5_nu_5_M_250}
	show the execution times of the \rev{LT-MOR} method for $M \in \{50,150,250\}$.
	As in Figure~\ref{fig:speed_square}, in each of this plots,
	the total time is broken down into the four main contributions:
	(1) Assembling the FE discretization ({\sf Assemble FEM}), (2) computing the snapshots 
	or high-fidelity solutions in the Laplace domain ({\sf LD-HF}),
	(3) \rev{building} the reduced basis ({\sf Build RB}),
	and (4) \rev{computing} the reduce solution in the time domain ({\sf Solve TD-RB}). 
	}
\end{figure}


%
%
%

\subsection{Discussion}
In the view of the results presented in Section~\ref{sec:results_square} and Section~\ref{sec:results_cube},
we proceed to comment on the performance of the \rev{LT-MOR} method regarding the three aspects 
presented at the beginning of Section~\ref{sec:numerical_results}. 

\begin{itemize}
	\item
	{\bf Accuracy with respect to the high-fidelity solution.}
	As portrayed in Figure~\ref{fig:error_rel_U1} and Figure~\ref{fig:error_rel_U2} for the setting 
	described in Section~\ref{sec:results_square}, and in Figure~\ref{fig:error_rel_cube} for the setting
	of Section~\ref{sec:results_cube}, we observe that the relative accuracy of the \rev{LT-MOR} with 
	respect to the high-fidelity solver rapidly improves as we increase the dimension of the 
	reduced space. Indeed, with a reduced space of only dimension $R=6$ we obtain
	a relative error of $10^{-6}$. The explanation of this behaviour is the exponential
	convergence result stated in Theorem~\ref{eq:convergence_lt_rb_method}
	together with Theorem~\ref{eq:convergence_lt_rb_method_full}. 
	\item
	{\bf Accuracy with respect to the number of snapshots in the offline phase.}
	As observed in Figure~\ref{fig:error_rel_U1},  Figure~\ref{fig:error_rel_U2}, and Figure~\ref{fig:error_rel_cube},
	as we increase the number of snapshots in the offline phase, the relative error with respect to the
	high-fidelity solution remains very similar when the dimension of the reduced space ranges between
	$R=1$ and $R=6$. However, starting at $R = 6$, for each curve in  Figure~\ref{fig:error_rel_U1},  Figure~\ref{fig:error_rel_U2}, and Figure~\ref{fig:error_rel_cube},
	the relative error tapers off at different levels, and the exact value of these levels is determined by the total number 
	of snapshots in the offline phase.
	The higher the number of snapshots the lower the level is at which this \emph{plateau}
	is attained. 
	\item
	{\bf Speed-up.}
	The most important advantage of the \rev{LT-MOR} is its speed-up with respect to the high-fidelity solver
	for a fixed prescribed relative accuracy. As portrayed in Figure~\ref{fig:speed_square} and 
	Figure~\ref{fig:speed_up_cube}, the \rev{LT-MOR} is able to compute an approximation of the high-fidelity solution
	within a prescribed accuracy while incurring considerably less computational time.
	Figure~\ref{fig:plot_speed_U1_HF} and Figure~\ref{fig:plot_times_speed_omega_x_5_omega_5_nu_5_HF}
	portray the execution time required in the computation
	of the high-fidelity solutions for the square and the cube, respectively, with initial condition $u^{(1)}_0$ as in \eqref{eq:ic_1}.
	The bulk of the computational time is used to compute the high-fidelity solution (labeled {\sf Solve TD-HF} in
	Figure~\ref{fig:plot_speed_U1_HF} and Figure~\ref{fig:plot_times_speed_omega_x_5_omega_5_nu_5_HF}),
	while an inspection of execution times for the \rev{LT-MOR}, i.e.~Figure~\ref{fig:plot_speed_U1_RB_50} through Figure~\ref{fig:plot_speed_U1_RB_450}
	for the square and Figure~\ref{fig:plot_times_speed_omega_x_5_omega_5_nu_5_M_50} through
	Figure~\ref{fig:plot_times_speed_omega_x_5_omega_5_nu_5_M_250} for the cube, reveals that the bulk of the total
	execution time is used in the computation of the snapshots during the offline phase.
	The solution of the time evolution problem projected in the reduced space (contribution labeled as {\sf Solve TD-RB}) 
	necessitates a negligible amount of computational time for $R=14$ and $R=15$, which is the largest
	dimension of the reduced space used in the relative error computation.
	In the case of the two-dimensional computations in the square,
	we observe that for $M \in \{50,150,250,350,450\}$ the \rev{LT-MOR} method is $423.3$, $147.7$,
	$89.8$, $64.5$, and $49.8$ times faster than the full-order model, respectively, while maintaining a
	relative accuracy of $0.001\%$, at the very least.
	In the case of the cube for $M \in \{50,150,250\}$ the \rev{LT-MOR} method is $224.1$, $78.8$,
	$47.8$ times faster than the full-order model, respectively, while again maintaining a
	relative accuracy of $0.001\%$, at the very least. Observe that the comparison 
	has been made by fixing the time discretization.
	We remark that the computation of the snapshots in the offline phase
	has been executed serially. It is certainly possible to perform these computations in 
	a parallel fashion. This would increase the effective speed-up of the \rev{LT-MOR}
	method with respect to the full-order model.
\end{itemize}

\section{Concluding Remarks} \label{sec:Summary}
In this work, we have introduced a fast solver 
for \rev{the} numerical approximation of linear, second-order parabolic PDEs. 
The method introduced here relies on two existing mathematical
tools: The reduced basis method and the Laplace transform.
After applying the Laplace transform to the time-evolution problem, 
we obtain a parametric elliptic PDE, where the parameter corresponds to
the Laplace \rev{variable} itself. We construct a reduced basis for this particular 
problem by sampling on a carefully selected \rev{set of complex points in the Laplace
domain}. We argue that the basis constructed in the Laplace domain is, up to 
discretization in the Laplace \rev{variable}, optimal to solve the time-evolution. 
In an online phase, we project this problem onto the reduced basis and obtain
a time-evolution problem for the reduced coefficients. Numerical experiments validate our theoretical 
claims and pose the \rev{LT-MOR} method as a competing algorithm for the fast numerical approximation
of parabolic problems.

At the end of Section~\ref{sec:FRB} we posed important questions
arising from the description of the \rev{LT-MOR} method. 
Based on the work presented here, we would like to present the
answer, indicate where the technical details supporting the
corresponding answer may be found, and comment on issues that may be improved.

\begin{itemize}
	\item[{\sf \encircle{A1}}]
	\emph{Why is $\bm{\Phi}^{\normalfont\textrm{(rb)}}_R$ as in \eqref{eq:reduced_basis_Phi}
	a suitable reduced basis for Problem~\ref{pr:sdpr}?}
	The norm equivalence established in Theorem~(\ref{eq:laplace_transform_bijective})
	and the properties of Hardy spaces (Proposition~\ref{prop:properties_hardy}) imply that the construction of
	\rev{an} optimal reduced basis for the time-evolution problem can be \rev{performed} in the Laplace domain, up
	to a sampling in the Laplace domain, as discussed in Section~\ref{sec:laplace_hardy}.
	\item[{\sf \encircle{A2}}]
	\emph{How does the accuracy of the reduced solution improve as the
	dimension of the reduced space increases?}
	In Theorems \ref{eq:convergence_lt_rb_method} and \ref{eq:convergence_lt_rb_method_full},
	we prove that the error between the high-fidelity solution
	and the reduced solution decays exponentially fast.
	\item[{\sf \encircle{A3}}]
	\emph{How can one judiciously \emph{a priori} select the snapshots and the weights in \eqref{eq:tPOD_algebraic_frequency}?}
	In Section~\ref{eq:snap_shot_selection}, we have proposed a construction of
	discretization points and weights for the computations of the snapshots. The exact expression is given in \eqref{eq:snapshots}.
	However, we remark that, after wrapping the infinite integration contour in \eqref{eq:snap_shot_selection}
	around the unit circle, greedy strategies may be put in place as well.
	\item[{\sf \encircle{A4}}]
	\emph{How does the quality of the reduced basis improve as the number of snapshots increases?}
	As proved in Lemma~\ref{lmm:exponential_convergence_rule}, the error between the discrete \eqref{eq:discrete_error_measure}
	and its continuous counterpart \eqref{eq:exact_error} decays exponentially fast.
	In practice, this indicates that only a handful of discretization points and weights
	are needed to appropriately sample the Laplace domain,
	as observed in the numerical experiments presented in Section~\ref{sec:numerical_results}.
	\item[{\sf \encircle{A5}}]
	\emph{Why is only the real part of the snapshots required for the construction of $\bm{\Phi}^{\normalfont\textrm{(rb)}}_R$?}
	As stated in Section~\ref{sec:real_complex_basis}, the reduced basis obtained by using only the real part of 
	the snapshots is exactly the same as the one obtained by doing a complex SVD, and in both cases
	the basis is real-valued. This statement is rigorously proved in Lemma~\ref{lmm:real_basis}.
\end{itemize} 

So far we have restricted our work to the fast numerical approximation of linear,
second-order parabolic problems. However, many extensions to this approach are possible. 
First and foremost, one could use the \rev{LT-MOR} method to tackle the parametric linear, second-order 
parabolic problems. Indeed, one could construct a reduced basis by sampling in both the Laplace domain
and in the parameter space, as opposed to usual approaches that solve the high-fidelity problem in time
for each sample in the parameter space. This is the subject of ongoing research.
In addition, current work encompasses the extension of the \rev{LT-MOR}
method to the linear, second-order wave equation.

\appendix

\section{Proof of Theorem~\ref{eq:well_posedness}}
\label{sec:existence}
Let $\widehat{w}(s)$ be as in \eqref{eq:decomposition_h_u_s}
and set 
\begin{equation}
	\widehat{p}(s) \coloneqq  s \widehat{w}(s) = s \widehat{u}(s) -u_0 \in H^1_0(\Omega).
\end{equation}
According to Lemma~\ref{lmm:error_bound_L2} and \eqref{eq:equation_w_s} for $s\in \Pi_{\alpha}$
\begin{equation}
	\norm{\widehat{p}(s)}_{H^{-1}(\Omega)}
	\leq
	\left(
	1
	+
	\frac{
		\overline{c}_{\bm{A}}
	}{
		\underline{c}_{\bm{A}}
	}
	\right)
	\left(
		C_P(\Omega)
		\norm{
			\widehat{f}(s)
		}_{L^2(\Omega)}
		+
		\frac{\overline{c}_{\bm{A}}}{\snorm{s}}
		\norm{
			u_0
		}_{H^1_0(\Omega)}
	\right).
\end{equation}
Recalling the definition
of the $\mathscr{H}^2(\Pi_\alpha;H^{-1}(\Omega))$--norm
\begin{equation}\label{eq:calculation_hardy_2}
\begin{aligned}
	\norm{
		\widehat{p}
	}^2_{\mathscr{H}^2(\Pi_\alpha;H^{-1}(\Omega))}
	=
	&
	\norm{
		s\widehat{w}
	}^2_{\mathscr{H}^2(\Pi_\alpha;H^{-1}(\Omega))}
	\\
	=
	&
	\int\limits_{-\infty}^{+\infty}
	\snorm{s}^2
	\norm{
		\widehat{w}
		(\alpha + \imath \tau)
	}^2_{H^{-1}(\Omega)}
	\,
	\frac{\text{d}\tau}{2\pi}
	\\
	\leq
	&
	2
	\left(
	1
	+
	\frac{
		\overline{c}_{\bm{A}}
	}{
		\underline{c}_{\bm{A}}
	}
	\right)^2
	\left(
	C^2_P(\Omega)
	\int\limits_{-\infty}^{+\infty}
	\norm{
		\widehat{f}(\alpha + \imath \tau)
	}^2_{L^2(\Omega)}
	\,
	\frac{\text{d}\tau}{2\pi}
	\right.
	\\
	&
	\left.
	+
	\frac{\overline{c}^2_{\bm{A}}}{2\pi}
	\norm{
		u_0
	}^2_{H^1_0(\Omega)}
	\int\limits_{-\infty}^{+\infty}
	\frac{	\text{d}\tau}{\snorm{\alpha + \imath \tau}^2}.
	\right)
\end{aligned}
\end{equation}
Therefore,
\begin{equation}\label{eq:a_priori_2}
\begin{aligned}
	\norm{
		\widehat{p}
	}_{\mathscr{H}^2(\Pi_\alpha;H^{-1}(\Omega))}
	\leq
	&
	\sqrt{2}
	\left(
	1
	+
	\frac{
		\overline{c}_{\bm{A}}
	}{
		\underline{c}_{\bm{A}}
	}
	\right)
	\left(
		C_P(\Omega)
		\norm{
			\widehat{f}
		}_{\mathscr{H}^2(\Pi_\alpha;L^2(\Omega))}
	\right.
	\\
	&
	+
	\left.
		\overline{c}_{\bm{A}}
		\sqrt{\frac{1}{{2\alpha}}}
		\norm{
			u_0
		}_{H^1_0(\Omega)}
	\right).
\end{aligned}
\end{equation}
Then, according to Theorem~\ref{eq:laplace_transform_bijective}
\begin{equation}\label{eq:inv_Laplace_u}
	u(t) = \mathcal{L}^{-1}\left\{\widehat{u}\right\}(t)
	\in 
	L^2_\alpha\left(\mathbb{R}_+;H^1_0(\Omega)\right),
\end{equation}
and
\begin{equation}
	p(t) 
	= 
	\mathcal{L}^{-1}\left\{\widehat{p}\right\}(t)
	\in 
	L^2_\alpha\left(\mathbb{R}_+;H^{-1}(\Omega)\right).
\end{equation}
Recalling Proposition~\ref{eq:operational_properties}, item (ii), 
for a.e. $t\geq 0$
\begin{equation}
	\dotp{
		u(t)
	}{
		v
	}_{L^2(\Omega)}
	-
	\dotp{
		u(0)
	}{
		v
	}_{L^2(\Omega)}
	=
	\int\limits_{0}^{t}
	\dual{
		p(\tau)
	}{
		v
	}_{H^{-1}(\Omega) \times H^{1}_0(\Omega) }	
	\;
	\text{d}
	\tau,
	\quad
	\forall
	v \in
	H^{1}_0(\Omega)
\end{equation}
therefore $p(t) = \partial_t u(t) \in H^{-1}(\Omega)$ for a.e. $t>0$, 
and $u(0) = u_0$ in $L^2(\Omega)$.

It follows from Problem~\ref{prb:frequency_problem} that
\begin{equation}
	\dual{
		\widehat{p}(s)
	}{
		v
	}_{{H^{-1}(\Omega) \times H^{1}_0(\Omega) }}
	+
	\mathsf{a}
	\left(
		\widehat{u}(s)
		,
		v
	\right)
	=
	\dotp{
		\widehat{f}(s)
	}{v}_{L^2(\Omega)},
	\quad
	\forall
	v \in H^1_0(\Omega).
\end{equation}
Recalling item (i) in Proposition~\ref{eq:operational_properties}, we get
for a.e. $t>0$
\begin{equation}
	\dual{
		 \partial_t u(t)
	}{
		v
	}_{{H^{-1}(\Omega) \times H^{1}_0(\Omega) }}
	+
	\mathsf{a}
	\left(
		u(t)
		,
		v
	\right)
	=
	\dotp{
		f(t)
	}{v}_{L^2(\Omega)},
	\;\;
	\forall
	v \in H^1_0(\Omega),
\end{equation}
thus $u(t) = \mathcal{L}^{-1}\left\{\widehat{u}\right\}(t)$ as in 
\ref{eq:inv_Laplace_u} is solution to Problem~\ref{pbm:wave_equation},
and uniqueness follows from Theorem~\ref{eq:laplace_transform_bijective}
together with the uniqueness of $\widehat{u} \in \mathscr{H}^2_\alpha(H^1_0(\Omega))$.

The {a priori} estimate stated in \eqref{eq:a_priori_estimate} follows from
\eqref{eq:a_priori_1} and \eqref{eq:a_priori_2} combined with, again, Theorem~\ref{eq:laplace_transform_bijective}.

\section*{Acknowledgment}
\rev{The authors thank Dr. Ricardo Reyes for his input during the initial phase of this work.}


\bibliographystyle{ws-m3as}
\bibliography{references}

\end{document}